\documentclass{compositio}
\usepackage{stmaryrd}
\usepackage{amsmath, amscd, amssymb, mathrsfs, accents, amsfonts}
\usepackage{url}
\usepackage[all]{xy}
\usepackage{longtable}
\usepackage{dsfont}
\usepackage{tikz}
\usetikzlibrary{decorations.pathmorphing}
\usetikzlibrary{calc}
\usetikzlibrary{decorations.markings}
\def\nicedashedcolourscheme{\shadedraw[top color=blue!22, bottom color=blue!22, draw=gray, dashed]}
\def\nicecolourscheme{\shadedraw[top color=blue!22, bottom color=blue!22, draw=white]}
\def\nicepalecolourscheme{\shadedraw[top color=blue!12, bottom color=blue!12, draw=white]}

\def\nicereallynocolourscheme{\shadedraw[top color=white!2, bottom color=white!25, draw=white]}
\definecolor{Myblue}{rgb}{0,0,0.6}
\usepackage[a4paper,colorlinks,citecolor=Myblue,linkcolor=Myblue,urlcolor=Myblue,pdfpagemode=None]{hyperref}

\SelectTips{cm}{}

\newtheorem{theorem}{Theorem}[section]
\newtheorem{proposition}[theorem]{Proposition}
\newtheorem{lemma}[theorem]{Lemma}
\newtheorem{corollary}[theorem]{Corollary}

\theoremstyle{definition}
\newtheorem{definition}[theorem]{Definition}
\newtheorem{example}[theorem]{Example}
\newtheorem{remark}[theorem]{Remark}

\numberwithin{equation}{section}
\numberwithin{figure}{section}

\def\globalsigneval{n}
\def\eval{\operatorname{ev}}
\def\coev{\operatorname{coev}}
\def\res{\operatorname{Res}}

\def\can{\operatorname{can}}

\def\Hom{\operatorname{Hom}}
\def\uHom{\underline{\Hom}}
\def\End{\operatorname{End}}
\def\uEnd{\underline{\End}}

\DeclareMathOperator{\str}{str}
\DeclareMathOperator{\hmf}{hmf}
\DeclareMathOperator{\HMF}{HMF}
\DeclareMathOperator{\HF}{HF}

\DeclareMathOperator{\lAt}{lAt}
\DeclareMathOperator{\At}{At}

\begin{document}

\def\Atlarrow{\overset{\leftarrow}{\At}}
\def\Res{\res\!}
\newcommand{\cat}[1]{\mathcal{#1}}
\newcommand{\lto}{\longrightarrow}
\newcommand{\xlto}[1]{\stackrel{#1}\lto}
\newcommand{\mf}[1]{\mathfrak{#1}}
\newcommand{\md}[1]{\mathscr{#1}}
\newcommand{\intvar}{\bs{x}_{\textup{int}}}
\newcommand{\extvar}{\bs{x}_{\textup{ext}}}
\newcommand{\qderu}[2]{\mathbf{D}^{#1}(#2)}
\newcommand{\ud}{\mathrm{d}}
\def\l{\,|\,}
\def\by{f}
\def\totimes{\otimes}
\def\di{Q}
\newcommand{\cotimes}[1]{\,\widehat{\otimes}_{#1}\,}
\def\QQ{\mathds{Q}}
\def\krc{C}
\def\diffm{d}
\def\diffh{d_{\chi}}
\def\redh{\overline{H}}
\def\ZZ{\mathds{Z}}
\def\bs{\boldsymbol}
\def\Ztwo{\mathds{Z}_2}
\def\mdual{^{\vee}}
\def\KR{\operatorname{KR}}
\def\I{\!\operatorname{i}\!}
\def\E{\operatorname{e}\!}
\def\sln{\mathfrak{sl}(N)}
\def\nN{\mathds{N}}
\def\nZ{\mathds{Z}}
\def\nQ{\mathds{Q}}
\def\nR{\mathds{R}}
\def\nC{\mathds{C}}
\def\Bar{\mathds{B}}
\def\cBar{\widehat{\mathds{B}}}
\def\Se{S^{\operatorname{e}}}
\def\Re{R^{\operatorname{e}}}
\def\Ae{A^{\operatorname{e}}}
\def\Be{B^{\operatorname{e}}}
\def\Aop{A^{\operatorname{op}}}
\def\Rop{R^{\operatorname{op}}}
\def\lra{\longrightarrow}
\def\lmt{\longmapsto}
\def\LG{\mathcal{LG}_k}
\def\dual{\dagger}
\def\dlangle{\big\langle\!\big\langle}
\def\drangle{\big\rangle\!\big\rangle}
\def\bigdlangle{\Big\langle\!\!\Big\langle}
\def\bigdrangle{\Big\rangle\!\!\Big\rangle}
\def\reprod{\gamma}
\newcommand{\Ress}[1]{\res_{#1}\!}
\newcommand{\be}{\begin{equation}}
\newcommand{\ee}{\end{equation}}
\def\Xcirc{%
\begin{tikzpicture}[inner sep=0mm]
\node (X) at (0,0) {$X$};
\node (0) at (0,0) [circle,inner sep=0.99pt, thin,draw=black,fill= white] {};
\end{tikzpicture}%
}
\def\Xbul{%
\begin{tikzpicture}[inner sep=0mm]
\node (X) at (0,0) {$X$};
\node (0) at (0,0) [circle,inner sep=0.99pt, thin,draw=black,fill= black] {};
\end{tikzpicture}%
}

\renewcommand{\labelenumi}{(\roman{enumi})}

\allowdisplaybreaks

\usetikzlibrary{arrows,calc,decorations.pathreplacing,decorations.markings,shapes.geometric,shadows}
\tikzset{
    string/.style={draw=#1, postaction={decorate}, decoration={markings,mark=at position .51 with {\arrow[draw=#1]{>}}}},
    costring/.style={draw=#1, postaction={decorate}, decoration={markings,mark=at position .51 with {\arrow[draw=#1]{<}}}},
    ostring/.style={draw=#1, postaction={decorate}, decoration={markings,mark=at position .47 with {\arrow[draw=#1]{>}}}},
    ustring/.style={draw=#1, postaction={decorate}, decoration={markings,mark=at position .56 with {\arrow[draw=#1]{>}}}},
    oostring/.style={draw=#1, postaction={decorate}, decoration={markings,mark=at position .43 with {\arrow[draw=#1]{>}}}},
    uustring/.style={draw=#1, postaction={decorate}, decoration={markings,mark=at position .59 with {\arrow[draw=#1]{>}}}},
    directed/.style={string=blue!50!black}, 
    odirected/.style={ostring=blue!50!black}, 
    udirected/.style={ustring=blue!50!black}, 
    oodirected/.style={oostring=blue!50!black}, 
    uudirected/.style={uustring=blue!50!black},     
    redirected/.style={costring= blue!50!black},
}

\usetikzlibrary{fadings,decorations.pathreplacing}

\newcommand\pgfmathsinandcos[3]{%
  \pgfmathsetmacro#1{sin(#3)}%
  \pgfmathsetmacro#2{cos(#3)}%
}
\newcommand\LongitudePlane[3][current plane]{%
  \pgfmathsinandcos\sinEl\cosEl{#2} 
  \pgfmathsinandcos\sint\cost{#3} 
  \tikzset{#1/.estyle={cm={\cost,\sint*\sinEl,0,\cosEl,(0,0)}}}
}
\newcommand\LatitudePlane[3][current plane]{%
  \pgfmathsinandcos\sinEl\cosEl{#2} 
  \pgfmathsinandcos\sint\cost{#3} 
  \pgfmathsetmacro\yshift{\cosEl*\sint}
  \tikzset{#1/.estyle={cm={\cost,0,0,\cost*\sinEl,(0,\yshift)}}} %
}
\newcommand\DrawLongitudeCircle[2][1]{
  \LongitudePlane{\angEl}{#2}
  \tikzset{current plane/.prefix style={scale=#1}}
  \pgfmathsetmacro\angVis{atan(sin(#2)*cos(\angEl)/sin(\angEl))} %
  \draw[redirected,current plane,color=blue!50!black, very thick] (\angVis:1) arc (\angVis:\angVis+180:1);
  \draw[current plane,dotted,color=blue!50!gray, very thick] (\angVis-180:1) arc (\angVis-180:\angVis:1);
}
\newcommand\DrawLatitudeCircle[2][1]{
  \LatitudePlane{\angEl}{#2}
  \tikzset{current plane/.prefix style={scale=#1}}
  \pgfmathsetmacro\sinVis{sin(#2)/cos(#2)*sin(\angEl)/cos(\angEl)}
  \pgfmathsetmacro\angVis{asin(min(1,max(\sinVis,-1)))}
  \draw[directed,current plane, color=blue!50!black] (\angVis:1) arc (\angVis:-\angVis-180:1);
  \draw[current plane,dashed, color=blue!50!gray] (180-\angVis:1) arc (180-\angVis:\angVis:1);
}
\newcommand\DrawLatitudeCircleU[2][1]{
  \LatitudePlane{\angEl}{#2}
  \tikzset{current plane/.prefix style={scale=#1}}
  \pgfmathsetmacro\sinVis{sin(#2)/cos(#2)*sin(\angEl)/cos(\angEl)}
  \pgfmathsetmacro\angVis{asin(min(1,max(\sinVis,-1)))}
  \draw[redirected,current plane, color=blue!50!black] (\angVis:1) arc (\angVis:-\angVis-180:1);
  \draw[current plane,dashed, color=blue!50!gray] (180-\angVis:1) arc (180-\angVis:\angVis:1);
}

\title{Adjunctions and defects in Landau-Ginzburg models}

\dedication{To Ragnar-Olaf Buchweitz on the occasion of his sixtieth birthday}

\author{Nils Carqueville}
\email{nils.carqueville@physik.uni-muenchen.de}
\address{Arnold Sommerfeld Center for Theoretical Physics, LMU M\"unchen \& Excellence Cluster Universe}

\author{Daniel Murfet}
\email{daniel.murfet@math.ucla.edu}
\address{Department of Mathematics, UCLA}


\begin{abstract}
We study the bicategory of Landau-Ginzburg models, which has polynomials as objects and matrix factorisations as $1$-morphisms. Our main result is the existence of adjoints in this bicategory and formulas for the evaluation and coevaluation maps in terms of Atiyah classes and homological perturbation. The bicategorical perspective offers a unified approach to Landau-Ginzburg models: we show how to compute arbitrary correlators and recover the full structure of open/closed TFT, including the Kapustin-Li disc correlator and a simple proof of the Cardy condition, in terms of defect operators which in turn are directly computable from the adjunctions. 
\end{abstract}

\maketitle

\section{Introduction}\label{sec:Introduction}

Landau-Ginzburg models play an important role in many areas of mathematical physics and pure mathematics including singularity theory, representation theory, (homological) mirror symmetry, knot invariants, and conformal or topological field theory. The interplay between these areas is one of the aesthetic motivations for studying Landau-Ginzburg models. Another general motivation is their dual nature of affording insight into deep structure while being concrete enough to allow for hands-on computations. 

In this paper we will show how this dichotomy manifests itself in the context of two-dimensional topological field theory (TFT) with defects. We explain how Landau-Ginzburg models give rise to a bicategory with adjoints (also called duals) and we describe the structure maps in this bicategory, which include the units and counits of adjunction (also called evaluation and coevaluation maps) in terms of basic invariants called Atiyah classes \cite{atiyahconn}. On the one hand this gives a satisfying explanation for duality in the setting of Landau-Ginzburg models in terms of commutation relations for Atiyah classes, and on the other hand it provides an effective way of evaluating arbitrary string diagrams in the bicategory. Since string diagrams in a bicategory can be identified with correlators in TFTs with defects, this opens the door for many applications.

In order to set the stage, and in particular explain how string diagrams are related to correlators, we recall a few aspects of TFTs with defects in an informal fashion; for more detailed accounts see~\cite{k1004.2307} and~\cite[Section~2]{dkr1107.0495}. We imagine \textsl{bulk sector theories}~$T_I$ to ``live'' on a two-dimensional surface called the \textsl{worldsheet}. More precisely, the worldsheet may be partitioned into various domains to which the (not necessarily distinct) theories~$T_I$ are associated, and which are separated by one-dimensional oriented \textsl{defect lines} $D_\alpha$. A sketch of a typical worldsheet is shown in Figure~\ref{worldsheetwithdefects}. 

In addition to the labels $T_I$ for the two-dimensional domains and $D_\alpha$ for the one-dimensional defect lines, we also include labels~$\phi_i$ for zero-dimensional points. These labels are interpreted as describing \textsl{fields} inserted at the points on the worldsheet. Note that the fields can also be placed at junctions of multiple defect lines. 

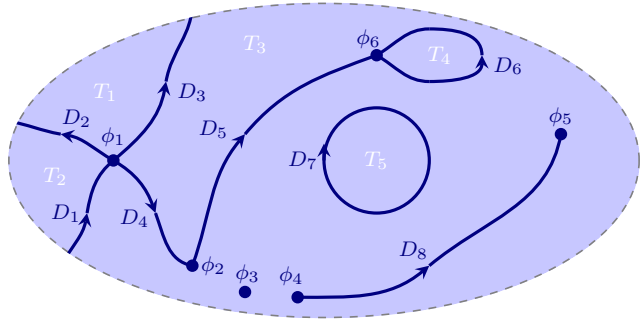
\begin{figure}[t]
$$
\begin{tikzpicture}[very thick,scale=0.7,color=blue!50!black, baseline,>=stealth]
\clip (0,0) ellipse (6cm and 3cm);
\nicedashedcolourscheme (0,0) ellipse (6cm and 3cm);

\draw (-4.15,1.3) [white] node {{\scriptsize $T_1$}};
\draw (-5.1,-0.3) [white] node {{\scriptsize $T_2$}};
\draw (-1.3,2.2) [white] node {{\scriptsize $T_3$}};
\draw (2.2,2.0) [white] node {{\scriptsize $T_4$}};
\draw (1,0) [white] node {{\scriptsize $T_5$}};

\draw (-4.9,-1) node {{\scriptsize $D_1$}};
\draw (-4.7,0.8) node {{\scriptsize $D_2$}};
\draw (-2.5,1.3) node {{\scriptsize $D_3$}};
\draw (-3.6,-1.1) node {{\scriptsize $D_4$}};
\draw (-2.1,0.6) node {{\scriptsize $D_5$}};
\draw (3.5,1.8) node {{\scriptsize $D_6$}};
\draw (-0.4,0) node {{\scriptsize $D_7$}};
\draw (1.7,-1.7) node {{\scriptsize $D_8$}};

\draw (-4,0.45) node {{\scriptsize $\phi_1$}};
\draw (-2.1,-2) node {{\scriptsize $\phi_2$}};
\draw (-1.45,-2.15) node {{\scriptsize $\phi_3$}};
\draw (-0.62,-2.2) node {{\scriptsize $\phi_4$}};
\draw (4.5,0.84) node {{\scriptsize $\phi_5$}};
\draw (0.85,2.32) node {{\scriptsize $\phi_6$}};

\draw[->, very thick, out=60, in=260] (-5,-2) to (-4.5,-1);
\draw[very thick, out=80, in=220] (-4.5,-1) to (-4,0);
\filldraw (-4,0) circle (2.5pt);
\draw[->,very thick, out=150, in=350] (-4,0) to (-5,0.5);
\draw[very thick, out=170, in=350] (-5,0.5) to (-6,0.75);

\draw[->,very thick, out=-30, in=110] (-4,0) to (-3.2,-1);
\draw[very thick, out=-70, in=170] (-3.2,-1) to (-2.5,-2);
\filldraw (-2.5,-2) circle (2.5pt);

\draw[->,very thick, out=70, in=230] (-2.5,-2) to (-1.5,0.5);
\draw[very thick, out=50, in=200] (-1.5,0.5) to (1,2);
\filldraw (1,2) circle (2.5pt);

\draw[very thick, out=-30, in=180] (1,2) to (2,1.5);
\draw[->,very thick, out=0, in=270] (2,1.5) to (3,2);
\draw[very thick, out=90, in=0] (3,2) to (2,2.5);
\draw[very thick, out=180, in=30] (2,2.5) to (1,2);

\draw[very thick] (1,0) circle (1);
\draw[->,very thick, out=90, in=270] (0,0) to (0,0);

\filldraw (-1.5,-2.5) circle (2.5pt);

\filldraw (-0.5,-2.6) circle (2.5pt);
\draw[->,very thick, out=0, in=220] (-0.5,-2.6) to (2,-2);
\draw[very thick, out=40, in=260] (2,-2) to (4.5,0.5);
\filldraw (4.5,0.5) circle (2.5pt);

\draw[->,very thick, out=40, in=260] (-4,0) to (-3,1.5);
\draw[very thick, out=80, in=270] (-3,1.5) to (-2.5,3);
\end{tikzpicture}
$$
\caption{Part of a worldsheet with defect lines and field insertions} 
\label{worldsheetwithdefects} 
\end{figure}

A TFT is a functor that assigns a number called the \textsl{correlator} to a labelled worldsheet like the one in Figure~\ref{worldsheetwithdefects}. Its topological nature implies that the value of any correlator does not depend on the precise position of the $D_\alpha$'s and $\phi_i$'s, only on their isotopy class. 

It is natural to organise the data in such a functor using a bicategory in which the objects are the theories $T_I$, 1-morphisms are labels for defect lines $D_\alpha$, and 2-morphisms are the fields $\phi_i$. The composition of 2-morphisms is the operator product of the fields, which is strictly associative because we only consider \textsl{topological} field theories. The composition of 1-morphisms comes about as follows: since the exact locus of the defect lines does not matter, two (or more) adjacent defect lines can be brought together arbitrarily close, and the limit of this \textsl{fusion} is well-defined and nonsingular. The unit of the fusion product is called the \textsl{invisible defect}. Thus any TFT with defects is expected to be associated with a bicategory \cite[Section~2.2]{dkr1107.0495}. A rich and interesting example is the bicategory of algebraic varieties with Fourier-Mukai kernels as 1-morphisms~\cite{ct1007.2679} which describes B-twisted sigma models, see~e.\,g.~\cite{MB2}.

The subject of this paper is the bicategory $\LG$ of Landau-Ginzburg models over a base ring~$k$. An object (or bulk theory) is a polynomial ring $R=k[x_1,\ldots,x_n]$ together with an element $W \in R$ (the \textsl{potential}) satisfying a finiteness condition which in the case $k = \nC$ holds for example when the critical points of $W$ are isolated. The 1- and 2-morphisms (or defects and fields) are described by the triangulated categories of matrix factorisations of potential differences $V-W$, with matrix factorisations of $V - W$ defining $1$-morphisms from $W$ to $V$. The fusion of 1-morphisms is the tensor product of matrix factorisations~\cite{yoshino98,br0707.0922}, and the invisible defect (or unit) between a potential $W \in k[x_1,\ldots,x_n]$ and itself is the stabilised diagonal $\Delta_W = \bigwedge (\bigoplus_{i=1}^n (R\otimes_k R)\cdot \theta_i)$. That $\LG$ is a bicategory was worked out in~\cite{McNameethesis, Calinetal, cr0909.4381}.

On general grounds it is expected that the bicategorical description of TFTs with defects involves additional structure. For example, defect lines are \textsl{oriented} so we expect that any defect line should be adjoint to the ``same'' defect line with reversed orientation. In the case of Landau-Ginzburg models this means that given a $1$-morphism $(X,D)$ from the object $(k[x_1,\ldots,x_n], W)$ to the object $(k[z_1,\ldots,z_m], V)$, that is, a $\nZ _2$-graded finite-rank free $k[x,z]$-module $X$ with odd operator $D$ satisfying $D^2 = (V - W) \cdot 1_X$, we expect there to be a matrix factorisation~$X^\dual$ of $W-V$ together with evaluation and coevaluation maps
\be\label{evcoevintro}
\widetilde\eval_X: X \otimes_{k[x]} X^\dual \lra \Delta _V
\, , \qquad
\widetilde\coev_X: \Delta_W \lra X^\dual \otimes_{k[z]} X
\ee
defining an adjunction between~$X$ and~$X^\dual$. Similarly, one expects a matrix factorisation~${}^\dual X$ and maps $\eval_X, \coev_X$ making ${}^\dual X$ a left adjoint of $X$. The special case where~$V = W$ depends on only one variable was worked out in~\cite{cr1006.5609}. 

In the present paper we prove that every $1$-morphism $X:W\lra V$ in $\LG$ as above has left and right adjoints given by ${}^\dual X = X^{\vee}[m]$ and $X^\dual = X^{\vee}[n]$ where $X^{\vee} = \Hom_{k[x,z]}(X, k[x,z])$ is the dual factorisation. The fact that left and right adjoints differ by a shift is encoded in $\LG$ being graded pivotal (a notion that we explain in Section~\ref{sec:wiggliesandsigns}). There is a natural pseudofunctor from $\LG$ to the bicategory of categories and functors which sends $W$ to the homotopy category of matrix factorisations of $W$ and a $1$-morphism $W \lto V$ to an integral functor between categories of matrix factorisations; the adjunctions just described on the level of $1$-morphisms translate to honest adjunctions.


The constructions can be made very concrete. For example, if $\{ e_i \}_i$ denotes a basis for the free $k[x,z]$-module~$X$ with dual basis $\{ e_i^* \}_i$ then the coevaluation map is given by
\be\label{eq:formula_intro_coev}
\widetilde\coev_X( \gamma ) = \sum_{i,j} (-1)^{(l+1)|e_j| + s} \big\{ \partial^{x,x'}_{[b_l]} D \ldots \partial^{x,x'}_{[b_1]} D \big\}_{ji} \cdot e_i^* \otimes e_j
\ee
where the sequence $b_1 < \cdots < b_l$ and integer $s$ are determined by $\gamma \wedge \theta_{b_1} \ldots \theta_{b_l} = (-1)^s \theta_1 \ldots \theta_n$, and the $\partial^{x,x'}_{[i]}$ are divided difference operators, see \eqref{diffquotop}. The evaluation maps have similar elementary presentations, which involve in addition supertraces and residues, see Section~\ref{sec:derivcoeval}.

We briefly mention two applications involving string diagrams. A bulk field in a Landau-Ginzburg model with potential $W\in k[x]$ is described by an endomorphism of the stabilised diagonal $\Delta_W$, so the field is an element in the Milnor algebra $M_W = k[x]/(\partial_{x_i} W)$. Given a matrix factorisation $(X,D)$ of $V-W$, i.\,e.~a defect between the theories~$W$ and $V\in k[z_1,\ldots,z_m]$, we obtain an operator $\mathcal D_r(X)$ between the spaces of bulk fields by sending $\psi\in M_W$ to an element in $M_V$ obtained by ``wrapping the defect line labelled by~$X$ around~$\psi$''. We can make rigorous sense of this \textsl{defect action on bulk fields} in terms of string diagrams in the bicategory $\LG$ as 
\begin{align}\label{defectactionIntro}
\mathcal D_r(X)(\psi) = 
\begin{tikzpicture}[very thick,scale=0.5,color=blue!50!black, baseline,>=stealth]
\nicepalecolourscheme (0,0) circle (3.5);
\fill (2.2,-2.2) circle (0pt) node[white] {{\small$V$}};
\nicecolourscheme (0,0) circle (2);
\fill (1.1,-1.1) circle (0pt) node[white] {{\small$W$}};
\draw (0,0) circle (2);
\fill (0:2) circle (0pt) node[right] {{\small$X$}};
\draw[<-, very thick] (0.100,2) -- (-0.101,2) node[above] {}; 
\draw[<-, very thick] (-0.100,-2) -- (0.101,-2) node[below] {}; 
\fill (135:0) circle (3.3pt) node[right] {{\small$\psi$}};
\fill (130:2) circle (3.3pt) node[left] {{\small$\rho_X$}};
\fill (230:2) circle (3.3pt) node[left] {{\small$\rho^{-1}_{X}$}};
\draw[dashed] (135:0) .. controls +(0,1) and +(0.5,-1) .. (130:2);
\fill (-0.1,1.15) circle (0pt) node {{\small $\Delta_W$}};
\fill (-0.1,-1.15) circle (0pt) node {{\small $\Delta_W$}};
\draw[dashed] (135:0) .. controls +(0,-1) and +(0.5,1) .. (230:2);
\draw[dashed] (270:2) -- (270:3.3)
node[near end,right] {{{\small$\Delta_V$}}};
\draw[dashed] (90:2) -- (90:3.3)
node[near end,right] {{{\small$\Delta_V$}}};
\end{tikzpicture} 
\equiv\;
\begin{tikzpicture}[very thick,scale=0.5,color=blue!50!black, baseline,>=stealth]
\nicepalecolourscheme (0,0) circle (3.5);
\fill (2.2,-2.2) circle (0pt) node[white] {{\small$V$}};
\nicecolourscheme (0,0) circle (2);
\fill (1.1,-1.1) circle (0pt) node[white] {{\small$W$}};
\fill (0:2) circle (0pt) node[right] {{\small$X$}};
\draw (0,0) circle (2);
\draw[<-, very thick] (0.100,2) -- (-0.101,2) node[above] {}; 
\draw[<-, very thick] (-0.100,-2) -- (0.101,-2) node[below] {}; 
\fill (135:0) circle (3.3pt) node[left] {{\small$\psi$}};
\end{tikzpicture} 
\; ,
\end{align}
where as usual (and explained in more detail in Section~\ref{subsec:bicatLG}) evaluation and coevaluation maps are denoted as caps and cups, respectively, $\rho_X$ is the right action of $\Delta_W$ on~$X$, and we always read diagrams like the above from bottom to top. Thus~\eqref{defectactionIntro} equals $\widetilde\eval_X \circ (1_{X}\otimes (\rho_{X} \circ (1_X \otimes \psi)\circ \rho_{X}^{-1})) \circ \coev_X$, from which in Section~\ref{sec:defectaction} we will prove the general formula 
\be\label{DrXpsi}
\mathcal D_r(X)(\psi)  = (-1)^{{m+1}\choose 2} \Res_{k[x,z]/k[z]} \left[ \frac{\psi \str\big( \partial_{x_1} D\ldots \partial_{x_n} D \, \partial_{z_1} D\ldots \partial_{z_m} D\big) \underline{\operatorname{d}\! x}}{\partial_{x_1} W \ldots \partial_{x_n} W} \right]
\ee
as well as various properties such as $\mathcal D_r(X\otimes Y) = \mathcal D_r(X) \circ \mathcal D_r(Y)$. 

One may also consider the situation in~\eqref{defectactionIntro} with an additional defect field $\Phi\in\End(X)$ inserted on the $X$-loop. We will see that this simply amounts to the insertion of~$\Phi$ as a factor inside the supertrace in~\eqref{DrXpsi}. As the two special cases $W=0$ and $V=0$ we thus obtain
$$
\begin{tikzpicture}[very thick,scale=0.4,color=blue!50!black, baseline,>=stealth]
\nicepalecolourscheme (0,0) circle (3.5);
\fill (-2.1,-2.1) circle (0pt) node[white] {{\small$V$}};
\shadedraw[top color=white, bottom color=white, draw=white] (0,0) circle (2);
\fill (180:1.8) circle (0pt) node[left] {{\small$X$}};
\fill (130:2) circle (3.3pt) node[left] {{\small$\Phi$}};
\draw (0,0) circle (2);
\draw[<-, very thick] (0.100,2) -- (-0.101,2) node[above] {}; 
\draw[<-, very thick] (-0.100,-2) -- (0.101,-2) node[below] {}; 
\end{tikzpicture} 
 = 
(-1)^{{m+1\choose 2}} \str\big( \Phi \Lambda_X^{(z)} \big)
\, , \qquad
\begin{tikzpicture}[very thick,scale=0.4,color=blue!50!black, baseline,>=stealth]
\nicecolourscheme (0,0) circle (2);
\fill (-1.0,-1.0) circle (0pt) node[white] {{\small$W$}};
\fill (180:1.8) circle (0pt) node[left] {{\small$X$}};
\fill (130:2) circle (3.3pt) node[left] {{\small$\Phi$}};
\draw (0,0) circle (2);
\draw[<-, very thick] (0.100,2) -- (-0.101,2) node[above] {}; 
\draw[<-, very thick] (-0.100,-2) -- (0.101,-2) node[below] {}; 
\fill (135:0) circle (3.3pt) node[left] {{\small$\psi$}};
\end{tikzpicture} 
= 
\Res_{k[x]/k} \!\!\left[ \frac{\psi \str\big( \Phi \Lambda_X^{(x)} \big) \underline{\operatorname{d}\! x}}{\partial_{x_1} W \ldots \partial_{x_n} W} \right] 
$$
where $\Lambda_X^{(x)} = \partial_{x_1} D\ldots \partial_{x_n} D$ and $\Lambda_X^{(z)} = \partial_{z_1} D\ldots \partial_{z_m} D$. 
In this way we respectively recover the \textsl{boundary-bulk map} (which reduces to the \textsl{Chern character} $(-1)^{{m+1\choose 2}} \str( \partial_{z_1} D\ldots \partial_{z_m} D)$ for $\Phi=1$) and the  \textsl{Kapustin-Li disc correlator}. 

Another application of our construction of adjunctions in $\LG$ is a new proof of the \textsl{Cardy condition} (see Section~\ref{sec:ocTFT} for its precise statement). This generalisation of the Hirzebruch-Riemann-Roch theorem is the most ``quantum'' among the axioms for open/closed TFTs (as it stems from a one-loop diagram) and may accordingly be viewed as a particularly deep structure. In the case of Landau-Ginzburg models it was proved only recently in~\cite{pv1002.2116} when $k$ is a field of characteristic zero. Our proof works for any ring~$k$ and simply follows from the fact that the 2-morphism in $\LG$ to be read off from the diagram 
$$
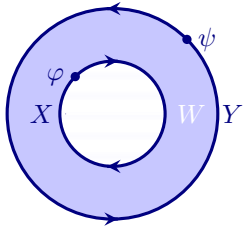
\begin{tikzpicture}[very thick,scale=0.7,color=blue!50!black, baseline,>=stealth]
\nicecolourscheme (0,0) circle (2);
\nicereallynocolourscheme (0,0) circle (1);
\fill (1.5,0) circle (0pt) node[white] {{\small$W$}};
\draw (0,0) circle (2);
\draw[->, very thick] (-0.100,2) -- (-0.101,2) node[above] {}; 
\draw[->, very thick] (0.100,-2) -- (0.101,-2) node[below] {}; 
\fill (45:2) circle (2.5pt) node[right] {{\small$\psi$}};
\draw (0,0) circle (1);
\draw[->, very thick] (0.100,1) -- (0.101,1) node[above] {}; 
\draw[->, very thick] (-0.100,-1) -- (-0.101,-1) node[below] {}; 
\fill (135:1) circle (2.5pt) node[left] {{\small$\varphi$}};
\fill (180:0.9) circle (0pt) node[left] {{\small$X$}};
\fill (0:2.7) circle (0pt) node[left] {{\small$Y$}};
\end{tikzpicture} 
$$
(which is to be identified with an annulus correlator) can be evaluated in two ways: either by first contracting the inner $X$-loop and then contracting the outer $Y$-loop, or by first fusing~$X$ with~$Y$ and then contracting the fused $(X^\vee \otimes Y)$-loop. Applying special cases of our evaluation and coevaluation maps then immediately produces the Cardy condition, see Theorem~\ref{thm:CardyCondition}. 

Let us conclude with future applications of our results. One of the most intriguing properties of Landau-Ginzburg models is that they are on one side of the \textsl{CFT/LG correspondence}. This roughly states that many aspects of a large class of conformal field theories (CFTs) can be described in terms of (non-conformal) Landau-Ginzburg models, and one may wonder which structures encountered in rational CFT can also be found in the theory of Landau-Ginzburg models.


One example is the \textsl{generalised orbifold} procedure of~\cite{ffrs0909.5013} which constructs all rational CFTs of fixed central charge and with identical left and right chiral algebras from any given single such CFT. Carried over to Landau-Ginzburg models this leads to the following picture: under the right circumstances a Landau-Ginzburg model with potential~$V$ can be obtained from a model with potential~$W$ by identifying an object~$A$ in the monoidal category $\LG(W,W)$ that can be equipped with the structure of a special symmetric Frobenius algebra (see e.\,g.~\cite[Section~3]{tft1}). Then the category of matrix factorisations of~$V$ is equivalent to the category of $A$-modules. The results of the present paper facilitate the construction of suitable algebras; the details appear in~\cite{genorb}. 



\medskip

The rest of the present paper is organised as follows. In Section~\ref{sec:Background} we collect necessary background material on bicategories with adjoints, matrix factorisations, noncommutative forms, residues, and homological perturbation theory. Section~\ref{section:atiyahclasses} introduces Atiyah classes, which together with homological perturbation allow us to invert and lift up to homotopy certain maps pertaining to the stabilised diagonal in Section~\ref{section:pertandhtpy}. Using these results we construct explicit evaluation and coevaluation maps in Section~\ref{sec:derivcoeval} and prove that they indeed endow the bicategory $\LG$ with left and right adjoints in Section~\ref{sec:Zorro}. In Section~\ref{sec:wiggliesandsigns} we discuss the details of the graphical calculus as well as pivotality, both important for applications of our main result, some of which we describe in the following three sections: defect action on bulk fields in Section~\ref{sec:defectaction}, open/closed TFT and in particular the Cardy condition in Section~\ref{sec:ocTFT}, and a bicategorical trace in Section~\ref{sec:shadows}. 

\begin{acknowledgements}
It is a pleasure to thank Ilka Brunner, Daniel Plencner, Peter Petersen and Ingo Runkel and the anonymous referee. We owe a special debt to Ragnar Buchweitz, who at an early stage of this project suggested the use of noncommutative forms and associative Atiyah classes; the reader can see for themselves the formative effect this had on the paper. Nils Carqueville thanks Paul Balmer, the UCLA math department and Espresso Profeta for hospitality and excellent working conditions. 
\end{acknowledgements}

\section{Background}\label{sec:Background}

Throughout rings are commutative and $k$ is any ring.

\subsection{Bicategories and adjunction}\label{subsec:bicat}

In this section we recall the theory of bicategories, with \cite{bor94} as our main reference. Standard examples of bicategories include the bicategory of small categories (with objects, 1- and 2-morphisms given by categories, functors and natural transformations, respectively) or rings (rings, bimodules, bimodule maps). In Section~\ref{subsec:bicatLG} we introduce the main example of interest, the bicategory of Landau-Ginzburg models. The basic references for bicategories are \cite{benabou, gray, kellystreet, lack}.

\begin{definition} A bicategory $\cat{B}$ consists of the following data:
\begin{itemize}
\item A class $|\cat{B}|$ of \textsl{objects}.
\item For each pair $A,B$ of objects a small category $\cat{B}(A,B)$ whose objects we call \textsl{1-morphisms} and whose arrows we call \textsl{2-morphisms}. Composition of $2$-morphisms is denoted $\delta \circ \gamma$.
\item For each triple $A,B,C$ of objects a functor
\[
c_{ABC}: \cat{B}(A,B) \times \cat{B}(B,C) \lto \cat{B}(A,C)\,.
\]
Given $1$-morphisms $f: A \lto B$ and $g: B \lto C$ we write $g \otimes f$ for their \textsl{composite} $c_{ABC}( f, g )$, and given $2$-morphisms $\gamma: f \lto f'$ and $\delta: g \lto g'$ we write $\delta \otimes \gamma$ for $c_{ABC}(\gamma, \delta)$.
\item For each object $A$ an identity $1$-morphism $\Delta_A: A \lto A$.
\item For each triple of composable $1$-morphisms $h, g, f$ a $2$-isomorphism
\[
\alpha_{fgh}: (h \otimes g) \otimes f \lto h \otimes (g \otimes f)
\]
natural with respect to $2$-morphisms in all three variables.
\item For each $1$-morphism $f: A \lto B$ a pair of $2$-isomorphisms
\begin{align*}
\lambda_f: \Delta_B \otimes f \lto f
\,,\qquad
\rho_f: f \otimes \Delta_A \lto f
\end{align*}
natural with respect to $2$-morphisms in the variable $f$.
\end{itemize}
This data is subject to two coherence axioms, one involving the associator $\alpha$ and another also involving the left and right unit actions $\lambda, \rho$, see \cite[(7.18),\,(7.19)]{bor94}.
\end{definition}

The identity $2$-endomorphism of a $1$-morphism $f: A \lto B$ is denoted $1_f$ and the identity $2$-endomorphism of $\Delta_A$ is denoted $1_A$. For the remainder of this section, $\cat{B}$ denotes a bicategory. A bicategory with one object is the same data as a monoidal category, and in general $\cat{B}(A,A)$ is a monoidal category with unit $\Delta_A$ for each object $A$. 

It is convenient to denote $2$-morphisms in a bicategory using \textsl{string diagram} notation. This was introduced in \cite{JSGoTCI,JSGoTCII} and the reader can also find very clear explanations in 
\cite{khovdia,ladia}.  
In order to fix our notation, recall that the diagram 
\be\label{eq:demonstratediagram}
\begin{tikzpicture}[very thick,scale=1.0,color=blue!50!black, baseline=.4cm]
\fill (0,0) circle (2.9pt) node (gamma) {};
\draw[line width=0pt] 
(0.25,0.2) node[line width=0pt] (gamma2) {{\small $\gamma$}}
(0,-1) node[line width=0pt] (A) {{\small $A$}}
(-1,0.75) node[line width=0pt] (B) {{\small $B$}}
(1,0.75) node[line width=0pt] (C) {{\small $C$}};
\draw
	(-1.5,-1.5) -- (0,0)
node[midway, left] {{{\small$f$}}};
\draw
	(1.5,-1.5) -- (0,0)
node[midway, right] {{{\small$h$}}};
\draw
	(0,0) -- (0,1.5)
node[midway, left] {{{\small$g$}}};
\end{tikzpicture} 
\ee 
represents a $2$-morphism $\gamma: f \otimes h \lto g$ in $\cat{B}$ with $f: A \lto B$, $g: C \lto B$ and $h: C \lto A$ all $1$-morphisms. We call $\gamma$ the \textsl{value} of the diagram \eqref{eq:demonstratediagram}. In the following we will often refrain from displaying labels for two-dimensional domains in such diagrams. 

All such diagrams are \textsl{progressively planar} in the sense of \cite{JSGoTCI}, i.\,e.~lines proceed strictly upwards. It is straightforward to check \cite[Theorem $1.2$]{JSGoTCI} that an arbitrary such diagram may be unambiguously assigned a value as a $2$-morphism in $\cat{B}$, by ``tensoring horizontally and composing vertically'' and this justifies rigorously the use of diagrams like the one above. Where appropriate, we allow ourselves to migrate the line labels so that they decorate the top and bottom horizontal boundaries, as e.\,g.~in the diagrams \eqref{Zorros} below.


Our references for adjunction in bicategories are \cite[Chapter $6$]{gray} and \cite{kellystreet,kelly}.

\begin{definition}\label{def:adjointbicats} An \textsl{adjunction} between $1$-morphisms $f: A \lto B$ and $g: B \lto A$ is a pair of $2$-morphisms
\be\label{evcoev}
\eval : g \otimes f \lra \Delta_A \, , \qquad \coev : \Delta_B \lra f \otimes g
\ee
such that the following two composites evaluate to identities:
\begin{align}
\xymatrix@C+1.5pc
{
f \ar[r]^-{\lambda_f^{-1}} & \Delta_B \otimes f \ar[r]^-{\coev \otimes 1_f} & ( f \otimes g ) \otimes f \ar[r]^-{\alpha_{fgf}} & f \otimes ( g \otimes f ) \ar[r]^-{1_f \otimes \eval} & f \otimes \Delta_A \ar[r]^-{\rho_f} & f
} ,\nonumber\\
\xymatrix@C+1.5pc
{
g \ar[r]^-{\rho_g^{-1}} & g \otimes \Delta_B \ar[r]^-{1_g \otimes \coev} & g \otimes ( f \otimes g ) \ar[r]^-{\alpha^{-1}_{gfg}} & (g \otimes f) \otimes g \ar[r]^-{\eval \otimes 1_g} & \Delta_A \otimes g \ar[r]^-{\lambda_g} & g
} .\label{uglyZorro1}
\end{align}
In this case we say that $g$ is \textsl{left adjoint} to $f$ and that $f$ is \textsl{right adjoint} to $g$, and we write $g \dashv f$. The $2$-morphisms $\eval$ and $\coev$ are referred to as the \textsl{evaluation} and \textsl{coevaluation} maps.
\end{definition}

\begin{definition}\label{def:bicatduals}
$\mathcal B$ \textsl{has left adjoints} (resp.~\textsl{has right adjoints}) if every 1-morphism in $\cat{B}$ admits a left adjoint (resp.~admits a right adjoint). If they exist these adjoints are unique up to canonical isomorphism, and the unique left and right adjoints of $f$ are denoted by ${}^\dual f$ and $f^\dual$, respectively.
\end{definition}

If a $1$-morphism $f: A \lto B$ has both a left and right adjoint then we write the evaluation and coevaluation maps for the adjunction ${}^\dual f \dashv f$ with $f$ as a subscript, that is
\be\label{evcoev2}
\eval_f : {}^\dual f \otimes f \lra \Delta_A \, , \qquad \coev_f : \Delta_B \lra f \otimes {}^\dual f\,.
\ee
For the adjunction $f \dashv f^\dual$ we write the evaluation and coevaluation maps as
\be\label{evcoev3}
\widetilde\eval_f : f \otimes f^\dual \lra \Delta_B \, , \qquad \widetilde\coev_f : \Delta_A \lra f^\dual \otimes f\,.
\ee
The choice, for each $1$-morphism $f: A \lto B$, of a right adjoint $f^\dual$ together with evaluation and coevaluation maps gives rise to a canonically defined functor 
\be\label{eq:rightadjointfunctor}
(-)^\dual: \cat{B}(A, B)^{\textrm{op}} \lto \cat{B}(B,A)
\ee
where for a $2$-morphism $\varphi: f_1 \lto f_2$ the $2$-morphism $\varphi^\dual$ is unique making the following diagrams commute: 
\be\label{eq:funcrightdual}
\xymatrix@C+1pc@R+1pc{
f_1 \otimes f_2^\dual \ar[r]^-{\varphi \otimes 1}\ar[d]_-{1 \otimes \varphi^\dual} & f_2 \otimes f_2^\dual \ar[d]^-{\widetilde\eval} \, , \\
f_1 \otimes f_1^\dual \ar[r]_{\widetilde\eval} & \Delta_B
} \qquad
\xymatrix@C+1pc@R+1pc{
\Delta_A \ar[d]_-{\widetilde\coev} \ar[r]^-{\widetilde\coev} & f_1^\dual \otimes f_1 \ar[d]^-{1 \otimes \varphi}\\
f_2^\dual \otimes f_2 \ar[r]_-{\varphi^\dual \otimes 1} & f_1^\dual \otimes f_2
}\,.
\ee
Similarly one defines a canonical functor 
\be\label{eq:leftadjointfunctor}
{}^\dual(-): \cat{B}(A, B)^{\textrm{op}} \lto \cat{B}(B,A)\,.
\ee
It is natural to use string diagrams when working with adjoints in a bicategory. In this language the evaluation and coevaluation maps~\eqref{evcoev2} are written as
\be\label{leftadjunctionmaps}
\eval_{f} = 
\begin{tikzpicture}[very thick,scale=1.0,color=blue!50!black, baseline=.6cm]
\draw[line width=0pt] 
(2.5,1.6) node[line width=0pt] (I) {{\small$\Delta_A$}}
(3,0) node[line width=0pt] (D) {{\small $f\vphantom{f^\dual}$}}
(2,0) node[line width=0pt] (s) {\small{${}^\dual f$}}; 
\draw[directed] (D) .. controls +(0,1) and +(0,1) .. (s);
\draw[dashed] (2.5,0.81) -- (I);
\end{tikzpicture}
\equiv
\begin{tikzpicture}[very thick,scale=1.0,color=blue!50!black, baseline=.6cm]
\draw[line width=0pt] 
(3,0) node[line width=0pt] (D) {{\small$f\vphantom{f^\dual}$}}
(2,0) node[line width=0pt] (s) {{\small${}^\dual f$}}; 
\draw[directed] (D) .. controls +(0,1) and +(0,1) .. (s);
\end{tikzpicture}
, 
\qquad
\coev_{f} = 
\begin{tikzpicture}[very thick,scale=1.0,color=blue!50!black, baseline=-.6cm,rotate=180]
\draw[line width=0pt] 
(2.5,1.6) node[line width=0pt] (I) {{\small$\Delta_B$}}
(3,0) node[line width=0pt] (D) {{\small$f\vphantom{{}^\dual f}$}}
(2,0) node[line width=0pt] (s) {{\small${}^\dual f$}}; 
\draw[redirected] (D) .. controls +(0,1) and +(0,1) .. (s);
\draw[dashed] (2.5,0.81) -- (I);
\end{tikzpicture}
\equiv
\begin{tikzpicture}[very thick,scale=1.0,color=blue!50!black, baseline=-.6cm,rotate=180]
\draw[line width=0pt] 
(3,0) node[line width=0pt] (D) {{\small$f\vphantom{A^\dual}$}}
(2,0) node[line width=0pt] (s) {{\small${}^\dual f$}}; 
\draw[redirected] (D) .. controls +(0,1) and +(0,1) .. (s);
\end{tikzpicture}
\, , 
\ee
where we started to abide by the rule typically not to display lines for identity 1-morphisms. 
We stress that, while suggestive, the arrows in these two diagrams have no meaning beyond that of the dot in \eqref{eq:demonstratediagram}. They are merely decorations intended to remind us that this diagram depicts an evaluation or coevaluation, with the direction of the arrow alerting us to the type. In particular, we do \textsl{not} mean that there is an oriented line labelled with $f$, rather, there are three unoriented lines adjacent to the vertex, which in the case of $\coev_f$ are labelled $f, {}^\dual f$ and $\Delta_B$.

In Section~\ref{sec:wiggliesandsigns} we will explain a richer kind of diagrammatics in which the lines are honestly oriented and a label $f$ on a downwards oriented line is understood in terms of the adjoints of $f$, but until then our diagrams have the simpler meaning described above.

With this notation the defining relations~\eqref{uglyZorro1} translate into the \textsl{Zorro moves}
\be\label{Zorros}
\begin{tikzpicture}[very thick,scale=1.0,color=blue!50!black, baseline=0cm]
\draw[line width=0] 
(-1,1.25) node[line width=0pt] (A) {{\small $f$}}
(1,-1.25) node[line width=0pt] (A2) {{\small $f$}}; 
\fill (0.5,1.1) circle (0pt) node {{\small $\eval_f$}};
\fill (-0.35,-1.2) circle (0pt) node {{\small $\coev_f$}};
\draw[directed] (0,0) .. controls +(0,-1) and +(0,-1) .. (-1,0);
\draw[directed] (1,0) .. controls +(0,1) and +(0,1) .. (0,0);
\draw (-1,0) -- (A); 
\draw (1,0) -- (A2); 
\end{tikzpicture}
=
\begin{tikzpicture}[very thick,scale=1.0,color=blue!50!black, baseline=0cm]
\draw[line width=0] 
(0,1.25) node[line width=0pt] (A) {{\small $f$}}
(0,-1.25) node[line width=0pt] (A2) {{\small $f$}}; 
\draw (A2) -- (A); 
\end{tikzpicture}
\, , \qquad
\begin{tikzpicture}[very thick,scale=1.0,color=blue!50!black, baseline=0cm]
\draw[line width=0] 
(1,1.25) node[line width=0pt] (A) {{\small ${}^\dual f$}}
(-1,-1.25) node[line width=0pt] (A2) {{\small ${}^\dual f$}}; 
\fill (-0.4,1.1) circle (0pt) node {{\small $\eval_f$}};
\fill (0.6,-1.2) circle (0pt) node {{\small $\coev_f$}};
\draw[directed] (0,0) .. controls +(0,1) and +(0,1) .. (-1,0);
\draw[directed] (1,0) .. controls +(0,-1) and +(0,-1) .. (0,0);
\draw (-1,0) -- (A2); 
\draw (1,0) -- (A); 
\end{tikzpicture}
=
\begin{tikzpicture}[very thick,scale=1.0,color=blue!50!black, baseline=0cm]
\draw[line width=0] 
(0,1.25) node[line width=0pt] (A) {{\small ${}^\dual f$}}
(0,-1.25) node[line width=0pt] (A2) {{\small ${}^\dual f$}}; 
\draw (A2) -- (A); 
\end{tikzpicture} \, .
\ee
From now on we shall not label the graphical representation of adjunction maps as we did in~\eqref{Zorros}. 
Similarly to~\eqref{leftadjunctionmaps} the evaluation and coevaluation maps \eqref{evcoev3} are written
$$
\widetilde\eval_{f} = 
\begin{tikzpicture}[very thick,scale=1.0,color=blue!50!black, baseline=.6cm]
\draw[line width=0pt] 
(2.5,1.6) node[line width=0pt] (I) {{\small$\Delta_B$}}
(3,0) node[line width=0pt] (D) {{\small $f^\dual\vphantom{f^\dual}$}}
(2,0) node[line width=0pt] (s) {\small{$f\vphantom{f^\dual}$}}; 
\draw[redirected] (D) .. controls +(0,1) and +(0,1) .. (s);
\draw[dashed] (2.5,0.81) -- (I);
\end{tikzpicture}
\equiv
\begin{tikzpicture}[very thick,scale=1.0,color=blue!50!black, baseline=.6cm]
\draw[line width=0pt] 
(3,0) node[line width=0pt] (D) {{\small$f^\dual$}}
(2,0) node[line width=0pt] (s) {{\small$f\vphantom{f^\dual}$}}; 
\draw[redirected] (D) .. controls +(0,1) and +(0,1) .. (s);
\end{tikzpicture}
, 
\qquad
\widetilde\coev_{f} = 
\begin{tikzpicture}[very thick,scale=1.0,color=blue!50!black, baseline=-.6cm,rotate=180]
\draw[line width=0pt] 
(2.5,1.6) node[line width=0pt] (I) {{\small$\Delta_A$}}
(3,0) node[line width=0pt] (D) {{\small$f^\dual$}}
(2,0) node[line width=0pt] (s) {{\small$f\vphantom{f^\dual}$}}; 
\draw[directed] (D) .. controls +(0,1) and +(0,1) .. (s);
\draw[dashed] (2.5,0.81) -- (I);
\end{tikzpicture}
\equiv
\begin{tikzpicture}[very thick,scale=1.0,color=blue!50!black, baseline=-.6cm,rotate=180]
\draw[line width=0pt] 
(3,0) node[line width=0pt] (D) {{\small$f^\dual$}}
(2,0) node[line width=0pt] (s) {{\small$f\vphantom{f^\dual}$}}; 
\draw[directed] (D) .. controls +(0,1) and +(0,1) .. (s);
\end{tikzpicture}
$$
satisfying the associated Zorro moves
\be\label{otherZorros}
\begin{tikzpicture}[very thick,scale=1.0,color=blue!50!black, baseline=0cm]
\draw[line width=0] 
(1,1.25) node[line width=0pt] (A) {{\small $f$}}
(-1,-1.25) node[line width=0pt] (A2) {{\small $f$}}; 
\draw[redirected] (0,0) .. controls +(0,1) and +(0,1) .. (-1,0);
\draw[redirected] (1,0) .. controls +(0,-1) and +(0,-1) .. (0,0);
\draw (-1,0) -- (A2); 
\draw (1,0) -- (A); 
\end{tikzpicture}
=
\begin{tikzpicture}[very thick,scale=1.0,color=blue!50!black, baseline=0cm]
\draw[line width=0] 
(0,1.25) node[line width=0pt] (A) {{\small $f$}}
(0,-1.25) node[line width=0pt] (A2) {{\small $f$}}; 
\draw (A2) -- (A); 
\end{tikzpicture}
\, , \qquad
\begin{tikzpicture}[very thick,scale=1.0,color=blue!50!black, baseline=0cm]
\draw[line width=0] 
(-1,1.25) node[line width=0pt] (A) {{\small $f^\dual$}}
(1,-1.25) node[line width=0pt] (A2) {{\small $f^\dual$}}; 
\draw[redirected] (0,0) .. controls +(0,-1) and +(0,-1) .. (-1,0);
\draw[redirected] (1,0) .. controls +(0,1) and +(0,1) .. (0,0);
\draw (-1,0) -- (A); 
\draw (1,0) -- (A2); 
\end{tikzpicture}
=
\begin{tikzpicture}[very thick,scale=1.0,color=blue!50!black, baseline=0cm]
\draw[line width=0] 
(0,1.25) node[line width=0pt] (A) {{\small $f^\dual$}}
(0,-1.25) node[line width=0pt] (A2) {{\small $f^\dual$}}; 
\draw (A2) -- (A); 
\end{tikzpicture} \, .
\ee
Note that in general there is no reason for the left and right adjoints to coincide, 
i.\,e.~in general there is no reason that $f^\dagger \cong {}^\dagger f$ for any given~$f$. 

\subsection{Bicategory of Landau-Ginzburg models}\label{subsec:bicatLG}

Next we define the bicategory $\LG$ of Landau-Ginzburg models over the base ring $k$, with relevant examples being $k=\nC$ and $k=\nC[t_1,\ldots,t_d]$. 

\begin{definition}\label{defn:potential} A polynomial $W \in k[x_1,\ldots,x_n]$ is a \textsl{potential} if (we write $f_i = \partial_{x_i} W$)
\begin{itemize}
\item[(i)] $f_1,\ldots,f_n$ is a quasi-regular sequence;
\item[(ii)] $k[x_1,\ldots,x_n]/(f_1,\ldots,f_n)$ is a finitely generated free $k$-module;
\item[(iii)] the Koszul complex of $f_1,\ldots,f_n$ is exact except in degree zero.
\end{itemize}
Any regular sequence is quasi-regular. If $k$ is noetherian then $f_1,\ldots,f_n$ is quasi-regular if and only if the image of the sequence is regular in $k[x_1,\ldots,x_n]_{\mf{m}}$ for every maximal ideal $\mf{m} \supseteq (f_1,\ldots,f_n)$. In particular (i) implies (iii) when $k$ is noetherian. For more on quasi-regular sequences see \cite[Chapitre $0$ \S 15.1]{EGA4} and \cite[Section\,10.68]{stacks_project}.
\end{definition}


\begin{example} Let $W \in \nC[x_1,\ldots,x_n]$ be given and for a point $P \in \nC^n$ with corresponding maximal ideal $\mf{m}_P$ denote the Milnor algebra and Milnor number of $W$ at $P$ by 
\[
M_{W,P} = \nC[ x_1,\ldots,x_n ]_{\mf{m}_P} / ( f_1,\ldots,f_n )\, , \qquad \mu(W, P) = \dim_{\mathbb{C}} M_{W,P}\,.
\]
This will be zero unless $P$ is a critical point of the function $W$, and $\mu(W,P) < \infty$ if and only if $P$ is an isolated critical point of $W$ \cite[Section 2.1]{greuel}. Suppose all the critical points of $W$ are isolated, that is, $\mu(W,P) < \infty$ for all points $P$. Then with $R = \nC[x_1,\ldots,x_n]/(f_1,\ldots,f_n)$ it follows that $R_{\mf{m}_P}$ is a finite-dimensional $\nC$-algebra for every point $P$, and thus that $R$ itself is finite-dimensional. Since the images of $f_1,\ldots,f_n$ in $\nC[x_1,\ldots,x_n]_{\mf{m}_P}$ are regular whenever $\mf{m}_P$ contains $(f_1,\ldots,f_n)$ it follows that this sequence is quasi-regular in $\nC[x_1,\ldots,x_n]$, and hence that $W$ is a potential.
\end{example}

Objects of $\LG$ are pairs $(x, W)$ where $x = (x_1,\ldots,x_n)$ is an ordered sequence of variables and $W \in R = k[x_1,\ldots,x_n]$ is a potential. Typically we will write $k[x]$ for $k[x_1,\ldots,x_n]$. We will usually leave the chosen variable ordering implicit, and refer to objects of $\LG$ as pairs $(R, W)$ or even just a potential $W$ if this will not cause confusion.

The category $\LG(W, V)$ is defined in terms of matrix factorisations, which we now recall. Let $R = k[x_1,\ldots,x_n]$. A \textsl{linear factorisation} of $W\in R$ is a $\nZ_2$-graded $R$-module $X=X^0\oplus X^1$ together with an odd $R$-linear endomorphism~$d_X$ such that $d_X^2=W\cdot 1_X$. If~$X$ is a free $R$-module then the pair $(X,d_X)$ is called a \textsl{matrix factorisation}
\cite{EisenbudMF}, 
and we often refer to it by~$X$ without explicitly mentioning the \textsl{differential} $d_X$; given a basis for~$X$ we sometimes identify the latter with the associated matrix
(we follow the convention that matrices act to the right): 
$$
d_{X} = \begin{pmatrix} 0 & d_X^1 \\ d_X^0 & 0\end{pmatrix} .
$$
Given two linear factorisations $X,Y$ of $W$, $\Hom_R(X,Y)$ is a $\nZ _2$-graded complex with differential
\[
\varphi \lmt d_Y \circ \varphi - (-1)^{|\varphi|} \varphi \circ d_X\,.
\]
A \textsl{morphism} of linear factorisations $(X,d_X)$ and $(Y,d_Y)$ is an even $R$-linear map $\varphi: X \longrightarrow Y$ such that $d_Y \varphi = \varphi d_X$. Two morphisms $\varphi, \psi: X\lra Y$ are \textsl{homotopic} if there exists an odd $R$-linear map $\lambda:X\lra Y$ such that $d_Y\lambda + \lambda d_X = \psi-\varphi$. Equality up to homotopy is an equivalence relation.

Given a linear factorisation $X$ of $W$ the \textsl{dual factorisation} $X^\vee = \Hom_R(X, R)$ is a linear factorisation of $-W$ with $d_{X^\vee}( \nu ) = -(-1)^{|\nu|} \nu \circ d_X$. In terms of matrices we have 
\be\label{eq:differentials_adjoints}
d_{X^\vee} = \begin{pmatrix} 0 & (d_X^0)^\vee \\ -(d_X^1)^\vee & 0\end{pmatrix} .
\ee

The \textsl{homotopy category of linear factorisations} $\HF(R,W)$ is the category of linear factorisations of $W\in R$ modulo homotopy. We denote by $\HMF(R,W)$ its full subcategory of matrix factorisations, and we write $\hmf(R,W)$ for the full subcategory of \textsl{finite-rank matrix factorisations}, i.\,e.~the matrix factorisations $X$ whose underlying $R$-module is free of finite rank. These three categories have standard triangulated structures whose shift functor we denote $[1]$
see e.\,g.~\cite{yoshino}. 

Since we work with polynomials rather than power series, $\hmf(R,W)$ is not necessarily idempotent complete.\footnote{For an example of this phenomenon we refer to~\cite[Example A$.5$]{keller}.} However $\HMF(R,W)$ has arbitrary coproducts and is therefore idempotent complete \cite{bokstedt,NeemanBook}, and $\hmf(R,W)^\omega$ denotes the idempotent closure of $\hmf(R,W)$ in this larger triangulated category. More concretely, this idempotent closure is the full subcategory of $\HMF(R,W)$ whose objects are those matrix factorisations~$Y$ which are direct summands of finite-rank matrix factorisations (in the homotopy category). This is an idempotent complete triangulated category.

For objects $(R = k[x], W)$ and $(S = k[z], V)$ of $\LG$ we define\footnote{Taking the idempotent completion is necessary because the composition of $1$-morphisms in $\LG$ results in matrix factorisations which, while not finite-rank, are \textsl{summands} in the homotopy category of something finite-rank. There are two natural ways to resolve this: work throughout with power series rings and completed tensor products, or work with idempotent completions. The latter seems less technical.}
\[
\LG( (R, W), (S,V) ) = \hmf( R \otimes_k S, V - W )^\omega = \hmf( k[x,z], V - W)^\omega\,.
\]
That is, a $1$-morphism between potentials $W \lto V$ is a matrix factorisation of $V - W$. We compose such $1$-morphisms using tensor products. Given potentials $W_i \in R_i$ consider matrix factorisations
\begin{align*}
X &\in \LG( W_1, W_2 ) = \hmf(R_1\otimes_k R_2, W_2-W_1)^\omega \, , \\
Y &\in \LG( W_2, W_3 ) = \hmf(R_2\otimes_k R_3, W_3-W_2)^\omega \,.
\end{align*}
The \textsl{tensor product} $Y\otimes_{R_2} X \in \HMF(R_1\otimes_k R_3, W_3-W_1)$ is the $\nZ_2$-graded module 
\begin{gather*}
Y\otimes_{R_2} X = \Big( (Y^0\otimes_{R_2} X^0) \oplus (Y^1\otimes_{R_2} X^1) \Big) \oplus \Big( (Y^0\otimes_{R_2} X^1) \oplus (Y^1\otimes_{R_2} X^0) \Big) \, , \\
d_{Y\otimes X} = d_Y \otimes 1 + 1 \otimes d_X
\end{gather*}
where $1 \otimes d_X$ has the usual Koszul signs when applied to elements. This is a free module of infinite rank over $R_1 \otimes_k R_3$, however by e.\,g.~the argument of~\cite[Section~12]{dm1102.2957} the tensor product $Y \otimes_{R_2} X$ is a direct summand in the homotopy category of something finite-rank, i.\,e.~we may define
\[
Y \circ X := Y \otimes_{R_2} X \in \hmf( R_1 \otimes_k R_3, W_3 - W_1 )^\omega = \LG( W_1, W_3 )\,.
\]
Letting the tensor product act in the obvious way on morphisms this defines a functor
\be\label{eq:backgroundcomptensor}
c_{W_1 W_2 W_3}: \LG( W_1, W_2 ) \times \LG( W_2, W_3 ) \lto \LG( W_1, W_3 ) \, , 
\qquad ( X, Y ) \lmt Y \otimes_{R_2} X\,.
\ee
Given a triple of composable $1$-morphisms $X, Y, Z$ there is a natural isomorphism $\alpha_{XYZ}: (X\otimes Y)\otimes Z \lra X\otimes (Y \otimes Z)$ given by the usual formula $(a \otimes b) \otimes c \lmt a \otimes (b \otimes c)$. In the rest of the paper we drop the ``$\circ$'' notation for composition and use only tensor products.

Finally, we define the units $\Delta_W: W \lto W$ and the left and right unit actions $\lambda, \rho$ for the bicategory $\LG$. Given a ring $R = k[x_1,\ldots,x_n]$ we write $\Re = R\otimes_k R = k[x,x']$ where $x_i = x_i \otimes 1$ and $x_i' = 1 \otimes x_i$. Given $W\in R$ there is always the \textsl{unit matrix factorisation} $\Delta_W \in \hmf(\Re, \widetilde W)$ where $\widetilde W = W\otimes 1 - 1\otimes W$. 
Using formal symbols $\theta_i$ we define the $\Re$-module
$$
\Delta_W = \bigwedge \Big( \bigoplus_{i=1}^n \Re \theta_i \Big)
$$
with the $\Ztwo$-grading given by $\theta$-degree (where deg$\,\theta_i=1$). 
Typically we will omit the wedge product and write  e.\,g.~$\theta_i\wedge \theta_j$ simply as $\theta_i \theta_j$. To describe the differential $d_{\Delta_W}$ we further need the variable-changing maps
${}^{t_i}(-)$ which in any polynomial replace the variable~$x_i$ by the variable~$x'_i$, 
$$
{}^{t_i}(-): k[x,x'] \lra k[x,x'] \, , \qquad f \lmt f\big|_{x_i\lmt x'_i} \, , 
$$
in terms of which we can define difference quotient operators
\be\label{diffquotop}
\partial^{x,x'}_{[i]}: k[x,x'] \lra k[x,x'] \, , \qquad f \lmt \frac{{}^{t_1\ldots t_{i-1}}f - {}^{t_1\ldots t_i}f}{x_i-x'_i} \, . 
\ee
Sometimes we write $\partial_{[i]}$ for $\partial^{x,x'}_{[i]}$. These operators satisfy a kind of Leibniz rule 
which can be checked directly by applying the above definitions. 
\begin{lemma}\label{lem:LeibnizForDQO}
For $f,g \in k[x,x']$ we have $\partial^{x,x'}_{[i]}(fg) = (\partial^{x,x'}_{[i]}f) ({}^{t_1\ldots t_{i}}g) + ({}^{t_1\ldots t_{i-1}}f) (\partial^{x,x'}_{[i]}g)$. 
\end{lemma}

Viewing~$W$ as an element in $k[x] \subset k[x,x']$, 
the differential on $\Delta_W$ is then given by\footnote{We observe that the matrix factorisation $(\Delta_W, d_{\Delta_W})$ depends on the chosen ordering of the ring variables via the operators $\partial^{x,x'}_{[i]}$ in the differential $\delta_{+}$. Another ordering will yield a different, but isomorphic, matrix factorisation; we address this point carefully in Appendix~\ref{app:variableordering}.}
\be\label{DeltaW}
d_{\Delta_W} = \delta_+ + \delta_- \, , \qquad \delta_+ = \sum_{i=1}^n \partial^{x,x'}_{[i]}W\cdot \theta_i \wedge (-)\, , \qquad \delta_- =  \sum_{i=1}^n (x_i-x'_i) \cdot \theta_i^* 
\ee
where~$\theta_i^*$ is defined by linear extension of $\theta_i^*(\theta_j)=\delta_{ij}$, and it acts on an element $\theta_{i_1} \ldots \theta_{i_l}$ of the exterior algebra by the Leibniz rule with Koszul signs. 
We call $\Delta_W$ the unit matrix factorisation as it is the unit with respect to the tensor product of matrix factorisations. It is also referred to as the \textsl{stabilised diagonal} \cite{d0904.4713} or \textsl{Koszul model} of the diagonal. The diagonal here refers to $R$ as an $\Re$-module, which is a linear factorisation of $\widetilde W$ with differential zero since $\widetilde W$ acts as zero on $R$. The morphism of linear factorisations of $\widetilde W$ defined for $r,r' \in R$ by
\be\label{DeltaWstabmap}
\pi: \Delta_W \lra R\,, \qquad \pi( (r \otimes r') \theta_{i_1} \cdots \theta_{i_k} ) = \delta_{k,0} rr'\,,
\ee
is the projection $\Delta_W \lra \Re$ to $\theta$-degree~$0$ followed by multiplication $\Re \lra R$. This morphism is \textsl{universal} \cite{d0904.4713} in the homotopy category of linear factorisations among all morphisms from finite-rank matrix factorisations to $R$ (for a more precise statement see Section \ref{section:liftingproblem}).

To define the unitor maps of $\LG$, take a $1$-morphism $X \in \LG(W_1, W_2) = \hmf(R_1 \otimes_k R_2,W_2-W_1)^\omega$ be given. There are natural maps
\begin{align}
\lambda_X &= \pi \otimes 1_X: \Delta_{W_2} \otimes_{R_2} X \lra X \,, \nonumber\\
\rho_X &= 1_X \otimes \pi : X \otimes_{R_1} \Delta_{W_1} \lra X \label{lambdarho}
\end{align}
which are morphisms in $\hmf(R_1 \otimes_k R_2,W_2-W_1)^\omega$. To summarise:

\begin{proposition} There is a bicategory $\LG$ consisting of the following data: 
\begin{itemize}
\item Objects are pairs $(R,W)$ with $W\in R=k[x]$ a potential. 
\item 1- and 2-morphisms are the objects and morphisms of the categories
\[
\LG((R, W), (S,V)) = \hmf(R \otimes_k S,V-W)^\omega\,.
\]
\item The unit 1-morphisms are $\Delta_W \in \hmf(\Re,\widetilde W)$ of \eqref{DeltaW}.
\item The composition functor is the tensor product functor $c$ of \eqref{eq:backgroundcomptensor}.
\item There are natural 2-isomorphisms $\alpha, \lambda, \rho$ as above. 
\end{itemize}
\end{proposition}
\begin{proof}
This is straightforward; see \cite{kr0401268, McNameethesis, Calinetal, cr0909.4381}. The associator~$\alpha$ is clearly an isomorphism and satisfies the coherence axiom. The unitors $\lambda, \rho$ are both homotopy equivalences by \cite{d0904.4713, kr0401268}, and we will reprove this as part of the developments of Section \ref{section:liftingproblem}. The remaining check is the coherence axiom for the unitors which asserts, for every composable pair $X: W_1 \lto W_2$ and $Y: W_2 \lto W_3$ as above, commutativity of the diagram
\[
\xymatrix{
(Y \otimes \Delta_{W_2}) \otimes X \ar[rr]^\alpha\ar[dr]_{\rho_Y \otimes 1_X} & & Y \otimes ( \Delta_{W_2} \otimes X ) \ar[dl]^{1_Y \otimes \lambda_X}\\
& Y \otimes X
}\,.
\]
But $\rho_Y \otimes 1_X = (1_Y \otimes \pi) \otimes 1_X$ and $1_Y \otimes \lambda_X = 1_Y \otimes (\pi \otimes 1_X)$ so this is clear.
\end{proof}

\begin{remark}\label{remark:integral_functors} It is often helpful to think of a potential $W \in k[x]$ as ``standing in'' for the triangulated category $\hmf(k[x], W)^\omega$, in which case a $1$-morphism $X: W \lto V$ stands for the functor
\[
\Phi_X: \hmf(k[x], W)^\omega \lto \hmf(k[z], V)^\omega\,,\qquad
\Phi_X( Y ) = X \otimes_{k[x]} Y\,.
\]
Working in the bicategory $\LG$ therefore amounts to doing algebra with integral functors. To make this precise, consider the object $\mathbb{I} = (k,0)$ of $\LG$. Then $\LG(\mathbb{I}, W) = \hmf(k[x], W)^\omega$ and $\Phi_X$ is the functor defined by composition with $X$: 
\[
\LG( \mathbb{I}, W ) \lto \LG( \mathbb{I}, V ) \, , \qquad Y \lmt X \otimes Y \, . 
\]
Moreover an adjunction in $\LG$ gives rise in this way to a pair of adjoint functors between categories of matrix factorisations. Another way to say this is that $\LG( \mathbb{I}, - )$ defines a pseudofunctor from $\LG$ to the bicategory of categories, functors and natural transformations, and adjoint pairs are preserved by a pseudofunctor; see \cite[Section\,$4$]{benabou} and \cite{lack}.

Less tersely, given an adjunction $X \dashv X^\dual$ and $1$-morphisms $Y: \mathbb{I} \lto W$ and $Z: \mathbb{I} \lto V$ one deduces from the equation satisfied by the evaluation and coevaluation maps a natural isomorphism
\[
\Hom_{\LG(\mathbb{I}, V)}( X \otimes Y, Z ) \cong \Hom_{\LG(\mathbb{I}, W)}( Y, X^\dual \otimes Z )
\]
which makes the integral functor $\Phi_X$ left adjoint to $\Phi_{X^\dual}$. 
\end{remark}

\subsection{Bar complex}\label{subsec:Bar}

The unit $1$-morphisms in $\LG$ are defined using a Koszul model for the diagonal, but in the study of adjoints we will need another model constructed from the bar complex. For this reason we recall the necessary background on noncommutative forms and the bar complex from 
\cite{cuntzquillen,Loday} 
and construct a map $\Psi$ which relates the Koszul and bar models.

For the first part of this section, algebras are associative unital $k$-algebras that are not necessarily commutative. \textsl{Noncommutative forms} over a $k$-algebra~$R$ can be defined via a universal property which we recall below, and we use two concrete realisations of this universal object. First, we define
\begin{align*}
\Omega^n R = R\otimes \bar R^{\otimes n}
\,,\qquad
\widetilde{\Omega}^n R = \bar R^{\otimes n} \otimes R
\end{align*}
where $\bar R = R/k$ and in this section by $\otimes$ we mean $\otimes_k$. We denote the projection of $a_0\otimes a_1 \otimes \ldots \otimes a_n \in R^{\otimes (n+1)}$ to $\Omega^n R$ by $(a_0,a_1,\ldots,a_n)$ and the projection of $a_n \otimes \ldots \otimes a_0$ to $\widetilde{\Omega}^n R$ by $[a_n,\ldots,a_0]$. The direct sums
$$
\Omega R = \bigoplus_{n\geq 0} \Omega^n R\,, \qquad \widetilde{\Omega} R = \bigoplus_{n\geq 0} \widetilde{\Omega}^n R
$$
are differential graded (dg) algebras $(\Omega R, d, \cdot)$, $(\widetilde{\Omega} R, s, *)$ with multiplication given by
\begin{align*}
(a_0,\ldots,a_m) \cdot (a_{m+1},\ldots, a_{m+n}) &= \sum_{i=0}^m (-1)^{m-i}(a_0,\ldots,a_{i-1},a_i a_{i+1},a_{i+2},\ldots, a_{m+n})\,,\\
[a_{m+n}, \ldots, a_{m+1}] * [a_m,\ldots,a_0] &= \sum_{i=0}^m (-1)^{m-i}[a_{m+n}, \ldots, a_{i+2}, a_{i+1} a_i, a_{i-1}\ldots, a_0]
\end{align*}
and differentials
$$
d: (a_0,\ldots,a_n) \lmt (1,a_0,\ldots,a_n)\,, \qquad s: [a_n,\ldots,a_0] \lmt (-1)^n [a_n, \ldots, a_0, 1]\,.
$$
The proof that $\Omega R$ is a dg-algebra is given in \cite{cuntzquillen} and the proof for $\widetilde{\Omega} R$ follows in the same way. We may therefore write $(a_0,a_1,\ldots,a_n)$ as $a_0da_1\ldots da_n$ and $[a_n,\ldots,a_1,a_0]$ as $sa_n * \ldots * sa_1 * a_0$.

There is a canonical algebra morphism $R \cong \Omega^0 R \lto \Omega R$ and the dg-algebra $\Omega R$ is the free dg-algebra over $k$ generated by $R$ in the following sense: given a dg $k$-algebra $\Gamma$ and a $k$-algebra morphism $u: R \lto \Gamma^0$ there is a unique morphism of dg-algebras $\Omega R \lto \Gamma$ extending $u$ \cite[Proposition $1.1$]{cuntzquillen}. If $S$ is another $k$-algebra and $R$ is an $S$-algebra, by which we mean that there is a morphism of $k$-algebras $S \lto R$ (which does not necessarily have image in the centre of $R$) then there is a notion of relative forms $\Omega_S R$ which has a similar universal property: given a dg $k$-algebra $\Gamma$ and a morphism of $k$-algebras $v: R \lto \Gamma^0$ with the property that $d( v S ) = 0$, there is a unique morphism of dg-algebras $\Omega_S R \lto \Gamma$ extending $v$ \cite[Proposition $2.1$]{cuntzquillen}.

An easy consequence of the universal property is the following:

\begin{lemma}\label{lemma:coeffomeage} If $S$ is a $k$-algebra then the map $R \otimes S \lto \Omega R \otimes S$ extends to an isomorphism of dg $S$-algebras $\Omega_S( R \otimes S) \lto \Omega R \otimes S$. Similarly $\Omega_R( R \otimes S) \cong R \otimes \Omega S$ as dg $R$-algebras.
\end{lemma}
\begin{proof}
This is \cite[Proposition $2.8$]{cuntzquillen}.
\end{proof}

In the case where all rings are commutative, $\Omega_S R$ may be defined by simply replacing $k$ by $S$ in the definition of $\Omega R$. This special case is enough for our applications.

If $Q$ is a dg-algebra then $Q^{\textrm{op}}$ denotes the \textsl{opposite dg-algebra} with $x \cdot y = (-1)^{|x||y|} yx$. The canonical map $R^{\textrm{op}} \lto (\Omega R)^{\textrm{op}}$ extends to an isomorphism of dg-algebras $\Omega(R^{\textrm{op}}) \lto (\Omega R)^{\textrm{op}}$.

\begin{lemma} The map $R \lto \widetilde{\Omega} R$ extends to an isomorphism of dg-algebras $\Omega R \lto \widetilde{\Omega} R$.
\end{lemma}
\begin{proof}
The map $\widetilde{\Omega} R \lto \Omega(R^{\textrm{op}})^{\textrm{op}}$ defined by $[a_n, \ldots, a_0] \lmt (-1)^{\binom{n}{2}} (a_0, a_1, \ldots, a_n)$ is an isomorphism of dg-algebras. But $\Omega(R^{\textrm{op}})^{\textrm{op}} \cong \Omega R$ and the composite isomorphism $\widetilde{\Omega} R \lto \Omega R$ is inverse to the induced map $\Omega R \lto \widetilde{\Omega} R$.
\end{proof}

This means that we are free to use either $\Omega R$ or $\widetilde{\Omega} R$ as a concrete realisation of the dg-algebra of noncommutative forms. If we write $R_i = R$ for $i \in \{1,2\}$ and $\Re = R_1 \otimes R_2$, then the dg-algebras
\be\label{eq:twoomegas}
\Omega R_1 \otimes R_2 \cong \Omega_{R_2}( R_1 \otimes R_2 ) 
\, , \qquad 
R_1 \otimes \widetilde{\Omega} R_2 \cong \widetilde{\Omega}_{R_1}( R_1 \otimes R_2 )
\ee
have the same underlying graded $k$-module which we denote by $\Bar$, i.\,e.
\be\label{BarComplex}
\Bar = \bigoplus_{n\geq 0} \Bar_n \, , \qquad \Bar_n = \Omega^n R_1 \otimes R_2 = R_1 \otimes \widetilde{\Omega}^n R_2 = R \otimes \bar{R}^{\otimes n} \otimes R\, .
\ee
This graded $k$-module $\Bar$ is made into a dg-algebra in two different ways by \eqref{eq:twoomegas}, so for example the projection to $\Bar_n$ of the tensor $a_0 \otimes a_1 \otimes \ldots \otimes a_n \otimes a_{n+1}$ can be presented as
\[
a_0 da_1 \ldots da_n \otimes a_{n+1} = a_0 \otimes sa_1 * \ldots * sa_n * a_{n+1}\,.
\]
The differentials in these two dg-algebra structures on $\Bar$ are written $d = d \otimes 1_R$ and $s = 1_R \otimes s$. 

Left and right multiplication makes $\Omega_{R_2}(\Re)$ into an $\Re$-$\Re$-bimodule. Using the identification with $\Bar$ the left action may be written
\be\label{eq:leftreactionb}
(r \otimes r') \cdot a_0 \otimes a_1 \otimes \ldots \otimes a_n \otimes a_{n+1} = r a_0 \otimes a_1 \otimes \ldots \otimes a_n \otimes r' a_{n+1}
\ee
Similarly, $\widetilde{\Omega}_{R_1}(\Re)$ is an $\Re$-$\Re$-bimodule by left and right multiplication and the right action is
\be\label{eq:rightreactionb}
a_0 \otimes a_1 \otimes \ldots \otimes a_n \otimes a_{n+1} * (r \otimes r') = a_0 r \otimes a_1 \otimes \ldots \otimes a_n \otimes a_{n+1} r'\,.
\ee
It is important to keep in mind that the two $\Re$-$\Re$-bimodule structures defined in this way on the underlying graded $k$-module $\Bar$ do not agree, e.\,g.~the right action of $\Re$ on $\Omega_{R_2}(\Re)$ is not~\eqref{eq:rightreactionb}.


Next we describe a third differential $b'$ on $\Bar$. The \textsl{(normalised) bar complex} is the resolution
\be
\xymatrix{
\cdots \ar[r]^-{b'} & R \otimes \bar R^{\otimes 2} \otimes R \ar[r]^-{b'} & R \otimes \bar R \otimes R \ar[r]^-{b'} & R\otimes R \ar[r]^-{b'} & R \ar[r] & 0
}\label{eq:normalised_bar_cpx_line}
\ee
of~$R$, where the degree-lowering differential~$b'$ is the $k$-linear map
$$
a_0 \otimes \cdots \otimes a_{n+1} \lmt \sum_{i=0}^{n-1} (-1)^i a_0 \otimes \cdots \otimes a_i a_{i+1} \otimes \cdots  \otimes a_{n+1} + (-1)^n a_0 \otimes \cdots \otimes a_{n-1} \otimes a_n a_{n+1} \, . 
$$
From this it is straightforward to check that we have the identities
\be\label{b'd+db'} 
b'd+db'=1_{\Bar} \, , \qquad b' s + s b' = -1_{\Bar}\,.
\ee
Hence $d$ (resp.~$s$) is a right $R$-linear (resp.~left $R$-linear) contracting homotopy for $b'$ on $\Bar$. These identities in degree zero involve $d,s: R \lto R \otimes R$ defined by $d(a) = 1 \otimes a$ and $s(a) = a \otimes 1$.

Henceforth we assume that $R$ is commutative. Recall that $(m,n)$-shuffles are permutations in
$$
\operatorname{Sh}(m,n) = \big\{ \sigma\in S_{m+n} \,|\, \sigma(1)<\sigma(2)<\ldots<\sigma(m), \, \sigma(m+1)<\sigma(m+2)<\ldots<\sigma(m+n) \big\} \, . 
$$
The \textsl{shuffle product}~$\times$ on~$\Bar$ is defined by
\begin{align*}
& (a_0da_1\ldots da_m \otimes a_{m+1}) \times (b_0db_1\ldots db_n \otimes b_{n+1}) \\
& \qquad 
= \sum_{\sigma \in \operatorname{Sh}(m,n)} (-1)^{|\sigma|} a_0 b_0 \, \sigma_\bullet (da_1\ldots da_m db_1 \ldots db_n) \otimes a_{m+1} b_{n+1}
\end{align*}
where 
$\sigma_\bullet(dc_1\ldots dc_j) = dc_{\sigma^{-1}(1)}\ldots dc_{\sigma^{-1}(j)}$. 
Equipped with this multiplication $(\Bar,b',\times)$ is a graded-commutative dg-algebra \cite[Proposition 4.2.2]{Loday}. Here for compactness we use the notation of forms in $\Omega_{R_2}(\Re)$, but one could just as easily write out the formulas in terms of $\widetilde{\Omega}_{R_1}(\Re)$.

Given $W \in R$ we define $\widetilde W = W\otimes 1 - 1\otimes W \in \Re$. Consider the $k$-linear operator
$$
d_{\Bar} = b' + d\widetilde W \times (-) = b' - s\widetilde W \times (-)\, . 
$$

\begin{lemma}\label{lemma:barisafactorisation_pre}
$(d_{\Bar})^2 = \widetilde{W} \times (-)$.
\end{lemma}
\begin{proof}
We have $b'^2=0$ and $d\widetilde W \times d\widetilde W = 0$, so
\begin{align*}
d_{\Bar}^2 (\omega) & = b'(d\widetilde W \times \omega) + d\widetilde W \times b'(\omega) \\
& = b'(d\widetilde W) \times \omega - d\widetilde W \times b'(\omega) + d\widetilde W \times b'(\omega) \\
& = \widetilde W \times \omega\,.
\end{align*}
\end{proof}

The $k$-module $\Bar=\Bar^0 \oplus \Bar^1$ is $\nZ_2$-graded with $\Bar^i = \bigoplus_{n\in 2\nN+i}\Bar_n$ and if we identify $\Omega_{R_2}(\Re)$ with $\Bar$ then \eqref{eq:leftreactionb} is the induced left $\Re$-action and $\widetilde W \times(-) = \widetilde W\cdot(-)$.

\begin{lemma}\label{lemma:barisafactorisation}
$(\Omega_{R_2}(\Re),d_{\Bar})$ is a left $\Re$-linear factorisation of $\widetilde W$. 
\end{lemma}

If we identify $\widetilde{\Omega}_{R_1}(\Re)$ with $\Bar$ then \eqref{eq:rightreactionb} is the right $\Re$-action and $\widetilde W \times (-) = (-) * \widetilde W$.

\begin{lemma}\label{lemma:barisafactorisation2}
$(\widetilde{\Omega}_{R_1}(\Re),d_{\Bar})$ is a right $\Re$-linear factorisation of $\widetilde W$. 
\end{lemma}

We refer to these two linear factorisations as the \textsl{bar models} for the diagonal. In Section \ref{section:pertandhtpy} we will define left and right actions $\widetilde{\Omega}_{R_1}(\Re) \otimes X \lto X$ and $X \otimes \Omega_{R_2}(\Re) \lto X$ as in \eqref{lambdarho} and prove that they are homotopy equivalences (after a suitable completion). For the moment we take the fact that $\Bar$ is another model for the unit action on matrix factorisations as motivation to discuss its relation to the Koszul matrix factorisation~$\Delta_W$. 

Before we do this on the level of linear factorisations we consider the case of $\nZ$-graded complexes. We return to the $k$-algebra $R=k[x_1,\ldots,x_n]$ and write $\Delta = \bigwedge( \bigoplus_{i=1}^n \Re \theta_i)$ so that $(\Delta, \delta_-)$ is the ordinary Koszul complex, see~\eqref{DeltaW}. 

There are two $k$-linear maps between~$\Bar$ and~$\Delta$ which will be important: 
\begin{align}
\Phi & : \Delta \lra \Bar \, , \qquad (r \otimes r')\theta_{i_1}\ldots \theta_{i_p} \lmt \sum_{\sigma\in S_p} (-1)^{|\sigma|} r dx_{i_{\sigma(1)}} \ldots dx_{i_{\sigma(p)}} \otimes r' \, , \label{intro_phi}\\
\Psi & : \Bar \lra \Delta \, , \qquad rdf_1\ldots df_p \otimes r' \lmt \sum_{1\leq i_1<\ldots<i_p\leq n} (r \otimes r')\Big( \prod_{k=1}^p \partial^{x,x'}_{[i_k]} f_k \Big) \, \theta_{i_1} \ldots \theta_{i_p} \, \label{intro_psi}.
\end{align}
The map $\Psi$ was studied in~\cite{sw0911.0917}, we only rephrase the presentation in~\cite[Definition $4.1$]{sw0911.0917} terms of the difference quotient operators suitable for our setting. One easily verifies that $\Psi\Phi = 1_\Delta$. If we give $\Bar$ the left $\Re$-module structure of \eqref{eq:leftreactionb} then $\Phi$ and $\Psi$ are left $\Re$-linear, and if we give $\Bar$ the right $\Re$-module structure of \eqref{eq:rightreactionb} then $\Phi, \Psi$ are right $\Re$-linear.

\begin{lemma}\label{PhiPsiDG}
Both~$\Phi$ and~$\Psi$ are maps of dg-algebras between $(\Delta, \delta_-, \wedge)$ and $(\Bar, b', \times)$. 
\end{lemma}

\begin{proof}
We refer to~\cite{sw0911.0917} for the case of~$\Phi$; since our expression for~$\Psi$ is not manifestly the same as in loc.~cit.~we spell out the proof. Let us first show that~$\Psi$ is compatible with the differentials, writing $\partial_{[i]} = \partial^{x,x'}_{[i]}$. On the one hand we compute $(\delta_- \Psi) (df_1\ldots df_p \otimes 1)$ to be 
\begin{align}
& \delta_ -\Big( \sum_{i_1<\ldots <i_p} (\partial_{[i_1]} f_1) \ldots (\partial_{[i_p]} f_p) \, \theta_{i_1} \ldots \theta_{i_p} \Big) \nonumber \\
& =  \sum_{i_1<\ldots <i_p} (\partial_{[i_1]} f_1) \ldots (\partial_{[i_p]} f_p)  \sum_{k=1}^p (-1)^{k+1} (x_{i_k} - x'_{i_k})\theta_{i_1} \ldots \widehat{\theta_{i_k}} \ldots \theta_{i_p}  \nonumber \\
& = \sum_{k=1}^p (-1)^{k+1}\sum_{i_1<\ldots <i_p} (\partial_{[i_1]} f_1) \ldots ({}^{t_1\ldots t_{i_{k-1}}} f_k - {}^{t_1\ldots t_{i_{k}}} f_k) \ldots (\partial_{[i_p]} f_p) \, \theta_{i_1} \ldots \widehat{\theta_{i_k}} \ldots \theta_{i_p}  \nonumber \\
& = \sum_{2\leq i_2<\ldots< i_p} (f_1 - {}^{t_1\ldots t_{i_{2}-1}} f_1) (\partial_{[i_2]} f_2) \ldots (\partial_{[i_p]} f_p) \, \theta_{i_2} \ldots \theta_{i_p} \nonumber \\
& \qquad + \sum_{k=2}^{p-1} (-1)^{k+1}\sum_{i_1 < \ldots <i_p} (\partial_{[i_1]} f_1) \ldots ({}^{t_1\ldots t_{i_{k-1}}} f_k - {}^{t_1 \ldots t_{i_{k+1}-1}} f_k) \ldots (\partial_{[i_p]} f_p) \, \theta_{i_1} \ldots \widehat{\theta_{i_k}} \ldots \theta_{i_p}  \nonumber \\
& \qquad + (-1)^{p+1} \sum_{i_1<\ldots< i_{p-1}\leq n-1} (\partial_{[i_1]} f_1) \ldots (\partial_{[i_{p-1}]} f_{p-1}) ({}^{t_1\ldots t_{i_{p-1}}} f_p - {}^{t_1\ldots t_n} f_p)  \, \theta_{i_1} \ldots \theta_{i_{p-1}} \nonumber \\
& = \sum_{2\leq t_2<\ldots <i_p} f_1 (\partial_{[i_2]} f_2) \ldots (\partial_{[i_p]} f_p) \, \theta_{i_2} \ldots \theta_{i_p} \nonumber \\
& \qquad + (-1)^{p} \sum_{i_1<\ldots< i_{p-1}\leq n-1} (\partial_{[i_1]} f_1) \ldots (\partial_{[i_{p-1}]} f_{p-1}) \, {}^{t_1\ldots t_n} f_p  \, \theta_{i_1} \ldots \theta_{i_{p-1}}  \label{delPsi} 
\end{align}
while on the other hand $(\Psi b') (df_1\ldots df_p \otimes 1)$ equals
\begin{align*}
& \Psi \Big( f_1 df_2 \ldots df_p \otimes 1 + \sum_{k=1}^{p-1} (-1)^k df_1 \ldots d(f_k f_{k+1}) \ldots df_p \otimes 1 + (-1)^p df_1 \ldots df_{p-1} \otimes f_p \Big) \\
& = \sum_{i_1<\ldots< i_{p-1}} f_1 (\partial_{[i_1]} f_2) \ldots (\partial_{[i_{p-1}]} f_p) \, \theta_{i_1} \ldots \theta_{i_{p-1}} \\
& \qquad + \sum_{k=1}^{p-1} (-1)^k \sum_{i_1<\ldots< i_{p-1}} (\partial_{[i_1]} f_1) \ldots (\partial_{[i_k]} (f_k f_{k+1}) )\ldots (\partial_{[i_{p-1}]} f_p) \, \theta_{i_1} \ldots \theta_{i_{p-1}} \\
& \qquad + (-1)^p \sum_{i_1<\ldots< i_{p-1}} (\partial_{[i_1]} f_1) \ldots (\partial_{[i_{p-1}]} f_{p-1}) \, {}^{t_1\ldots t_n} f_p \, \theta_{i_1} \ldots \theta_{i_{p-1}} \\
& = \sum_{2\leq i_1<\ldots< i_{p-1}} f_1 (\partial_{[i_1]} f_2) \ldots (\partial_{[i_{p-1}]} f_p) \, \theta_{i_1} \ldots \theta_{i_{p-1}} \\ 
& \qquad + \sum_{2\leq i_2 \ldots< i_{p-1}} f_1 (\partial_{[1]} f_2) (\partial_{[i_2]} f_3) \ldots (\partial_{[i_{p-1}]} f_p) \, \theta_{i_1} \ldots \theta_{i_{p-1}} \\ 
& \qquad + \sum_{k=1}^{p-1} (-1)^k \sum_{i_1<\ldots< i_{p-1}} (\partial_{[i_1]} f_1) \ldots \big\{ ({}^{t_1\ldots t_{i_k -1}}f_k)(\partial_{[i_k]} f_{k+1}) \\
& \qquad\quad + (\partial_{[i_k]} f_k) ({}^{t_1\ldots t_{i_k}} f_{k+1}) \big\} \ldots (\partial_{[i_{p-1}]} f_p) \, \theta_{i_1} \ldots \theta_{i_{p-1}} \\
& \qquad + (-1)^p \sum_{i_1<\ldots< i_{p-1}\leq n-1} (\partial_{[i_1]} f_1) \ldots (\partial_{[i_{p-1}]} f_{p-1}) \, {}^{t_1\ldots t_n} f_p \, \theta_{i_1} \ldots \theta_{i_{p-1}} \\
& \qquad\quad + (-1)^p \sum_{i_1<\ldots< i_{p-2}\leq n-1} (\partial_{[i_1]} f_1) \ldots (\partial_{[i_{p-1}]} f_{p-1}) \, {}^{t_1\ldots t_n} f_p \, \theta_{i_1} \ldots \theta_{i_{p-1}} \\
& = \sum_{2\leq i_1<\ldots< i_{p-1}} f_1 (\partial_{[i_1]} f_2) \ldots (\partial_{[i_{p-1}]} f_p) \, \theta_{i_1} \ldots \theta_{i_{p-1}} \\ 
& \qquad + (-1)^p \sum_{i_1<\ldots< i_{p-1}\leq n-1} (\partial_{[i_1]} f_1) \ldots (\partial_{[i_{p-1}]} f_{p-1}) \, {}^{t_1\ldots t_n} f_p \, \theta_{i_1} \ldots \theta_{i_{p-1}} 
\end{align*}
which agrees with~\eqref{delPsi}. 

To establish compatibility with the products we compute 
\begin{align*}
& \Psi( df_1\ldots df_p \otimes 1) \wedge \Psi( df_{p+1}\ldots df_{p+q} \otimes 1) \\ 
=\, & \Big\{ \sum_{i_1<\ldots< i_p} \Big( \prod_{k=1}^p \partial_{[i_k]} f_k \Big) \theta_{i_1} \ldots \theta_{i_p}\Big\} 
\wedge 
\Big\{ \sum_{i_{p+1}<\ldots< i_{p+q}} \Big( \prod_{k=p+1}^{p+q} \partial_{[i_k]} f_k \Big) \theta_{i_{p+1}} \ldots \theta_{i_{p+q}} \Big\} \\
= \, & \sum_{i_1<\ldots< i_p} \sum_{i_{p+1}<\ldots< i_{p+q}} \Big( \prod_{k=1}^{p+q} \partial_{[i_k]} f_k \Big) \theta_{i_{1}} \ldots \theta_{i_{p+q}} \\
= \, & \sum_{i_1<\ldots< i_{p+q}} \sum_{\sigma\in\operatorname{Sh}(p,q)} (-1)^{|\sigma|} \Big( \prod_{k=1}^{p+q} \partial_{[i_{\sigma(k)}]} f_k \Big) \theta_{i_{1}} \ldots \theta_{i_{p+q}} \\
= \, & \Psi \big((df_1\ldots df_p \otimes 1) \times( df_{p+1}\ldots df_{p+q} \otimes 1) \big)
\end{align*}
where in the third step the anti-commutativity of the~$\theta_i$ allowed us to sum over the longer sequences $i_1<\ldots< i_{p+q}$ by introducing an additional sum over shuffles. 
\end{proof}

We will also need the following property of $\Psi$, where we identify $\Bar = \Omega R_1 \otimes R_2$.

\begin{lemma}\label{lem:PsiToTheRight}
Let $\omega\in\Omega^n R_1$ and $f \in R$ be given. Then $\Psi(\omega \cdot (f \otimes 1))=\Psi(\omega) \cdot (1 \otimes f)$. 
\end{lemma}

\begin{proof}
We may assume that $\omega$ is of the form $da_1\ldots da_n$ for some $a_i\in R$. Then
\begin{align*}
\Psi(\omega f) & = \Psi \Big( \sum_{i=1}^n (-1)^{n-i} da_1\ldots da_{i-1} d(a_i a_{i+1}) da_{i+2}\ldots da_n df + (-1)^n a_1 da_2\ldots da_n df \Big) \\
& = \sum_{i=1}^n (-1)^{n-i} \partial_{[1]} a_1\ldots \partial_{[i-1]} a_{i-1} \partial_{[i]} (a_i a_{i+1}) \partial_{[i+1]} a_{i+2} \ldots \partial_{[n-1]} a_n \partial_{[n]} f \\
& \qquad + (-1)^n a_1 \partial_{[1]} a_2 \ldots \partial_{[n-1]} a_n \partial_{[n]} f \\
& = \sum_{i=1}^n (-1)^{n-i} \partial_{[1]} a_1\ldots \partial_{[i-1]} a_{i-1} \Big(\partial_{[i]} a_i {}	^{t_1\ldots t_i}a_{i+1} + {}^{t_1\ldots t_{i-1}} a_i \partial_{[i]} a_{i+1} \Big) \\
& \qquad \cdot  \partial_{[i+1]} a_{i+2} \ldots \partial_{[n-1]} a_n \partial_{[n]} f + (-1)^n a_1 \partial_{[1]} a_2 \ldots \partial_{[n-1]} a_n \partial_{[n]} f \\
& = \partial_{[1]} a_1 \ldots \partial_{[n]} a_n (1 \otimes f) \\
& = \Psi(\omega) (1 \otimes f) \, ,
\end{align*}
where we used Lemma~\ref{lem:LeibnizForDQO} in the third step. 
\end{proof}


The map~$\Psi$ is a morphism of linear factorisations, where we denote by $(\Bar, d_{\Bar})$ either of the linear factorisations in Lemma \ref{lemma:barisafactorisation} or Lemma \ref{lemma:barisafactorisation2}.

\begin{lemma}\label{PsiHF}
$\Psi: (\Bar, d_\Bar) \lra (\Delta_W, d_{\Delta_W})$ is a morphism in $\HF(\Re,\widetilde W)$. 
\end{lemma}

\begin{proof}
We need to show $d_{\Delta_W} \Psi = \Psi d_\Bar$. 
We know from~\eqref{DeltaW} that $d_{\Delta_W} = \delta_+ + \delta_-$, and $d_\Bar = b' + d\widetilde W \times (-)$ by Lemma~\ref{PhiPsiDG}. 
What remains to be checked is $\delta_+ \Psi = \Psi (d\widetilde W\times (-))$. This can be done: 
\be\label{eq:deltaplusandpsi}
\delta_+ \Psi = \Big( \sum_{i=1}^n \partial_{[i]} W \cdot \theta_i^* \Big) \wedge \Psi(-) = \Psi(dW\otimes 1) \wedge \Psi(-) = \Psi (d\widetilde W\times (-)) 
\ee
where the last equality is due to Lemma~\ref{PhiPsiDG}. 
\end{proof}

Although $(\Bar, d_{\Bar})$ depends only on the pair $(R,W)$, the matrix factorisation $(\Delta_W, d_{\Delta_W})$ depends in addition on the ordering of the ring variables. The compatibility of $\delta_{+}$ and $\Psi$ in \eqref{eq:deltaplusandpsi} depends on the fact that $\Delta_W$ and $\Psi$ are defined using the \textsl{same} ordering.

Finally, using the ordering of the ring variables to order the $\theta$'s, we obtain an $\Re$-linear map 
\be\label{eq:vareps}
\varepsilon: \Delta \lra \Re[n] 
\, , \qquad 
\theta_1\ldots \theta_n \lmt 1
\ee
that is non-zero only on elements in top $\theta$-degree. 


\subsection{Residues}\label{section:residuebackground}

Residues feature prominently throughout the paper, so we briefly recall their definition and basic property; general references for residues are \cite{Lipman84, Lipman} and \cite[Appendix A]{Conrad00}. Given a regular sequence $(f_1,\ldots,f_n)$ in $k[x,y]$ where $x = (x_1,\ldots,x_n)$, the residue is a $k[y]$-linear map that sends a polynomial $g\in k[x,y]$ to an element
\begin{equation}\label{eq:defin_residue_1}
\Res_{k[x,y]/k[y]} \left[ \frac{g \, \underline{\operatorname{d}\! x}}{f_1, \ldots, f_n} \right] 
\in k[y]
\end{equation}
where $\underline{\operatorname{d}\!x}=\operatorname{d}\!x_1\ldots \operatorname{d}\!x_n$. Defining residues in a way which is independent of a choice of coordinates is a delicate business, and in the case of rings the most elegant approach we are aware of is by Lipman \cite{Lipman} who uses the canonical pairing between Hochschild homology and cohomology. This amounts to defining the residue as the \textsl{trace} of a certain carefully constructed $k[y]$-linear operator on $k[x,y]/(f_1,\ldots,f_n)$, as explained in terms of connections in \cite[Section $8$]{dm1102.2957}.

While it is not straightforward to define residues abstractly, they are easy to compute in practice. The residue is zero if $g$ belongs to the ideal $I$ generated by the $f_1,\ldots,f_n$, and the expression inside the brackets in \eqref{eq:defin_residue_1} behaves like a fraction in the sense that
\begin{equation}\label{eq:defin_residue_2}
\Res_{k[x,y]/k[y]} \left[ \frac{g \, f_i \, \underline{\operatorname{d}\! x}}{f_1, \ldots, f_i^2, \ldots f_n} \right] = \Res_{k[x,y]/k[y]} \left[ \frac{g \, \underline{\operatorname{d}\! x}}{f_1, \ldots, f_i, \ldots, f_n} \right]\,.
\end{equation}
More precisely, the expression inside the brackets is a \textsl{generalised fraction}, which is terminology for elements of the local cohomology module $H^n_{I}( \Omega^n_{k[x,y]/k[y]} )$. Nontrivial residues are computed using the so-called \textsl{transformation rule}:
$$
\Res_{k[x,y]/k[y]} \left[ \frac{g \, \underline{\operatorname{d}\! x}}{f_1, \ldots, f_n} \right] 
= 
\Res_{k[x,y]/k[y]} \left[ \frac{\det(C) g \, \underline{\operatorname{d}\! x}}{f'_1, \ldots, f'_n} \right] 
, \qquad
f'_i = \sum_{j=1}^n C_{ij} f_j 
\, , \qquad
C_{ij} \in k[x,y] \, .
$$
Such an expression for the $f_i$ will always exist with denominators $f_i' = x_i^{a_i}$ for sufficiently large exponents $a_i$. This transforms an arbitrary residue into one whose denominator consists of powers of the variables, and this, together with \eqref{eq:defin_residue_2} and the defining property $\Res_{k[x,y]/k[y]} \left[ \underline{\operatorname{d}\! x}/(x_1^{a_1} ,\ldots, x_n^{a_n}) \right] = \delta_{a_1,1}\ldots \delta_{a_n,1}$ is enough to compute any residue. Note that the order of the elements of the regular sequence in the denominator plays a role; changing that order produces a permutation sign.

In the next result we use the divided difference operators of \eqref{diffquotop}.

\begin{proposition}\label{prop:determinantresidue} 
The element $\delta = \det \big( (\partial^{x,x'}_{[i]} f_j)_{i,j} \big)$ in $k[x,x']$ has the property that for $g \in k[x]$
\be\label{eq:integrateagainstd}
\Res_{k[x,x']/k[x']} \left[ \frac{g \, \delta \,  \underline{\operatorname{d}\! x}}{f_1, \ldots, f_n} \right] = g(x')
\ee
as an element of the algebra $k[x']/(f_1(x'),\ldots,f_n(x'))$.
\end{proposition}
\begin{proof}
The result is proven in \cite[Proposition $4.1.2$]{pv1002.2116} for $k$ a field, but the same proof works in the present generality. 
\end{proof}

We will also use the transitivity property of residues:
 
\begin{proposition}\label{prop:transitivity} Given regular sequences $(f_1,\ldots,f_n)$ in $k[x,y,z]$ and $(g_1,\ldots,g_m)$ in $k[y,z]$, and a polynomial $h \in k[x,y,z]$, the iterated residue
\[
\Res_{k[y,z]/k[z]} \begin{bmatrix} \Res_{k[x,y,z]/k[y,z]} \begin{bmatrix} h \, \underline{\operatorname{d}\! x} \\ f_1,\ldots, f_n \end{bmatrix} \cdot \underline{\operatorname{d}\! y} \\ g_1, \ldots, g_m \end{bmatrix} = \Res_{k[x,y,z]/k[z]} \begin{bmatrix} h \, \underline{\operatorname{d}\! x} \wedge \underline{\operatorname{d}\! y} \\ f_1,\ldots,f_n,g_1,\ldots,g_m \end{bmatrix}\,.
\]
\end{proposition}
\begin{proof}
See \cite{lipman92} and \cite[Appendix A, p.244]{Conrad00}.
\end{proof}

\subsection{Perturbation}\label{section:perturbation_lemma}

A crucial role will be played by the homological perturbation lemma, which we will use to promote homotopy equivalences of complexes (arising from the bar and Koszul resolutions of the diagonal) to homotopy equivalences of associated matrix factorisations. More importantly, the perturbation lemma will provide explicit homotopy inverses in terms of Atiyah classes.

Let $R$ be a ring and $W \in R$. An $R$-linear \textsl{deformation retract datum} is a diagram
\begin{equation}\label{eq:perturbeddiagr1}
\xymatrix@C+2pc{
(X,d_X) \ar@<-1ex>[r]_-{\sigma} & (Y,d_Y) \ar@<-1ex>[l]_-{\pi}
}%
\!\!\!\xymatrix{%
{}\ar@(ur,dr)[]^{h}
}%
\end{equation}
in which $(X,d_X)$ and $(Y,d_Y)$ are linear factorisations of $W$, $\pi, \sigma$ are morphisms of linear factorisations and $h: Y \lto Y$ is a degree one $R$-linear map such that
$$
\pi \sigma = 1 \, , \qquad
\sigma \pi = 1 + d_Yh + hd_Y \, .
$$
A degree one morphism $\delta: Y \lto Y$ is a \textsl{small perturbation} of the deformation retract datum if $1_Y - \delta h$ is an isomorphism of $R$-modules. In this case we define
\[
\tau = (1- \delta h)^{-1} \delta
\]
and consider the new ``perturbed'' diagram
\begin{equation}\label{eq:perturbeddiagr}
\xymatrix@C+2pc{
(X,d_{X,\infty}) \ar@<-1ex>[r]_-{\sigma_\infty} & (Y,d_Y+\delta) \ar@<-1ex>[l]_-{\pi_\infty}
}%
\!\!\!\xymatrix{%
{}\ar@(ur,dr)[]^{h_\infty}
}%
\end{equation}
where
\begin{align*}
\sigma_\infty &= \sigma + h\tau\sigma \,, & h_\infty &= h + h \tau h\,,\\
\pi_\infty &= \pi + \pi \tau h \, , & d_{X,\infty} &= d_X + \pi \tau \sigma\,.
\end{align*}

\begin{proposition}\label{prop:pertlemma} Suppose that $h \sigma = 0, \pi h = 0$ and $h^2 = 0$. If $\delta$ is a small perturbation of (\ref{eq:perturbeddiagr1}) such that $(d_Y + \delta)^2 = W' \cdot 1_X$ for some $W' \in R$ then (\ref{eq:perturbeddiagr}) is a deformation retract datum of linear factorisations of $W'$ over $R$.
\end{proposition}
\begin{proof}
This follows from the standard results in~\cite{c0403266}. In more detail: the identities (2.3) and (2.4) of \cite[Lemma $2.5$]{c0403266} still hold, but the identity (2.5) does not (note that $\tau, \pi, \sigma, d_Y$ here are $A, p, i, b$ there). The proof in loc. cit. of (2.5) instead shows that
\begin{align*}
\tau \sigma \pi \tau + \tau d_Y + d_Y \tau &= (1 - \delta h)^{-1}\big[ \delta \delta + \delta d_Y + d_Y \delta \big]( 1 - h \delta )^{-1}\\
&= (1 + \tau h) \big[ (d_Y+\delta)^2 - d_Y^2 \big](1 + h \tau)\\
&= (W' - W) (1 + \tau h)(1 + h \tau)\\
&= (W' - W)( 1 + h \tau + \tau h)
\end{align*}
where in the last line we use $h^2 = 0$. Using this identity together with $h \sigma = 0$ and $\pi h = 0$ we get
\begin{equation}\label{eq:pert_form_1}
\pi(\tau d_Y + d_Y \tau) \sigma = (W' - W)\pi(1 + h \tau + \tau h) \sigma - \pi(\tau \sigma \pi \tau ) \sigma = W' - W - \pi(\tau \sigma \pi \tau)\sigma\,,
\end{equation}
\begin{equation}\label{eq:pert_form_2}
h(\tau d_Y + d_Y \tau)\sigma = (W' - W)h(1+h\tau + \tau h)\sigma - h(\tau \sigma \pi \tau)\sigma = - h(\tau \sigma \pi \tau) \sigma\,.
\end{equation}
Similarly,
\begin{equation}\label{eq:pert_form_3}
\pi(\tau d_Y + d_Y \tau)h = -\pi(\tau \sigma \pi \tau)h, \qquad h(\tau d_Y+ d_Y \tau)h = -h(\tau \sigma \pi \tau)h\,.
\end{equation}
Using these identities \eqref{eq:pert_form_1},\eqref{eq:pert_form_2},\eqref{eq:pert_form_3} the proof of \cite[Theorem 2.3]{c0403266} goes through.
\end{proof}

In the cases of interest to us the sum $\sum_{m \ge 0} (\delta h)^m$ converges, so that $\tau = \sum_{m \ge 0} (\delta h)^m \delta$ and
\[
\sigma_\infty = \sigma + \sum_{m \ge 0} h(\delta h)^m \delta \sigma = \sum_{m \ge 0} (h \delta)^m \sigma\,.
\]

\subsection{Canonical morphisms}\label{section:canonicalmaps}

In this section~$R$ is a ring and~$X$ and~$Y$ denote linear factorisations of~$W$ and~$V$ over $R$, respectively. The $R$-module $\Hom_R(X,Y)$ with its natural $\mathbb{Z}_2$-grading has a differential $\alpha \mapsto d_Y \circ \alpha - (-1)^{|\alpha|} \alpha \circ d_X$ which makes it into a linear factorisation of $V - W$, and there is a natural morphism of linear factorisations of $V - W$, 
\be
\xi: X^{\vee} \otimes_R Y \lto \Hom_R(X,Y)\, , \qquad
\xi( \nu \otimes y )(x) = (-1)^{|\nu||y|} \nu(x) \cdot y \, ,
\ee
which is an isomorphism if~$X$ is a finitely generated projective $R$-module.

There is a natural isomorphism of linear factorisations of $W + V$, 
\be
\textup{swap}: X \otimes_R Y \lto Y \otimes_R X\,, \qquad
x \otimes y \lmt (-1)^{|x||y|} y \otimes x\,.
\ee
Precomposing~$\xi$ with this swap isomorphism we have a canonical morphism
\be\label{eq:iso_tensorhom2}
Y \otimes_R X^{\vee} \lto \Hom_R(X,Y)\,, \qquad
y \otimes \nu \lmt \big\{ x \lmt \nu(x) \cdot y \big\}\,.
\ee
Since it is unlikely to cause confusion, we also denote this map by $\xi$.

There are natural isomorphisms of linear factorisations of $W + V$, 
\begin{align}
X[i] \otimes_R Y &\lto (X \otimes_R Y)[i]\,, \qquad
x \otimes y \lmt x \otimes y\,,\\
X \otimes_R Y[i] &\lto (X \otimes _R Y)[i]\,,\qquad 
x \otimes y \lmt (-1)^{i|x|} x \otimes y\,.
\end{align}

\section{Atiyah classes}\label{section:atiyahclasses}

As mentioned in the Introduction, the structure of the bicategory $\LG$ can be understood in terms of Atiyah classes. In this section we develop the theory of these operators, beginning with noncommutative form-valued connections. The reader might prefer to begin reading Section \ref{section:pertandhtpy} to see how Atiyah classes naturally arise via perturbation, and then return here.

Let $A$ be a unital associative $k$-algebra which is not necessarily commutative, and $\Omega A = \Omega_k A$ the dg-algebra of noncommutative forms defined in Section \ref{subsec:Bar}. Following \cite[Section $8$]{cuntzquillen} a $k$-linear \textsl{connection} on a $\nZ _2$-graded right $A$-module $X$ is a $k$-linear map (we write $\otimes$ for $\otimes_A$)
\[
\nabla: X \lto X \otimes \Omega^1 A\,,
\]
which sends $X^i$ into $X^i \otimes \Omega^1 A$ and satisfies the graded Leibniz rule
\[
\nabla( x a ) = \nabla( x ) a + (-1)^{|x|} x \otimes da
\]
for $x \in X$, $a \in A$. The differential on $\Omega A$ gives a map $A \lto \Omega^1 A$ which is a connection, and so if~$X$ is a free $A$-module (or even projective) there exists a connection $\nabla$ on $X$. A connection extends uniquely to a degree one $k$-linear operator on $X \otimes \Omega A$, still denoted $\nabla$, with the property that for homogeneous $\xi \in X \otimes \Omega A$ and $\omega \in \Omega A$
\[
\nabla( \xi \omega ) = \nabla( \xi ) \omega + (-1)^{|\xi|} \xi d \omega\,.
\]
Using the Leibniz rule it is easy to check that for any odd $A$-linear operator $D$ on $X$ the operator $[\nabla, D] = \nabla D + D \nabla$ on $X \otimes \Omega A$ is right $A$-linear. Here we use the notation of graded commutators: for homogeneous operators $F, G$, 
\[
[F,G] = FG - (-1)^{|F||G|} GF \,.
\]

\begin{definition} The operator $\At = [\nabla, D]$ on $X \otimes \Omega A$ is called the (associative) \textsl{Atiyah class} of the pair $(X,D)$ with respect to the ring morphism $k \lto A$ and connection $\nabla$.
\end{definition}

The terminology of characteristic classes is appropriate when $D^2 = W \cdot 1_X$ for some $W \in k$, since then $[\nabla, D]$ is a closed map, i.\,e.~it gives a class in the cohomology of $\Hom_A(X, X \otimes \Omega^1 A)$ with respect to the differential defined by taking the commutator with $D$. By a standard argument this cohomology class is independent of the choice of connection on $X$. 

In general the Atiyah class depends on a choice of connection, so we should perhaps refer to it as an Atiyah \textsl{operator}, or as a component of the curvature of a superconnection as Remark \ref{remark:curvature_superconnection} below explains. However for us the purpose of the operator $\At$ is to give closed formulas for chain maps that are already known to be canonically defined up to homotopy, so we will ignore this issue.

\begin{remark}
We call these \textsl{associative} Atiyah classes to distinguish them from the standard Atiyah classes defined using commutative differential forms, see \cite{atiyahconn,illusie,buchweitz03,markarian,buchweitz08}. This kind of Atiyah class does not appear here, so we will usually drop the qualifier ``associative''.
\end{remark}

Let $X$ be a free $\nZ _2$-graded right $A$-module with homogeneous basis $\{ e_i \}_i$. Then
\[
d: X \lto X \otimes \Omega^1 A \, , \qquad d( e_i a ) = (-1)^{|e_i|} e_i \otimes da
\]
is a $k$-linear connection. It extends to a $k$-linear operator
\be
d: X \otimes \Omega A \lto X \otimes \Omega A \,, \qquad d( e_i \otimes \omega ) = (-1)^{|e_i|} e_i \otimes d \omega\,, \label{eq:extendd}\\
\ee
and powers of the Atiyah class $\At = [d, D]$ are given by:

\begin{lemma} For $l \ge 1$ and $\omega \in \Omega A$ we have
\be
\At^l( e_i \otimes \omega ) = (-1)^{l|e_i| + \binom{l+1}{2}} \sum_{j_1,\ldots,j_l} e_{j_l} \otimes d( D_{j_l j_{l-1}} ) \ldots d(D_{j_2 j_1})d( D_{j_1 i} ) \omega \label{eq:powerofatiyah}
\ee
where $D_{ij}$ are the entries of the matrix representing~$D$ in the basis $\{ e_i \}_i$.
\end{lemma}
\begin{proof}
The proof is by induction on $l$. We have
\begin{align}
\At^{l+1}( e_i \otimes \omega ) &= \At \At^l( e_i \otimes \omega )\\
&= (-1)^{l|e_i| + \binom{l+1}{2}} \sum_{j_1,\ldots,j_l} \At\left( e_{j_l} \otimes d( D_{j_l j_{l-1}} ) \ldots d(D_{j_2 j_1})d( D_{j_1 i} ) \omega \right)\,.\label{eq:mad_signs_everywhere}
\end{align}
For a fixed tuple of indices $j_1,\ldots,j_l$ let us write $\kappa$ for $d( D_{j_l j_{l-1}} ) \ldots d(D_{j_2 j_1})d( D_{j_1 i} )$. Then
\begin{align*}
Dd( e_{j_l} \otimes \kappa \omega ) &= (-1)^{|e_{j_l}|} D( e_{j_l} ) \otimes d( \kappa \omega ) =  \sum_{j_{l+1}} (-1)^{|e_{j_l}|} e_{j_{l+1}} \otimes D_{j_{l+1}j_l} d( \kappa \omega )\,,\\
d D\left( e_{j_l} \otimes \kappa \omega \right) &= \sum_{j_{l+1}} d\left( e_{j_{l+1}} D_{j_{l+1}j_l}  \otimes \kappa\omega \right) = \sum_{j_{l+1}} (-1)^{|e_{j_{l}}| + 1} e_{j_{l+1}} \otimes d( D_{j_{l+1}j_l} \kappa\omega )\,.
\end{align*}
Adding these using the Leibniz rule we find that
\[
\At( e_{j_l} \otimes \kappa \omega ) = \sum_{j_{l+1}} (-1)^{|e_{j_l}| + 1} e_{j_{l+1}} \otimes d( D_{j_{l+1}j_l} ) \kappa\omega\,.
\]
Since $|e_{j_l}| = |e_i| + l$ the overall sign on the right hand side of \eqref{eq:mad_signs_everywhere} is
\[
l|e_i| + \binom{l+1}{2} + |e_i| + l + 1 = (l+1)|e_i| + \binom{l+2}{2}\,,
\]
which completes the inductive step.
\end{proof}

\begin{remark}\label{remark:leftatiyah} A connection on a $\nZ _2$-graded left $A$-module $X$ is a $k$-linear map $\nabla: X \lto \Omega^1 A \otimes X$ satisfying the Leibniz rule, and the \textsl{left Atiyah class} of a pair $(X,D)$ is defined to be the commutator $\lAt = [\nabla, D]$. Everything we say has a natural analogue for left Atiyah classes.
\end{remark}

\begin{example}\label{example:leftatiyahexplicit} Let $X$ be a free $\nZ _2$-graded left $A$-module with homogeneous basis $\{ e_i \}_i$. Then
\[
d: X \lto \Omega^1 A \otimes X \, , \qquad d( a e_i ) = da \otimes e_i
\]
is a $k$-linear connection. Powers of the left Atiyah class $\lAt = [d, D]$ are operators on $\Omega A \otimes X$,
\be\label{eq:leftatiyahexp}
\lAt^l( \omega \otimes e_i ) = \sum_{j_1,\ldots,j_l} \omega d( D_{j_1 i} ) d(D_{j_2 j_1}) \ldots d(D_{j_l j_{l-1}}) \otimes e_{j_l}\,.
\ee
\end{example}

Let $X$ and $X'$ be $\nZ _2$-graded right $A$-modules equipped with odd $A$-linear operators which for convenience we denote by $D$ in both cases, and suppose $X$ and $X'$ admit $k$-linear connections, both of which we denote $\nabla$.

\begin{lemma}\label{lemma:atiyahnat} Given an $A$-linear homogeneous map $\varphi: X \lto X'$ consider the diagram
\[
\xymatrix@C+1pc@R+1pc{
X \ar[r]^-{\varphi}\ar[d]_-{[\nabla, D]}\ar@{.>}[dr]^-{g} & X' \ar[d]^-{[\nabla, D]}\\
X \otimes \Omega^1 A \ar[r]_{\varphi \otimes 1} & X' \otimes \Omega^1 A
}
\]
where $g = [\varphi, \nabla]$. This is a right $A$-linear map, and
\[
[\varphi, [\nabla, D]] = (-1)^{|\varphi|} [ D, [\varphi, \nabla]] - [ \nabla, [D, \varphi]]\,.
\]
\end{lemma}
\begin{proof}
Follows from the graded Jacobi identity for commutators.
\end{proof}

In particular if $D^2 = W$ on both $X$ and $X'$ for some $W \in k$, so $[\nabla, D]$ is a closed map, then for any morphism $\varphi$ we see that there is a right $A$-linear homotopy $\varphi \circ [ \nabla, D] \simeq [\nabla, D] \circ \varphi$.

\begin{remark}\label{remark:curvature_superconnection} According to Quillen \cite{superquillen} a \textsl{superconnection} $\Theta$ on a $\nZ _2$-graded right $A$-module $X$ is a $k$-linear degree one operator~$\Theta$ on $X \otimes \Omega A$ satisfying the graded Leibniz rule
\[
\Theta( \xi \omega ) = \Theta( \xi ) \omega + (-1)^{|\xi|} \xi d \omega
\]
for all homogeneous $\xi \in X \otimes \Omega A$ and $\omega \in \Omega A$. Here we use grading $|x \otimes \omega| = |x| + |\omega|$ mod $2$. If $\nabla$ is a connection on $X$ then we extend $\nabla$ to an operator on $X \otimes \Omega A$ and it is easily checked that for any odd $A$-linear operator $D$ on $X$ the operator $\Theta = \nabla + D$ is a superconnection. The \textsl{curvature} of this superconnection is $\Theta^2 = \nabla^2 + D^2 + [\nabla, D]$. In this paper our connections $\nabla$ are flat, i.\,e.~$\nabla^2 = 0$, and so if $(X,D)$ is a complex the curvature~$\Theta^2$ is precisely the Atiyah class $[\nabla, D]$.
\end{remark}

We now explain the special cases that will arise in later sections. In the rest of the present section, all rings are commutative. In all the examples there will be a ring $A$ and a pair $(X,D)$ consisting of a free $\nZ _2$-graded $A$-module $X$ with homogeneous basis $\{ e_i \}_i$ and an odd $A$-linear operator $D$ on~$X$. Let~$R$ and~$S$ denote $k$-algebras.

\begin{example}\label{remark:linearatotherR} With $A = S \otimes_k R$ the $S$-linear map defined by
\[
d: X \otimes_R \Omega R \lto X \otimes_R \Omega R \, , \qquad d( e_i a ) = (-1)^{|e_i|} e_i \otimes d a
\]
for $a \in R$ is the extension of a $k$-linear connection $X \lto X \otimes_R \Omega R$ on $X$ as an $R$-module, and so the Atiyah class of $X$ relative to $k \lto R$ is the operator
\[
\At_R(X) := [d, D]: X \otimes_R \Omega R \lto X \otimes_R \Omega R\,.
\]
The subscript indicates the fact that the connection involved differentiates in the ``$R$-directions''. The explicit formula is given by \eqref{eq:powerofatiyah}. Under the isomorphism of Lemma \ref{lemma:coeffomeage}, 
\[
X \otimes_A \Omega_S A \cong X \otimes_A (S \otimes_k \Omega R) \cong X \otimes_R \Omega R \, , 
\]
the operator $\At_R(X)$ corresponds to the Atiyah class of $X$ relative to the ring morphism $S \lto A$. If we write $\Re = R_1 \otimes_k R_2$ with $R_i = R$ and $A = S \otimes_k R_1$ then the operator $\At_{R_1}(X) \otimes 1_{R_2}$ on
\[
(X \otimes_{R_1} \Omega R_1) \otimes_k R_2 \cong X \otimes_{R_1} \Omega_{R_2}(\Re) \cong (X \otimes_{R_1} \Re) \otimes_{\Re} \Omega_{R_2}(\Re)
\]
is the Atiyah class of the extension of scalars $X \otimes_{R_1} \Re$ relative to $R_2 \lto \Re$.
\end{example}

In the next example we use the alternative realisation of noncommutative forms given in Section~\ref{subsec:Bar}, as this will be convenient later.

\begin{example}\label{remark:linearatotherS} Again with $A = S \otimes_k R$ the $R$-linear map defined by
\[
s: \widetilde{\Omega} S \otimes_S X \lto \widetilde{\Omega} S \otimes_S X \, , \qquad d( a e_i ) = sa \otimes e_i
\]
for $a \in S$ is the extension of a $k$-linear connection $X \lto \widetilde{\Omega} S \otimes_S X$ and so the left Atiyah class of~$X$ relative to $k \lto S$ is the operator
\[
\lAt_S(X) := [s,D]: \widetilde{\Omega} S \otimes_S X \lto \widetilde{\Omega} S \otimes_S X\,.
\]
The explicit formula is \eqref{eq:leftatiyahexp} with $s$ in place of $d$. Under the isomorphism $\widetilde{\Omega}_R A \otimes_A X \cong \widetilde{\Omega} S \otimes_S X$ this operator corresponds to the left Atiyah class of $X$ relative to the ring morphism $S \lto A$.
\end{example}




\begin{example}\label{remark:linearatother} Set $\Re = R_1 \otimes_k R_2$ with $R_i = R$ and $A = S \otimes_k \Re$. The Atiyah class of $X$ relative to $R_2 \lto \Re$ is the operator
\[
\At_{R_1}(X) := [d, D]: X \otimes_{\Re} \Omega_{R_2}(\Re) \lto X \otimes_{\Re} \Omega_{R_2}(\Re)
\]
where $d$ is the $(S \otimes_k R_2)$-linear operator $d( e_i \otimes \omega ) = (-1)^{|e_i|} e_i \otimes d\omega$. The operator $\At_{R_1}(X)$ can also be identified with the Atiyah class of $X$ relative to the ring morphism $S \otimes_k R_2 \lto A$. Similarly, if we define the $(S \otimes_k R_1)$-linear operator
\[
\lAt_{R_2}(X) := [s,D]: \widetilde{\Omega}_{R_1}(\Re) \otimes_{\Re} X \lto \widetilde{\Omega}_{R_1}(\Re) \otimes_{\Re} X 
\]
where $s( \omega \otimes e_i ) = s\omega \otimes e_i$, then the canonical isomorphism $\widetilde{\Omega}_{S \otimes_k R_1} A \otimes_{S \otimes_k \Re} X \cong \widetilde{\Omega}_{R_1}(\Re) \otimes_{\Re} X$ identifies $\lAt_{R_2}(X)$ with the left Atiyah class of $X$ relative to $S \otimes_k R_1 \lto A$.
\end{example}

\begin{example}\label{example:linearother3} With $A = \Re$ the left Atiyah class of $X$ relative to $R_2 \lto \Re$ is the operator
\[
\lAt_{R_1}(X) := \lAt: \Omega_{R_2}(\Re) \otimes_{\Re} X \lto \Omega_{R_2}(\Re) \otimes_{\Re} X
\]
given explicitly by Example \ref{example:leftatiyahexplicit}.
\end{example}

We will need ``cosmetic'' variants of these Atiyah classes where we switch the order of components in tensors. This is natural in situations where multiple Atiyah classes are combined; see Lemma \ref{lemma:atshufat} below. We take $A = \Re$ with all other notation as above, and make use of the $k$-algebra morphism $\reprod: A \otimes_k A  \lto A \otimes_k A$ defined by $\reprod ( a \otimes a' ) = a a' \otimes 1$. Given an $A$-bimodule $M$ the bimodule $\gamma_*M$ is the restriction of scalars along the ring map $\gamma$, that is, $\gamma_* M$ has the same underlying $k$-module as $M$ and the $A$-bimodule action $\star$ given by $a \star m \star a' = aa' m$. 

\begin{definition}\label{eq:leftarrowAt} We define $\Atlarrow_{R_1}(X)$ to be the operator making the diagram
\[
\xymatrix@C+2pc@R+1pc{
X \otimes \Omega_{R_1} A \ar[d]_{\At_R(X)} \ar[r]^-{\tau}_{\cong} & \reprod_* \Omega_{R_1} A \otimes X \ar[d]^{\Atlarrow_{R_1}(X)}\\
X \otimes \Omega_{R_1} A \ar[r]_-{\tau}^{\cong} & \reprod_* \Omega_{R_1} A \otimes X
}
\]
commute, where $\tau(x \otimes \omega) = (-1)^{|x||\omega|} \omega \otimes x$ is the graded twist map. Explicitly,
\be
\Atlarrow_{R_1}(X)^l( \omega \otimes e_i ) = (-1)^{l|\omega| + \binom{l}{2}} \sum_{j_1,\ldots,j_l} d(D_{j_l j_{l-1}}) \ldots d(D_{j_2 j_1}) d(D_{j_1 i}) \omega \otimes e_{j_l}\,. \label{eq:atlarrow_formula}
\ee
The arrow indicates that the operator $\Atlarrow$ puts forms on the left, whereas $\At$ puts them on the right. By pre- and post-composing with $\tau, \tau^{-1}$ we similarly define operators $\Atlarrow_{R_1}, \overset{\rightarrow}{\lAt}_{R_2}, \overset{\rightarrow}{\lAt}_{R_2}$.
\end{definition}


An important property of the ordinary Atiyah class defined using commutative forms is that the class of $X \otimes Y$ can be expressed in terms of the classes of $X$ and $Y$. The analogue for associative Atiyah classes involves the shuffle product. Hence we restrict to the special case of $A = \Re$ and let $(X,d_X)$ and $(Y,d_Y)$ be $\nZ _2$-graded free $\Re $-modules equipped with odd $\Re $-linear operators and respective homogeneous bases $\{ e_i \}_{i}$ and $\{ \by_j \}_{j}$. Then $X \otimes Y$ is given the differential $d_{X \otimes Y} = d_X \otimes 1 + 1 \otimes d_Y$ and the basis $\{ e_i \otimes \by_j \}_{i,j}$.

There are three natural operators
\begin{align*}
\At_{R_2}(X) &\in \End_k( X \otimes \Omega_{R_2}(\Re) ) \, ,\\
\Atlarrow_{R_2}(Y) &\in \End_k( \reprod_*\Omega_{R_2}(\Re) \otimes Y ) \, ,\\
\At_{R_2}(X \otimes Y) &\in \End_k( X \otimes Y \otimes \Omega_{R_2}(\Re) )\,.
\end{align*}
Via the identification of $\Omega_{R_2}(\Re)$ with $\Bar$, the shuffle product defines an operation
\begin{equation}\label{eq:shuffle_opat}
\xymatrix@C+1pc{
X \otimes \reprod_*\Omega_{R_2}(\Re) \otimes \reprod_*\Omega_{R_2}(\Re) \otimes Y \ar[r]^-{1 \otimes \times \otimes 1} & X \otimes \reprod_*\Omega_{R_2}(\Re) \otimes Y \ar[r]^-{\cong} & X \otimes Y \otimes \Omega_{R_2}(\Re)
}
\end{equation}
where the last map is the graded twist on the second two components. 

\begin{lemma}\label{lemma:atshufat} 
For $l \ge 0$ we have
\begin{align}
\At_{R_2}(X \otimes Y)^l( e_i \otimes \by_j ) &= \sum_{p+q = l} \At_{R_2}(X)^p(e_i) \times \Atlarrow_{R_2}(Y)^q(\by_j)\label{eq:atshufat}
\end{align}
where on the right-hand side we use the operation \eqref{eq:shuffle_opat}. 
\end{lemma}
\begin{proof}
We have
\[
\At_{R_2}(X \otimes Y)(e_i \otimes \by_j) = \sum_{i_1} (-1)^{|e_{i_1}| + |\by_j|} e_{i_1} \otimes \by_j \otimes d( d_{X,i_1 i}) + \sum_{j_1} (-1)^{|\by_{j_1}|} e_i \otimes \by_{j_1} \otimes d( d_{Y, j_1 j} )\,.
\]
Iterating, we see that $\At_{R_2}(X \otimes Y)^l(e_i \otimes \by_j)$ is a sum over all indices of terms
\be\label{eq:proofofatiyahshuff}
e_{i_{l-p}} \otimes \by_{j_p} \otimes \sigma_\bullet( \underbrace{d(d_{X,i_{l-p}i_{l-p-1}}), \ldots, d(d_{X,i_1i})}_{l-p}, \underbrace{d(d_{Y,j_pj_{p-1}}), \ldots, d(d_{Y,j_1j})}_p )
\ee
where $\sigma$ is an $(l-p,p)$ shuffle and the sign that is attached to such a term is $(-1)^q$ where
\[
q = |e_{i_{l-p}}| + \ldots + |e_{i_1}| + |\by_{j_p}| + \ldots + |\by_{j_1}| + (l-p)|\by_{j_p}| + |\sigma|\,.
\]
But considering \eqref{eq:powerofatiyah} it is straightforward to check that the right-hand side of~\eqref{eq:atshufat} is a sum over the same collection of terms, and the signs match, so that \eqref{eq:atshufat} holds. 
\end{proof}

The analogous statement for $\At_{R_1}$ is proved in the same way.

\section{Perturbation and inverting unit actions}\label{section:pertandhtpy}

The fundamental technical results in this paper are constructions, using the perturbation lemma, of explicit homotopy inverses to morphisms involving the stabilised diagonal. Generically the results are geometric series in the Atiyah classes of the previous section. To give a specific example, recall that we have specified for any $1$-morphism $X \in \LG(W,V)$ a pair of natural isomorphisms
\begin{equation}\label{eq:pertandhtpy1}
\lambda: \Delta_V \otimes X \lto X \, , \qquad \rho: X \otimes \Delta_W \lto X
\end{equation}
called the \textsl{unit actions}, see \eqref{lambdarho}. Representing chain maps for the inverses of these morphisms are necessary for computing with diagrams in $\LG$, and in particular are needed for proving the Zorro moves in Section \ref{sec:Zorro}, but finding such representatives is nontrivial. 

Instead of inverting $\lambda$ and $\rho$ directly, which is difficult, we proceed by identifying $\rho$ as the shadow of a similar canonical map $\pi$ on the bar model for the diagonal via a commutative diagram
\begin{equation}\label{eq:pertdia1}
\xymatrix{
X \,\widehat{\otimes}\, \Bar \ar[dr]_-{\pi}\ar[rr]^{1 \otimes \Psi} & & X \otimes \Delta_W \ar[dl]^-{\rho}\\
& X
}
\end{equation}
where $\Psi$ is the canonical map given in (\ref{intro_psi}). Roughly speaking the inverse of $\pi$ is the geometric series in powers of the Atiyah class of $X$ and by postcomposing with $\Psi$ we obtain the desired inverse to $\rho$. A similar argument works for inverting $\lambda$. A completion of $X \otimes \Bar$ is used in order to guarantee that the geometric series converges. We return to this example in Section \ref{section:rhoandlambdainverse} below.

Another important example is the following: given an object $(k[x],W) \in \LG$ and a $1$-morphism $X \in \LG(W,W)$ consider the problem of lifting a $k[x]$-bilinear morphism of linear factorisations $X \lto k[x]$ to a morphism $X \lto \Delta_W$ along the stabilisation map $\pi_{\Delta}: \Delta_W \lto k[x]$, 
\[
\xymatrix{
X \ar[dr] \ar@{.>}[rr] & & \Delta_W \ar[dl]^-{\pi_{\Delta}} \\
& k[x]\,.
}
\]
The solution to this lifting problem is given in Section \ref{section:liftingproblem}, and will be used to give an explicit formula for the evaluation maps in Section \ref{sec:derivcoeval}.

In order to address these examples and others at the same time, we work in the following general setting: $k$ is a ring, $R, S$ are $k$-algebras and $\Re  = R_1 \otimes_k R_2$ where $R_i = R$ for $i \in \{1,2\}$. We assume that a matrix factorisation $X \in \HMF(S \otimes_k \Re, U)$ is given with homogeneous basis $\{ e_i \}_{i}$ as an $(S \otimes_k \Re)$-module. Let~$D$ denote the differential on~$X$ and write $\otimes$ for the tensor product $\otimes_{\Re}$.

We set $\Bar = \Omega_{R_2}(\Re)$ and consider the module
\[
X \,\widehat{\otimes}\, \Bar := \prod_{l \ge 0} X \otimes \Bar_l
\]
with the $\nZ _2$-grading
\[
\big( X \,\widehat{\otimes}\, \Bar \big)^i = \left( \prod_{l \in \nZ } X^i \otimes \Bar_{2l} \right) \oplus \left( \prod_{l \in \nZ } X^{i+1} \otimes \Bar_{2l+1} \right) . 
\]
It is sometimes helpful to view $X \,\widehat{\otimes}\, \Bar$ as the inverse limit of the system
\begin{equation}\label{eq:inverse_system}
\cdots \lto X \otimes \Bar/\Bar_{\ge 2} \lto X \otimes \Bar/\Bar_{\ge 1}
\end{equation}
where $\Bar_{\ge l} = \bigoplus_{i \ge l} \Bar_i \subseteq \Bar$ and the maps are the obvious quotients $\Bar/\Bar_{\ge l+1} \lto \Bar/\Bar_{\ge l}$. Note that by presenting $\Bar$ as $\Omega_{R_2}(\Re)$ we are also fixing a left and right $\Re$-action on $\Bar$, as in \eqref{eq:leftreactionb}. The left action is consumed by the tensor product with $X$, so that overall $X \, \widehat{\otimes}\, \Bar$ is an $S$-$\Re$-bimodule using the right action of $\Re$ on $\Bar$.

Let $W \in R$ be arbitrary and set $\widetilde{W} = W \otimes 1 - 1 \otimes W \in \Re$. Our first observation is that $X \,\widehat{\otimes}\, \Bar$ can be equipped as a linear factorisation of $U + \widetilde{W}$. One checks that there are well-defined $S$-$\Re$-bilinear operators on $X \,\widehat{\otimes}\, \Bar$ given by
\begin{align}
D(x_0, x_1, \ldots) &= (D(x_0), D(x_1), \ldots) \, , \nonumber\\
b'(x_0, x_1, \ldots) &= (b'(x_1),b'(x_2),\ldots) \, , \nonumber\\
d(x_0,x_1,\ldots) &= (0, d(x_0), d(x_1),\ldots) \, , \nonumber\\
s(x_0,x_1,\ldots) &= (0, s(x_0), s(x_1), \ldots) \, , \nonumber\\
d \widetilde{W} \times (x_0,x_1,\ldots) &= (0, d \widetilde{W} \times x_0, d \widetilde{W} \times x_1, \ldots) \label{eq:listofcompletions}
\end{align}
where $d$ and $s$ are extended to $X \otimes \Bar$ as in \eqref{eq:extendd} using the chosen basis. 

\begin{lemma} $( X \, \widehat{\otimes}\, \Bar, D + b' + d\widetilde{W} \times (-) )$ is an $S \otimes_k \Re$-linear factorisation of $U + \widetilde{W}$.
\end{lemma}
\begin{proof}
The operator $D + b' + d\widetilde{W} \times (-)$ on $X \, \widehat{\otimes}\, \Bar$ arises from a map of the inverse system \eqref{eq:inverse_system} to itself, and the fact that this map squares to $U + \widetilde{W}$ can therefore be checked on the truncations $X \otimes \Bar/\Bar_{\ge i}$ where it follows from Lemma \ref{lemma:barisafactorisation_pre}.
\end{proof}

There is a morphism of linear factorisations of $U + \widetilde W$ over $S \otimes_k \Re$ (note that $\widetilde W = 0$ on $X \otimes R$)
\begin{equation}\label{eq:pibar}
\pi: \xymatrix{
X \, \widehat{\otimes}\, \Bar = \prod_{l \ge 0} X \otimes \Bar_l \ar@{->>}[r] & X \otimes \Bar_0 \ar[r] & X \otimes R
}
\end{equation}
where the first map is the projection and the second is the product $\Re \lto R$. Next we show using the perturbation lemma that this map is a homotopy equivalence, and we give an explicit homotopy inverse in terms of Atiyah classes. 

\begin{remark}\label{remark:specialcaseinvert} The reader should keep in mind the special case where $S = k[z], R = k[x]$ and $X$ is the extension of scalars from $S \otimes_k R$ to $S \otimes_k \Re$ via the ring map $s \otimes r \mapsto s \otimes r \otimes 1$ of a matrix factorisation $X'$ of $V - W$ over $k[x,z]$, with $V \in k[z], W \in k[x]$. Then $\pi$ is a morphism $X' \,\widehat{\otimes}_R\, \Bar \lto X'$ of linear factorisations of $V - W$.
\end{remark}

For the next lemma, note that the two factors of $\Re = R_1 \otimes_k R_2$ act in the same way on $X \otimes R$, but they act differently on $\Bar$, and hence on $X \, \widehat{\otimes} \, \Bar$. 

\begin{lemma}\label{lemma:firstdefo} There is an $R_2$-linear deformation retract of $\nZ _2$-graded complexes
\begin{equation}\label{eq:firstdefo1}
\xymatrix@C+2pc{
(X \otimes R, 0) \ar@<-1ex>[r]_-{\sigma_2} & (X \, \widehat{\otimes} \, \Bar, b') \ar@<-1ex>[l]_-{\pi}
}%
\!\!\!\xymatrix{%
{}\ar@(ur,dr)[]^{-d}
}%
\end{equation}
where $\sigma_2(e_i \otimes a ) = e_i \otimes (1 \otimes a)$.
\end{lemma}
\begin{proof}
To prove that this is a deformation retract in the sense of Section \ref{section:perturbation_lemma} we have to show that $\sigma_2 \pi = 1 - b' d - db'$. For this $X$ is irrelevant, and we are dealing with the operator $b'$ on the normalised bar complex \eqref{eq:normalised_bar_cpx_line} and its splitting $d$, with $\sigma_2 = d: R \lto R \otimes R = \Bar_0$. The identity follows from $b' d + d b' = 1$ given in (\ref{b'd+db'}), as long as one is careful to note that the $b'$ in $( X \, \widehat{\otimes} \, \Bar, b' )$ vanishes on $\Bar_0$ whereas the $b'$ in (\ref{b'd+db'}) is, in degree zero, what in the context of \eqref{eq:firstdefo1} we are calling $\pi$.
\end{proof}

\begin{lemma}\label{lemma:smallpertde} The perturbation $\delta = D + d\widetilde{W} \times (-)$ is small on $X\,\widehat{\otimes} \, \Bar$.
\end{lemma}
\begin{proof}
Let $h$ be $-d$. Because we are working with $\prod_{l \ge 0} X \otimes \Bar_l$ it is clear that the sum $\sum_{l \ge 0} (\delta h)^l$ converges as an operator on $X \, \widehat{\otimes}\, \Bar$ and this gives the desired inverse to $1 - \delta h$.
\end{proof}

Next we describe the homotopy inverse of $\pi$ using Atiyah classes. Note that $X$ is a free $(S \otimes_k \Re)$-module and we are using the notation of Example \ref{remark:linearatother}, i.\,e.
\[
\At_{R_1}(X) = [d, D]: X \otimes \Bar \lto X \otimes \Bar\,.
\]
We interpret this as an operator on the product $X \,\widehat{\otimes}\, \Bar$ as in \eqref{eq:listofcompletions}.

\begin{proposition}\label{prop:finalpertdefo} The morphism $\pi$ of \eqref{eq:pibar} is an $R_2$-linear homotopy equivalence with inverse
\[
\sigma_\infty = \sum_{l \ge 0} (-1)^l \At_{R_1}(X)^l \sigma_2\,.
\]
More precisely, there is an $R_2$-linear deformation retract
\be\label{eq:finalpertdefo001}
\xymatrix@C+2pc{
(X \otimes R, D \otimes 1) \ar@<-1ex>[r]_-{\sigma_\infty} & (X \, \widehat{\otimes} \, \Bar, D + b' + d\widetilde{W} \times (-)) \ar@<-1ex>[l]_-{\pi}
}.
\ee
\end{proposition}
\begin{proof}
It follows from the perturbation lemma (Proposition \ref{prop:pertlemma}) with Lemmas~\ref{lemma:firstdefo} and~\ref{lemma:smallpertde} that
\[
\xymatrix@C+2pc{
(X \otimes R, b_\infty) \ar@<-1ex>[r]_-{\sigma_\infty} & (X \, \widehat{\otimes} \, \Bar, D + b' + d\widetilde{W} \times (-)) \ar@<-1ex>[l]_-{\pi_\infty}
}
\]
is a deformation retract datum, where $\tau = \sum_{l \ge 0} (-1)^l (\delta d)^l \delta$
\begin{align*}
\sigma_\infty = \sum_{l \ge 0} (-1)^l (d \delta)^l \sigma_2 \, , \qquad
\pi_\infty = \pi + \pi \tau h\, , \qquad
b_\infty = \pi \tau \sigma_2\,.
\end{align*}
Clearly $\pi \delta = \pi D$ and $\pi$ vanishes on $\delta d$, so $\pi \tau = \pi D$. It follows that $b_\infty = D \otimes 1$ and $\pi_\infty = \pi - \pi D d = \pi$. So there is a deformation retract datum \eqref{eq:finalpertdefo001}, where we may use $d^2 = 0$ and $d \sigma_2 = 0$ to write
\[
\sigma_\infty = \sum_{l \ge 0} (-1)^l \big[ d, \delta \big]^l \sigma_2 = \sum_{l \ge 0} (-1)^l [d, D + d\widetilde{W} \times (-)]^l \sigma_2\,.
\]
Expanding this yields $\sum_{l \ge 0} (-1)^l [d, D]^l \sigma_2$ plus terms that look like
\begin{equation}\label{eq:finalpertdefo2}
\cdots [d, d\widetilde{W} \times (-)] [d,D]^i \sigma_2
\end{equation}
for some $i \ge 0$. Applying $[d,D]^i \sigma_2$ to an element of $X \otimes R$ produces a tensor whose form component is of the type $da_0 \ldots da_n \otimes a_{n+1}$. By Lemma \ref{lemma:shortobs} below $[d, d\widetilde{W} \times (-)]$ vanishes on such a tensor, so all terms of the form (\ref{eq:finalpertdefo2}) vanish and $\sigma_\infty$ is as given in the statement of the proposition.
\end{proof}

\begin{lemma}\label{lemma:shortobs} On $\Bar$ we have $[d, d\widetilde{W} \times (-)] = d \widetilde{W} \cdot d(-)$ and $[s, d\widetilde{W} \times (-)] = s(-) \cdot d\widetilde{W}$.
\end{lemma}
\begin{proof}
With $\alpha = d( W \otimes 1 - 1 \otimes W ) = dW \otimes 1$ and $\omega = a_0 da_1 \ldots da_n \otimes a_{n+1}$ we have
\begin{align*}
\alpha \times \omega &= a_0 dW da_1 \ldots da_n \otimes a_{n+1} - a_0 da_1 dW da_2 \ldots da_n \otimes a_{n+1}\\
& \qquad+ \ldots + (-1)^n a_0 da_1 \ldots da_n dW \otimes a_{n+1}\,.
\end{align*}
On the other hand
\begin{align*}
\alpha \times d\omega &= dW da_0 da_1 \ldots da_n \otimes a_{n+1} - da_0 dW da_1 \ldots da_n \otimes a_{n+1}\\
&\qquad+ \ldots + (-1)^{n+1} da_0 da_1 \ldots da_n dW \otimes a_{n+1}\,,
\end{align*}
from which it is clear that $d( \alpha \times \omega ) + \alpha \times d\omega = \alpha \cdot d\omega$. The proof for $s$ is similar.
\end{proof}

We will also need to consider the left action of $\Bar$ on $X$. Since the proofs are completely parallel, we only state the results. We now take $\Bar = \widetilde{\Omega}_{R_1}(\Re)$ which is the same underlying graded $k$-module as $\Omega_{R_2}(\Re)$ but with a different left and right action of $\Re$, and define
\[
\Bar \,\widehat{\otimes}\, X := \prod_{l \ge 0} (\Bar_l \otimes X)\,.
\]
By Lemma \ref{lemma:barisafactorisation2} this is a linear factorisation of $U + \widetilde W$ over $\Re \otimes_k S$ when equipped with the operator $D + b' + d\widetilde{W} \times (-)$. The morphism of linear factorisations $\pi$ is defined as the projection followed by multiplication
\begin{equation}\label{eq:pibar2}
\pi: \xymatrix{
\Bar \, \widehat{\otimes}\, X = \prod_{l \ge 0} \Bar_l \otimes X \ar@{->>}[r] & \Bar_0 \otimes X \ar[r] & R \otimes X \, .
}
\end{equation}
There is an $R_1$-linear deformation retract of $\nZ _2$-graded complexes
\be
\xymatrix@C+2pc{
(R \otimes X, 0) \ar@<-1ex>[r]_-{\sigma_1} & (\Bar \, \widehat{\otimes} \, X, b') \ar@<-1ex>[l]_-{\pi}
}%
\!\!\!\xymatrix{%
{}\ar@(ur,dr)[]^{s}
}%
\ee
where $\sigma_1( a \otimes e_i ) = (a \otimes 1) \otimes e_i$. This follows from the identity $b' s + s b' = 1$ in (\ref{b'd+db'}). To this we apply the perturbation lemma with perturbation $\delta = D + d\widetilde{W} \times (-)$ to obtain an $R_1$-linear deformation retract
\[
\xymatrix@C+2pc{
(R \otimes X, 1 \otimes D) \ar@<-1ex>[r]_-{\sigma_\infty} & (\Bar \, \widehat{\otimes} \, X, D + b' + d\widetilde{W} \times (-)) \ar@<-1ex>[l]_-{\pi} \, .
}
\]
The upshot is (again using the notation of Example \ref{remark:linearatother}):

\begin{proposition}\label{prop:finalpertdefo2} The morphism $\pi$ of \eqref{eq:pibar2} is an $R_1$-linear homotopy equivalence with inverse
\[
\sigma_\infty = \sum_{l \ge 0} \lAt_{R_2}(X)^l \sigma_1\,.
\]
\end{proposition}

Now we turn to the Koszul stabilisation of the diagonal. Assume that $R = k[x_1,\ldots,x_n]$ and let $(\Delta, d_{\Delta})$ be the finite-rank matrix factorisation of $\widetilde{W}$ over $\Re = R_1 \otimes_k R_2$ given in Section~\ref{subsec:bicatLG} where $d_{\Delta} = \delta_{+} + \delta_{-}$. The stabilisation morphism $\pi_{\Delta}: \Delta \lto R$ of \eqref{DeltaWstabmap} is compatible with the morphism $\pi$ of \eqref{eq:pibar} in the sense that there is a commutative diagram
\[
\xymatrix{
X \,\widehat{\otimes}\, \Bar \ar[dr]_-{\pi}\ar[rr]^{1 \otimes \Psi} & & X \otimes \Delta \ar[dl]^-{1 \otimes \pi_\Delta}\\
& X \otimes R
}
\]
where the horizontal map is the morphism of linear factorisations over $S \otimes_k \Re$
\[
1 \otimes \Psi: \xymatrix{
X \,\widehat{\otimes}\, \Bar = \prod_{l \ge 0} X \otimes \Bar_l \ar@{->>}[r] & \bigoplus_{0 \le l \le n} X \otimes \Bar_l \ar[r]^-{1 \otimes \Psi} & X \otimes \Delta
}
\]
in which the first map is the projection. Recall that $X \, \widehat{\otimes}\, \Bar$ is made into an $S$-$\Re$-bimodule using right multiplication by $\Re$ on $\Bar$, and $X \otimes \Delta = X \otimes_{\Re} \Delta$ comes with a natural $\Re$-action. In particular the maps $1 \otimes \pi_{\Delta}$ and its cousin $\pi_{\Delta} \otimes 1: \Delta \otimes X \lto R \otimes X$ are $\Re$-linear, and it makes sense to say that as $R_1$ or $R_2$-linear maps they are homotopy equivalences.

\begin{lemma}\label{lemma:koszulstab1} $1 \otimes \pi_\Delta$ is an $R_2$-linear (resp.~ $\pi_\Delta \otimes 1$ is an $R_1$-linear) homotopy equivalence.
\end{lemma}
\begin{proof}
With $\Delta_j = \bigwedge^j (\bigoplus_{i=1}^n \Re \theta_i)$ there is a split exact sequence
\[
\xymatrix{
0 \ar[r] & \Delta_n \ar[r]^-{\delta_{-}} & \cdots \ar[r]^-{\delta_{-}} & \Delta_0 \ar[r] & R \ar[r] & 0
}
\]
and the splittings provide the $R_2$-linear $\sigma: R \lto \Delta$ and homotopy $h$ making
\[
\xymatrix@C+2pc{
(X \otimes R, 0) \ar@<-1ex>[r]_-{\sigma} & (X \otimes \Delta, 1 \otimes \delta_{-}) \ar@<-1ex>[l]_-{1 \otimes \pi_\Delta}
}%
\!\!\!\xymatrix{%
{}\ar@(ur,dr)[]^{-h}
}%
\]
into an $R_2$-linear deformation retract datum. Moreover $\delta = D \otimes 1 + 1 \otimes \delta_{+}$ is a small perturbation and the perturbation lemma shows that $1 \otimes \pi_\Delta$ is a homotopy equivalence. A similar argument applies to $\pi_{\Delta} \otimes 1$.
\end{proof}

\begin{corollary}\label{corollary:koszulstab2} An $R_2$-linear homotopy inverse to $1 \otimes \pi_\Delta$ is given by
\begin{equation}\label{eq:koszulstab2_1}
(1 \otimes \pi_\Delta)^{-1} := \sum_{l \ge 0} (-1)^l \Psi \At_{R_1}(X)^l \sigma_2\,.
\end{equation}
\end{corollary}
\begin{proof}
By Proposition \ref{prop:finalpertdefo}, $\pi: X \,\widehat{\otimes}\, \Bar \lto X \otimes R$ is an $R_2$-linear homotopy equivalence with inverse $\sum_{l \ge 0} (-1)^l \At_{R_1}(X)^l \sigma_2$. Postcomposing with $\Psi$ gives the desired inverse for $1 \otimes \pi_{\Delta}$.
\end{proof}

From commutativity of the diagram
\[
\xymatrix{
\Bar \,\widehat{\otimes}\, X \ar[dr]_-{\pi}\ar[rr]^{\Psi \otimes 1} & & \Delta \otimes X \ar[dl]^-{\pi_\Delta \otimes 1}\\
& R \otimes X
}
\]
with the morphism $\pi$ from \eqref{eq:pibar2}, we deduce:

\begin{corollary}\label{corollary:koszulstab2r1} An $R_1$-linear homotopy inverse to $\pi_\Delta \otimes 1$ is given by
\begin{equation}\label{eq:koszulstab2_2}
(\pi_\Delta \otimes 1)^{-1} := \sum_{l \ge 0} \Psi \lAt_{R_2}(X)^l \sigma_1\,.
\end{equation}
\end{corollary}

\subsection{Inverses of unit actions}\label{section:rhoandlambdainverse}

Let us return to the problem of inverting the unit actions $\lambda$ and $\rho$ in the setting of Remark~\ref{remark:specialcaseinvert}, so that $R = k[x], S = k[z]$, $X$ is a matrix factorisation over $S \otimes_k R$ of $V - W$ with basis $\{ e_i \}_i$.

We begin with the right action $\rho$, so we take $X' = X \otimes_R \Re$ as the relevant matrix factorisation in the above. In this case $1 \otimes \pi_{\Delta}$ is the right unit action $\rho: X \otimes_R \Delta_W \lto X$, which is therefore an $(S \otimes _k R)$-linear homotopy equivalence. The homotopy inverse of $\rho$ is given by \eqref{eq:koszulstab2_1}, which is
\be\label{eq:formula_rhoinverse}
\rho^{-1} = \sum_{l \ge 0} (-1)^l \Psi \At_R(X)^l \sigma_2
\ee
using the notation of Example \ref{remark:linearatotherR}. The summands are composites of $(S \otimes_k R)$-linear maps
\[
\xymatrix@C+3.5pc{
X \ar[r]^-{\sigma_2} & X \otimes_R \Re \ar[r]^-{\At_R(X)^l \otimes_k 1_{R_2}} & X \otimes_R \Omega_{R_2}(\Re) \ar[r]^-{1 \otimes \Psi} & X \otimes_R \Delta_W
}
\]
where $\sigma_2(e_ia) = e_i \otimes (1 \otimes a)$ for $a \in R$. 

Often \eqref{eq:formula_rhoinverse} is the best presentation of $\rho^{-1}$, but note that using the formulas \eqref{intro_psi} and \eqref{eq:powerofatiyah} we can also be very concrete
(now writing $d_X$ instead of~$D$ in~\eqref{eq:powerofatiyah}): 
\be\label{eq:formularhoinv}
\rho^{-1}(e_i) = \sum_{l \ge 0} \sum_{i_1 < \cdots < i_l} \sum_j (-1)^{\binom{l}{2} + l|e_i|} e_j \otimes \left\{ \partial^{x,x'}_{[i_1]} d_X \ldots \partial^{x,x'}_{[i_l]} d_X \right\}_{ji} \theta_{i_1} \ldots \theta_{i_l}\,.
\ee

To invert the left action $\lambda$ we write $\Se = S_1 \otimes_k S_2$ with $S_i = S$, and take $X' = \Se \otimes_S X$ in the above (i.\,e. we switch the roles of $R$ and $S$) so that $\pi_{\Delta} \otimes 1$ is the left unit action $\lambda: \Delta_V \otimes_S X \lto X$. The $(S \otimes_k R)$-linear homotopy inverse of $\lambda$ is given by \eqref{eq:koszulstab2_2}, which is
\be\label{eq:formulalambdainv}
\lambda^{-1} = \sum_{l \ge 0} \Psi \lAt_S(X)^l \sigma_1
\ee
using the notation of Example \ref{remark:linearatotherS}. The summands are composites of $(S \otimes_k R)$-linear maps
\[
\xymatrix@C+3.5pc{
X \ar[r]^-{\sigma_1} & \Se \otimes_S X \ar[r]^-{1_{S_1} \otimes_k \lAt_S(X)^l} & \widetilde{\Omega}_{S_1}(\Se) \otimes_S X \ar[r]^-{\Psi} & \Delta_V \otimes_S X
}
\]
where $\sigma_1(ae_i) = (a \otimes 1) \otimes e_i$ for $a \in S$. Explicitly
\be
\lambda^{-1}(e_i) = \sum_{l \ge 0} \sum_{i_1 < \cdots < i_l} \sum_j \theta_{i_1} \ldots \theta_{i_l} \left\{ \partial^{z,z'}_{[i_l]} d_X \ldots \partial^{z,z'}_{[i_1]}d_X \right\}_{ji} \otimes e_j\,.
\ee
Notice that we have shown that $\lambda, \rho$ are homotopy equivalences, and provided explicit inverses, for an arbitrary matrix factorisation $X$ (i.\,e.~not necessarily finite-rank).

\subsection{The lifting problem}\label{section:liftingproblem}

The universal property of $\pi_{\Delta}: \Delta_W \lto R$ is that for any matrix factorisation $Y \in \hmf(\Re , \widetilde{W})$ and morphism of linear factorisations $\varphi: Y \lto R$ there is a unique (up to homotopy) morphism $\varphi_{\textup{lift}}$ making the following diagram commute up to homotopy: 
\be\label{eq:liftingproblem2}
\xymatrix{
Y \ar[dr]_-{\varphi} \ar@{.>}[rr]^-{\varphi_{\textup{lift}}} & & \Delta_W \ar[dl]^-{\pi_\Delta} \\
& R
} .
\ee
In this section we give an explicit formula for $\varphi_{\textup{lift}}$.

Having chosen a homogeneous basis $\{ e_i \}_{i}$ for $Y$, the formula is written in terms of the maps
\[
\xymatrix@C+2pc{
Y \ar[r]^-{\lAt_{R_1}(Y)^l} & \Omega_{R_2}(\Re) \otimes_{\Re } Y \ar[r]^-{1 \otimes \varphi'} & \Omega_{R_2}(\Re) \otimes_{\Re } \Re  \cong \Omega_{R_2}(\Re) \ar[r]^-{\Psi} & \Delta_W
} 
\]
using the notation of Example \ref{example:linearother3}, and the map $\varphi': Y \lto \Re$ which is the unique $\Re$-linear map defined on basis elements by $\varphi'(e_i) = 1 \otimes \varphi(e_i)$.

\begin{proposition}\label{prop:liftingresult} 
In the above notation
\begin{equation}\label{eq:liftingresult}
\varphi_{\textup{lift}} = \sum_{l \ge 0} \Psi (1 \otimes \varphi') \lAt_{R_1}(Y)^l
\end{equation}
is a morphism of matrix factorisations making \eqref{eq:liftingproblem2} commute.
\end{proposition}
\begin{proof}
We apply Corollary \ref{corollary:koszulstab2} to see that the composite
\[
\xymatrix@C+2pc{
\Hom_{\Re }(Y, \Delta_W) \ar[r]^-{\cong}_-{\xi^{-1}} & Y^{\vee} \otimes_{\Re } \Delta_W \ar[r]_-{1 \otimes \pi_{\Delta}} & Y^{\vee} \otimes_{\Re } R \ar[r]^-{\cong}_-{\xi} & \Hom_{\Re }(Y, R)
}
\]
is a homotopy equivalence, with $\xi$ denoting the canonical isomorphisms (see Section \ref{section:canonicalmaps}). Evaluating the homotopy inverse $\xi \circ (1 \otimes \pi_\Delta)^{-1} \circ \xi^{-1}$ on the cohomology class of $\varphi$ yields the map $\varphi_{\textup{lift}}$ defined by
\[
\varphi_{\textup{lift}} = \sum_i \sum_{l \ge 0} (-1)^l \xi \Psi \At_{R_1}(Y^\vee)^l \sigma_2( e_i^* \otimes \varphi(e_i) )\,.
\]
A computation yields
\[
\varphi_{\textup{lift}}(e_j) = \sum_{l \ge 0}\sum_{j_1,\ldots,j_l} (-1)^{l|e_j| + l} \Psi\left( d( d_{Y,j_1j}) \ldots d( d_{Y,j_l j_{l-1}} ) \cdot (1 \otimes \varphi(e_{j_l}) ) \right)
\]
which by \eqref{eq:leftatiyahexp} agrees with the right-hand side of \eqref{eq:liftingresult}.
\end{proof}

\section{Evaluation and coevaluation}\label{sec:derivcoeval}

We will show that every $1$-morphism in $\LG$ has a left and right adjoint. Specifically, if a $1$-morphism $W \lto V$ in $\LG$ is given by a finite-rank matrix factorisation $X$ of $V - W$ over $k[x,z]$, where $W\in R = k[x_1,\ldots,x_n]$ and $V\in S = k[z_1,\ldots,z_m]$, then we prove that the $1$-morphisms
\be\label{eq:derivcoeval11}
X^\dual = R[n] \otimes_R X^\vee \, , \qquad {}^\dual X = X^\vee \otimes_S S[m]
\ee
are respectively the right and left adjoints of $X$ in $\LG$. The dual 
$X^\vee = \Hom_{S\otimes_k R}(X, S\otimes_k R)$ 
is the matrix factorisation of $W - V$ described in \eqref{eq:differentials_adjoints}. To prove that these $1$-morphisms are adjoint to $X$ we define in this section two pairs of evaluation and coevaluation maps
\begin{align}
\coev_X &: \Delta_V \lto X \otimes_R {}^\dual X
\, , \qquad
\eval_X: {}^\dual X \otimes_S X \lto \Delta_W \, , \label{eq:derivcoeval1}\\
\widetilde\coev_X&: \Delta_W \lto X^\dual \otimes_S X
\, , \qquad
\widetilde\eval_X: X \otimes_R X^\dual \lto \Delta_V \label{eq:derivcoeval2}
\end{align}
and then in Section~\ref{sec:Zorro} we will prove that the Zorro moves~\eqref{Zorros} and~\eqref{otherZorros} hold for these maps. We begin by giving the explicit formulas for the evaluation and coevaluation maps in terms of Atiyah classes and residues, in Definition \ref{defa:first_coevblah} and Definition \ref{defa:first_evblah} below. It is not immediately obvious that these complicated formulas even define chain maps, but in Section \ref{subsec:derivcoeval} and Section \ref{subsec:eval} we show how to construct these chain maps from simple inputs using homological perturbation. On a first reading, we suggest that the reader proceed directly from the definition of the evaluation and coevaluation maps to the proof in Section~\ref{sec:Zorro} that these morphisms define adjunctions. For earlier work in this direction see~\cite{brs0909.0696, cr1006.5609}. 

Throughout we write $R = k[x]$, $S = k[z]$ and $\Re  = R_1 \otimes_k R_2$ with $R_i = R$ and $\Se  = S_1 \otimes_k S_2$ with $S_i = S$. Then $\eval_X, \widetilde\coev_X$ are $\Re$-linear morphisms of linear factorisations of $\widetilde{W} = W \otimes 1 - 1 \otimes W$, while $\coev_X, \widetilde\eval_X$ are $\Se$-linear morphisms of linear factorisations of $\widetilde{V} = V \otimes 1 - 1 \otimes V$. 

\begin{remark} The matrix factorisations $X^\dual, {}^\dual X$ are canonically isomorphic to $X^\vee[n]$ and $X^\vee[m]$ respectively, but there are good reasons to prefer the presentation of \eqref{eq:derivcoeval11}, for example the graphical calculus of Section \ref{sec:wiggliesandsigns}. While $R[n] \otimes_R X^\vee$ is isomorphic to $X^\vee[n]$ with no intervention of signs, the isomorphism $X^\vee[m] \cong X^\vee \otimes_S S[m]$, $\nu \lmt (-1)^{m|\nu|} \nu$, does involve signs.
\end{remark}

The chain level representations for the evaluation and coevaluation maps depend on a choice of homogeneous basis $\{ e_i \}_{i}$ for $X$ with dual basis $\{ e_i^* \}_{i}$, but the homotopy equivalence classes are independent of this choice. We also give completely elementary expressions for these maps involving only divided difference operators, see Remarks \ref{remark:explicitev} and \ref{remark:coev_divideddiff}. 

To define the coevaluation maps we need to introduce various notation, beginning with:
\[
\iota_X = \sum_{j} (-1)^{|e_j|} e_j^* \otimes e_j \in X^\dual \otimes_S X, \qquad \iota'_X = \sum_j e_j \otimes e_j^* \in X \otimes_R {}^\dual X\,.
\]
In the exterior algebra which is the underlying graded module of $\Delta_W$ we write $\wedge$ for the multiplication, and $\varepsilon: \Delta_W \lto \Re$ is the map from \eqref{eq:vareps}. Since $X^\dual \otimes_S X$ is an $\Re$-module, we can talk about the Atiyah class with respect to the ring map $R_2 \lto \Re$ (Example \ref{remark:linearatother}) which is defined by taking the commutator with the operator $d$ on $\Bar = \Omega_{R_2}(\Re)$ which ``differentiates'' in the $R_1$-directions:
\[
\At_{R_1}(X^\dual \otimes_S X) = [d, d_{X^\dual \otimes X}]: (X^\dual \otimes_S X) \otimes_{\Re} \Bar \lto (X^\dual \otimes_S X) \otimes_{\Re} \Bar\,.
\]
Note that $[d, d_{X^\dual \otimes X}] = [d, d_{X^\dual} \otimes 1_X]$ as $d$ is $R_2$-linear. The twisted version $\Atlarrow_{R_1}(X^\dual \otimes_S X)$ of this Atiyah class is defined by formally moving the noncommutative forms to the left (see Definition \ref{eq:leftarrowAt}) and given $\gamma \in \Delta_W$ and $l \ge 0$ we may define a map, writing $X^\dual X$ for $X^\dual \otimes_S X$,
\[
\xymatrix{
X^\dual X \ar[rr]^-{\Atlarrow_{R_1}(X^\dual X)^l} & & \gamma_* \Bar \otimes_{\Re} ( X^\dual X ) \ar[r]^-{\Psi} & \Delta_W \otimes_{\Re} (X^\dual X) \ar[r]^-{\gamma \wedge (-)} & \Delta_W \otimes_{\Re} (X^\dual X) \ar[r]^-{\varepsilon} & X^\dual X
}.
\]
These maps are used to define $\widetilde\coev_X$, and similar constructions go into $\coev_X$:

\begin{definition}\label{defa:first_coevblah} The coevaluation maps are defined for $\gamma \in \Delta_W$ and $\gamma' \in \Delta_V$ by
\begin{align}
\widetilde\coev_X(\gamma) &= \sum_{l \ge 0} \varepsilon \left( \gamma \wedge  (-1)^{l+nl} \Psi \Atlarrow_{R_1}(X^\dual \otimes_S X)^l( \iota_X ) \right) , \label{eq:coev10-2}\\
\coev_X( \gamma' ) &= \sum_{l \ge 0} \varepsilon \left( \gamma' \wedge (-1)^{m+l+ml} \Psi \Atlarrow_{S_1}(X \otimes_R {}^\dual X)^l( \iota'_X ) \right) . \label{eq:coev102}
\end{align}
\end{definition}

Next we come to the evaluation maps. These maps take values in the exterior algebra $\Delta$, which is $\nZ$-graded by form-degree, and we remark that to prove the adjointness of ${}^\dual X$ and $X^\dual$ to $X$ in $\LG$ it will actually be enough to describe the components of $\widetilde \eval_X$ and $\eval_X$ that land in form-degree zero (see Lemmas \ref{lemma:morphismevalzero} and \ref{lemma:morphismevalzeroother}). However, the full formulas are necessary in general string diagrams.

To write down chain maps it is necessary to make various choices. As the construction in Section~\ref{subsec:eval} below will show (see particularly Remark~\ref{remark:indeptlambda}) the result is independent of all of these choices up to homotopy. For $1 \le i \le n$ we choose a null-homotopy $\lambda_i$ on $X$ for the action of $\partial_{x_i} W$, for example $\lambda_i = - \partial_{x_i}(d_X)$ is such a null-homotopy by the Leibniz rule, and for $1 \le j \le m$ we choose a null-homotopy $\mu_j$ on $X$ for the action of $\partial_{z_j} V$, for example $\mu_j = \partial_{z_j}(d_X)$, and define
\[
\Lambda^{(x)} = \lambda_1 \ldots \lambda_n\,, \qquad
\Lambda^{(z)} = \mu_1 \ldots \mu_m\,.
\]

\begin{definition}\label{defa:first_evblah} The chain level representatives for the evaluation maps are
\begin{align}
\widetilde\eval_X( \eta \otimes \nu ) & = \sum_{l \ge 0} (-1)^{n+n|\eta|} \Res_{R/k} \left[ \frac{ \Psi \big< \lAt_{S_1}(X \otimes_R X^\vee)^l( \Lambda^{(x)} \eta \otimes \nu ) \big> \underline{\operatorname{d}\!x}}{\partial_{x_1}W, \ldots, \partial_{x_n} W} \right]\label{eq:constructevalmap00} , \\
\eval_X( \nu \otimes \eta ) & = \sum_{l \ge 0} (-1)^m \Res_{S/k} \left[ \frac{ \Psi \big< \lAt_{R_1}(X^{\vee} \otimes_S X)^l( \nu \Lambda^{(z)} \otimes \eta )\big> \underline{\operatorname{d}\!z}}{\partial_{z_1}V, \ldots, \partial_{z_m} V} \right]\label{eq:constructevalmap01}
\end{align}
where $\underline{\operatorname{d}\!x} = \ud x_1 \ldots \ud x_n$ and $\underline{\ud z} = \ud z_1 \ldots \ud z_m$,
and in~\eqref{eq:constructevalmap00} the map $\langle - \rangle$ is the $(\Se \otimes_k R)$-linear map
\begin{gather*}
\langle - \rangle: X \otimes_R X^\vee \lto \Se \otimes_k R\,,\qquad 
\langle e_i \otimes e_j^* \rangle = (-1)^{|e_i||e_j|} \delta_{ij} \, , 
\end{gather*}
while in \eqref{eq:constructevalmap01} it denotes the $(\Re \otimes_k S)$-linear map
$$
\langle - \rangle: X^{\vee} \otimes_S X \lto \Re \otimes_k S 
\, , \qquad 
\langle e_i^* \otimes e_j \rangle = \delta_{ij}\,.
$$
\end{definition}

Writing $X X^\vee$ for $X \otimes_R X^\vee$ the numerator of \eqref{eq:constructevalmap00} involves the maps
\[
\xymatrix@C+1pc{
X X^\vee \ar[r]^-{\lAt_{S_1}^l} & \Omega_{S_2}(\Se) \otimes_{\Se} X X^\vee \ar[r]^-{1 \otimes \langle - \rangle} & \Omega_{S_2}(\Se) \otimes_k R \ar[r]^-{\Psi \otimes 1} & \Delta_V \otimes_k R \ar[r]^-{\Res} & \Delta_V
}.
\]
We have written the formulas with Atiyah classes of e.\,g.~$X \otimes_R X^\vee$ rather than $X \otimes_R {}^\dual X$ since this simplifies the consideration of signs.

\begin{remark} There is a question of global signs. If we multiply $\eval$ and $\coev$ by a nonzero scalars~$a$ and~$a^{-1}$, respectively, then the resulting $2$-morphisms still define an adjunction between $X^\dual$ and $X$. The interaction between the two adjunctions fixing our normalisation\footnote{which differs from \cite{survey} by a sign $(-1)^m$ on both $\eval$ and $\coev$} is the pivotality identity of Section \ref{sec:wiggliesandsigns}.
\end{remark}

\subsection{Coevaluation}\label{subsec:derivcoeval}

In this section we present the coevaluation morphism $\widetilde\coev_X$ and the derivation of its explicit chain level representative. Afterwards we consider the other coevaluation $\coev_X$, but since the derivation is almost identical we will not provide all details. 
Let~$Y$ also be a $1$-morphism $W \lto V$ in $\LG$. 

\begin{proposition}\label{prop:isogivescoev} There is a canonical homotopy equivalence of $\nZ _2$-graded $k$-complexes
\begin{equation}
\Hom_{\Re }( \Delta_W, X^\dual \otimes_{S} Y ) \lto \Hom_{R \otimes_k S}(X,Y)\,. \label{eq:isogivescoev}
\end{equation}
\end{proposition}

\begin{definition}\label{def:coeval} We define $\widetilde\coev_X$ to be the morphism in the category $\HMF(\Re, \widetilde{W})$ whose cohomology class maps to $1_X$ under the quasi-isomorphism \eqref{eq:isogivescoev}
in the case $X=Y$. 
\end{definition}

\begin{proof}[Proof of Proposition \ref{prop:isogivescoev}]
Let $\Delta_W'$ denote the matrix factorisation of $- \widetilde{W}$ with the same underlying graded free module as $\Delta_W$ but the modified differential $d_{\Delta'} = - \delta_{+} + \delta_{-}$. This approximates the diagonal as a matrix factorisation of $- \widetilde{W}$ in the same way that $\Delta_W$ approximates it as a factorisation of $\widetilde{W}$. Also let~$K$ denote the $\nZ _2$-graded complex with the same underlying graded module as $\Delta_W$, but the differential~$\delta_{-}$, so~$K$ is the usual Koszul complex of $x_1 - x'_1, \ldots, x_n - x'_n$.

The product in the exterior algebra induces a morphism of complexes $\Delta_W' \otimes_{\Re} \Delta_W \lto K$. Here $\Delta_{W}' \otimes_{\Re} \Delta_W$ has the differential $d_{\Delta'} \otimes 1 + 1 \otimes d_{\Delta}$, which squares to $- \widetilde{W} + \widetilde{W} = 0$, so that this is a $\mathbb{Z}_2$-graded complex. Composing this with $\varepsilon: K \lto \Re[n]$ from~\eqref{eq:vareps} and taking the adjoint in the monoidal category of $\mathbb{Z}_2$-graded $\Re$-modules, we obtain the isomorphism of matrix factorisations $\zeta: \Delta_W' \lto \Delta_W^\vee [n]$ defined by $\omega \lmt \varepsilon( \omega \wedge - )$. Using the map $\xi$ of Section~\ref{section:canonicalmaps} we have a diagram
\begin{equation}
\xymatrix@C+3pc{
\Hom_{\Re}( \Delta_W, X^\vee[n] \otimes_{S} Y) & \Hom_{R \otimes_k S}(X,Y)\\
\left( X^\vee[n] \otimes_S Y \right) \otimes_{\Re} \Delta_W^{\vee} \ar[u]^-{\cong}_-{\xi} & \left( X^\vee \otimes_{S} Y \right) \otimes_{\Re} R \ar[u]_-{\kappa}^-{\cong}\\
\left( X^\vee \otimes_{S} Y \right) \otimes_{\Re} \Delta_W^\vee[n]  \ar[u]^-{\cong} & \left( X^\vee \otimes_S Y \right) \otimes_{\Re} \Delta_W' \ar[l]^-{\cong}_-{1 \otimes \zeta}\ar[u]_-{1 \otimes \pi}\\
}\label{eq:isogivescoev2}
\end{equation}
where for $f \in X^\vee$, $g \in Y$ and $h \in X$, $\kappa( f \otimes g \otimes 1 )(h) = (-1)^{|f||g|} f(h) \cdot g$. 
To complete the proof we need only to show that $1 \otimes \pi$ is a homotopy equivalence, but this is Lemma \ref{lemma:koszulstab1}.
\end{proof}


\begin{proposition}\label{prop:coev1} A representative for $\widetilde\coev_X$ is the chain map \eqref{eq:coev10-2}.
\end{proposition}
\begin{proof}
We lift the identity $1_X$ through the quasi-isomorphisms in \eqref{eq:isogivescoev2}
in the case $X=Y$, 
noting that $\iota_X = \kappa^{-1}(1_X)$. To find the inverse image in cohomology of $\iota_X$ under $1 \otimes \pi$, we use the homotopy inverse $(1 \otimes \pi_\Delta)^{-1}$ of Corollary \ref{corollary:koszulstab2} which gives $(1 \otimes \pi_\Delta)^{-1} \kappa^{-1}(1_X) = \sum_{l \ge 0} (-1)^l \Psi \At_{R_1}(X^{\vee} \otimes_S X)^l( \iota_X )$. After applying $\xi$ and $\zeta$ we arrive at the desired formula.
\end{proof}

To derive the explicit formula in terms of divided difference operators \eqref{eq:formula_intro_coev} for $\widetilde \coev_X$, we simply substitute the formula \eqref{eq:atlarrow_formula} for Atiyah classes and the definition \eqref{intro_psi} of $\Psi$. In this way we also see that there is an alternative presentation of the coevaluation map using $\lAt_{R_2}$, as
\be\label{eq:alt_formula_coev}
\widetilde\coev_X(\gamma) = \sum_{l \ge 0} (-1)^{nl} \varepsilon \left( \gamma \wedge \Psi \lAt_{R_2}(X^{\dual} \otimes_S X)^l( \iota_X ) \right) .
\ee

To construct the other coevaluation one begins with the canonical homotopy equivalence
\be\label{eq:othercoev11}
\Hom_{\Se}(\Delta_V, X \otimes_R {}^\dual X) \lto \Hom_{R \otimes_k S}(X, X)
\ee
defined as in the proof of Proposition \ref{prop:isogivescoev}. The coevaluation $\coev_X$ is defined to be the morphism whose cohomology class is the preimage of the identity under the quasi-isomorphism \eqref{eq:othercoev11} times a global sign factor of $(-1)^m$. Once again using the homotopy inverses computed in Corollary \ref{corollary:koszulstab2} and \ref{corollary:koszulstab2r1} one finds two different presentations of the same chain map:
\begin{proposition} 
A representative for $\coev_X$ is the chain map \eqref{eq:coev102}, and for $\gamma' \in \Delta_V$ also
\be
\coev_X(\gamma') = \sum_{l \ge 0} (-1)^{m+ml} \varepsilon \left( \gamma' \wedge \Psi \lAt_{S_2}(X \otimes_R {}^\dual X)^l( \iota'_X ) \right) .
\ee
\end{proposition}

\begin{remark}\label{remark:coev_divideddiff} Given $\gamma' \in \Delta_V$ then there is a sequence $b_1 < \cdots < b_l$ and integer $s$ with $\gamma' \wedge \theta_{b_1} \ldots \theta_{b_l} = (-1)^s \theta_1 \ldots \theta_m$, and again using the explicit formulas for Atiyah classes and $\Psi$ one computes that
\be\label{eq:coev_divideddiff}
\coev_X(\gamma') = \sum_{i,j} (-1)^{\binom{l+1}{2}+s+m+ml} \big\{ \partial^{z,z'}_{[b_1]} d_X \ldots \partial^{z,z'}_{[b_l]} d_X \big\}_{ij} e_{i} \otimes e_j^* \, .
\ee
\end{remark}

\subsection{Evaluation}\label{subsec:eval}

In this section we construct $\widetilde\eval_X$ by defining a simpler map $\widetilde\eval_0$ and then lifting this via perturbation. The other evaluation $\eval_X$ is defined in a similar way at the end of the section.

The partial derivatives $\partial_{x_i} W$ act null-homotopically on $X$ and we let $\lambda_i \in \Hom_{R \otimes_k S}(X,X)$ denote a degree-one map with $[d_X, \lambda_i] = \partial_{x_i} W \cdot 1_X$. The construction is independent of this choice up to homotopy, although this is not obvious; see Remark \ref{remark:indeptlambda}. We also set $\Lambda^{(x)} = \lambda_1 \ldots \lambda_n$.

\begin{lemma}\label{lemma:morphismevalzero} There is a morphism $\widetilde\eval_0: X \otimes_R X^\dual \lto S$ of linear factorisations of $\widetilde{V}$ over $\Se $
\begin{align}
\widetilde\eval_0( \eta \otimes \nu ) = (-1)^{n+n|\eta|} \Res_{R/k} \left[ \frac{ \str( \Lambda^{(x)} \circ \eta \circ \nu ) \, \underline{\operatorname{d}\!x}}{\partial_{x_1}W, \ldots, \partial_{x_n} W} \right] . 
\end{align}
\end{lemma}
\begin{proof}
While it is straightforward to see that $\widetilde\eval_0$ is a closed map from the formula, it will be useful later to express it as a composite of simple maps. To begin with there is a canonical map
\be\label{eq:defneval0_00}
X \otimes_R X^{\vee}[n] \cong (X \otimes_R X^{\vee})[n] \lto \left( X \otimes_R X^{\vee} \right)[n] \otimes_R \bar{R} \cong (\bar{X} \otimes_R \bar{X}^{\vee})[n]
\ee
where we write $\bar{R} = R/(\partial_{x_1} W,\ldots,\partial_{x_n} W)$ and $\bar{X} = X \otimes_R \bar{R}$. Then we compose with
\be\label{eq:defneval0_0}
\xymatrix@C+2.5pc{
(\bar{X} \otimes_R \bar{X}^{\vee})[n] \ar[r]^-{\Lambda^{(x)} \otimes 1} & \bar{X} \otimes_R \bar{X}^{\vee}
}
\ee
which is closed because $\Lambda^{(x)}$ is a closed map $X[n] \lto X$ modulo the $\partial_{x_i} W$. Finally compose with
\be\label{eq:defneval0_1}
\xymatrix@C+1.5pc{
\bar{X} \otimes_R \bar{X}^{\vee} \ar[r]^-{\can} & S \otimes_{\Se} ( \bar{X} \otimes_R \bar{X}^{\vee} ) \cong \Hom_{R \otimes_k S}(\bar{X}, \bar{X}) \ar[r]^-{\str} & \bar{R} \otimes_k S \ar[r]^-{\Res} & S
}
\ee
where the last map marked is the $S$-linear residue symbol (see Section~\ref{section:residuebackground}). We note that the first map in~\eqref{eq:defneval0_1} and the second in \eqref{eq:defneval0_1} are as defined in Section \ref{section:canonicalmaps}. Finally $\widetilde\eval_0$ is $(-1)^{\globalsigneval}$ times the composite of \eqref{eq:defneval0_00}, \eqref{eq:defneval0_0} and \eqref{eq:defneval0_1}.
\end{proof}

There is a stabilisation morphism $\pi_\Delta: \Delta_V \lto S$, and while $X \otimes_R X^\dual$ is not free of finite rank over $\Se $, it is a direct summand of a finite-rank matrix factorisation in $\HMF(\Se , \widetilde{V})$. The unique lifting statement of Proposition \ref{prop:liftingresult} can be extended to summands so there is up to homotopy a unique morphism making the following diagram commute:
\be\label{eq:defineeval}
\xymatrix{
X \otimes_R X^\dual \ar[dr]_-{\widetilde\eval_0} \ar@{.>}[rr]^-{\widetilde\eval_X} & & \Delta_V \ar[dl]^-{\pi_\Delta} \\
& S
} .
\ee

\begin{definition} $\widetilde\eval_X$ is the unique morphism in $\HMF(\Se , \widetilde{V})$ making \eqref{eq:defineeval} commute.
\end{definition}

\begin{proposition}\label{prop:constructevalmap} A representative for $\widetilde\eval_X$ is the chain map \eqref{eq:constructevalmap00}.
\end{proposition}
\begin{proof}
If $X \otimes_R X^\dual$ were finite-rank we could apply Proposition \ref{prop:liftingresult} directly to lift $\widetilde\eval_X$. As this is not the case, we first appeal to the idempotent pushforward construction of \cite{dm1102.2957}. Since $W$ is a potential, $\bar{R} = R/(\partial_{x_1} W,\ldots, \partial_{x_n} W)$ is a finite-rank free $k$-module, and Theorem $7.4$ of loc.\,cit.~shows that there is a diagram
\be\label{eq:idempotentpushdia}
\xymatrix@C+2pc{
(X \otimes_R X^\vee)[n] \ar@<-1ex>[r]_-{\vartheta} & \bar{X} \otimes_{R} \bar{X}^\vee := (X \otimes_R X^\vee) \otimes_R \bar{R} \ar@<-1ex>[l]_-{\upsilon}
}
\ee
in $\HMF(\Se , \widetilde{V})$ with $\upsilon \circ \vartheta = 1$ and $\vartheta = \Lambda^{(x)} \otimes 1$. Consider the commutative diagram
\be\label{eq:proofconstructevalmap}
\xymatrix@C+2pc{
\Hom_{\Se }( (X \otimes_R X^\vee)[n], \Delta ) \ar[d]_{\pi^\bullet_\Delta} & \Hom_{\Se }(\bar{X} \otimes_{R} \bar{X}^\vee, \Delta ) \ar[l]_-{\vartheta_\bullet} \ar[d]^{\pi^\bullet_\Delta}\\
\Hom_{\Se }( (X \otimes_R X^\vee)[n], S ) & \Hom_{\Se }( \bar{X} \otimes_{R} \bar{X}^\vee, S) \ar[l]_-{\vartheta_\bullet}
}
\ee
where a ``$\bullet$'' as a superscript indicates postcomposition and as a subscript it denotes precomposition. By Proposition \ref{prop:liftingresult} the right-hand vertical map is a homotopy equivalence and therefore so is the left-hand vertical map. This justifies why in \eqref{eq:defineeval} there is a unique morphism $\widetilde\eval_X$ making the diagram commute: it is the image in cohomology of $\widetilde\eval_0$ under the homotopy inverse of $\pi_{\Delta}^\bullet$.

Let $\widetilde\eval'_0$ denote $(-1)^{\globalsigneval}$ times the morphism in \eqref{eq:defneval0_1}, and let $\widetilde\eval'$ denote the morphism lifting $\widetilde\eval'_0$ which is produced by Proposition \ref{prop:liftingresult}, so that up to homotopy $\pi_{\Delta} \circ \widetilde\eval' = \widetilde\eval'_0$. To run the lifting construction we use the $\Se$-basis $\{ g_\alpha e_i \otimes e_j ^* \}_{\alpha, i, j}$ for $\bar{X} \otimes_{R} \bar{X}^\vee$, where $g_\alpha$ gives a $k$-basis of $\bar{R}$. Let $\psi$ be the map \eqref{eq:defneval0_00}. We define $\widetilde\eval_X = \vartheta_\bullet( \widetilde\eval' ) \circ \psi = \widetilde\eval' \circ (\Lambda^{(x)} \otimes 1) \circ \psi$. By construction this morphism makes \eqref{eq:defineeval} commute, and it just remains to compute it explicitly.

The statement of Proposition \ref{prop:liftingresult} gives us
\[
\widetilde\eval_X = \sum_{l \ge 0} \Psi \circ (1 \otimes \widetilde\eval''_0) \circ \lAt_{S_1}(\bar{X} \otimes_R \bar{X}^\vee)^l \circ ( \Lambda^{(x)} \otimes 1 ) \circ \psi
\]
where $\widetilde\eval''_0( g_\alpha e_i \otimes e_j^* ) = 1 \otimes \widetilde\eval'_0( g_\alpha e_i \otimes e_j^* )$. We compute using \eqref{eq:leftatiyahexp} that for $g \in R$
\begin{align*}
\widetilde\eval_X( g e_i \otimes e_j^* ) &= \sum_{l \ge 0} \sum_k (-1)^{n|e_i|} \Psi (1 \otimes \widetilde\eval''_0)\left( g \lAt_{S_1}(\bar{X} \otimes_R \bar{X}^\vee)^l( \Lambda^{(x)}_{ki} e_k \otimes e_j ^*) \right)\\
&= \sum_{l \ge 0} \sum_{k, k_1, \ldots, k_l} (-1)^{n|e_i|} \Psi (1 \otimes \widetilde\eval''_0)\left( g  \Lambda^{(x)}_{ki} \cdot d( d_{X,k_1k} ) \ldots d( d_{X,k_l k_{l-1}} ) \otimes e_{k_l} \otimes e_j^* \right) .
\end{align*}
Applying $\widetilde\eval''_0$ involves the supertrace $\str( e_{k_l} \circ e_j^* ) = (-1)^{|e_j|} \delta_{jk_l}$, so we are left with
\[
\sum_{l \ge 0} \sum_{k,k_1,\ldots,k_{l-1}} (-1)^{\globalsigneval+n|e_i|+|e_j|} \Res_{R/k} \left[ \frac{ g \Lambda^{(x)}_{ki} \Psi( d(d_{X,k_1k}) \ldots d(d_{X,jk_{l-1}}) )  \,  \underline{\operatorname{d}\!x}}{\partial_{x_1}W, \ldots, \partial_{x_n} W} \right]
\]
which agrees with \eqref{eq:constructevalmap00} for $\eta = g e_i$ and $\nu = e_j^*$, completing the proof.
\end{proof}

Next we give a brief derivation of $\eval_X$. As before we define a map $\eval_0$ and lift it via perturbation. Each $\partial_{z_i} V$ acts null-homotopically on $X$ and we let $\mu_i \in \Hom_{R \otimes_k S}(X,X)$ denote a degree-one map with $[d_X, \mu_i] = \partial_{z_i} V \cdot 1_X$. For example $\mu_i = \partial_{z_i} d_X$ would do. We set $\Lambda^{(z)} = \mu_1 \ldots \mu_m$.



\begin{lemma}\label{lemma:morphismevalzeroother} There is a morphism $\eval_0: {}^\dual X \otimes_S X \lto R$ of linear factorisations of $\widetilde W$ over $\Re$
\begin{align}
\eval_0( \nu \otimes \eta ) &= (-1)^{m+|\nu|}\Res_{S/k} \left[ \frac{ \str( \Lambda^{(z)} \circ \eta \circ \nu) \, \underline{\operatorname{d}\!z}}{\partial_{z_1}V, \ldots, \partial_{z_m} V} \right].
\end{align}
\end{lemma}
\begin{proof}
Either directly, or as in the proof of Lemma \ref{lemma:morphismevalzero}.
\end{proof}

We define $\eval_X$ to be the unique morphism in $\HMF(\Re, \widetilde{W})$ making the diagram
\be
\xymatrix{
{}^\dual X \otimes_S X \ar[dr]_-{\eval_0} \ar@{.>}[rr]^-{\eval_X} & & \Delta_W \ar[dl]^-{\pi_\Delta} \\
& R
}
\ee
commute. The existence of such a unique morphism, and the explicit formula for it, are established as before. First we have to write the infinite-rank matrix factorisation $X^{\vee}[m] \otimes_S X$ as a direct summand of a finite-rank factorisation as in \eqref{eq:idempotentpushdia}. Then Proposition \ref{prop:liftingresult} applies to this finite-rank factorisation in which we have embedded to produce the desired lifting.

\begin{proposition}\label{prop:constructevalmapnontilde} A representative for $\eval_X$ is the chain map \eqref{eq:constructevalmap01}.
\end{proposition}

\begin{remark}\label{remark:indeptlambda} 
The morphisms $\widetilde\eval_X$ and $\eval_X$ are independent, up to homotopy, of the choices of null-homotopies $\lambda_i$. To justify this claim it clearly suffices to argue that $\widetilde\eval_0$ and $\eval_0$ are independent of the choices, and this is a consequence of the remarks in Appendix \ref{app:reshom}. In fact it is also shown there that, up to the sign of the relevant permutation, $\widetilde\eval_X$ and $\eval_X$ are independent up to homotopy of the chosen \textsl{ordering} of the homotopies $\lambda_i$. 
\end{remark}

\begin{remark}\label{remark:explicitev} The formulas \eqref{eq:constructevalmap00} and \eqref{eq:constructevalmap01} for $\widetilde\eval_X$ and $\eval_X$ can be written in terms of divided difference operators by expanding the Atiyah classes $\lAt_R$ and $\lAt_S$ similarly to \eqref{eq:atlarrow_formula}: for $g \in R$ we have
$$
\widetilde\eval_X( g e_j \otimes e_i^* ) = \sum_{l \ge 0} \sum_{i_1 < \cdots < i_l} (-1)^{l + (n+1)|e_j|} \, \theta_{i_1} \ldots \theta_{i_l}
\, \Res_{R/k} \left[ \frac{ \{ \partial^{z,z'}_{[i_l]} d_X  \ldots \partial^{z,z'}_{[i_1]} d_X \, \Lambda^{(x)} \}_{ij} g \, \underline{\operatorname{d}\!x}}{\partial_{x_1}W, \ldots, \partial_{x_n} W} \right]. 
$$
Similarly for $g \in S$, we have
$$
\eval_X( e_i^* \otimes g e_j ) = \sum_{l \ge 0} \sum_{i_1 < \cdots < i_l} (-1)^{\binom{l}{2}+l|e_j|+m} \, \theta_{i_1} \ldots \theta_{i_l}  
\, \Res_{S/k} \left[ \frac{ \{ \Lambda^{(z)} \, \partial^{x,x'}_{[i_1]} d_X  \ldots \partial^{x,x'}_{[i_l]} d_X \}_{ij} g \, \underline{\operatorname{d}\!z}}{\partial_{z_1}V, \ldots, \partial_{z_m} V} \right].
$$
\end{remark}

\section{Zorro moves}\label{sec:Zorro}

In this section we will show that the bicategory $\LG$ of Landau-Ginzburg models has adjoints, by proving that given potentials $W\in k[x] = k[x_1,\ldots,x_n]$ and $V\in k[z] = k[z_1,\ldots,z_m]$ and any matrix factorisation $X\in \hmf(k[x,z], V-W)$ the evaluation and coevaluation maps of Section \ref{sec:derivcoeval} satisfy the Zorro moves~\eqref{Zorros} and~\eqref{otherZorros}. Let us consider the first identity of~\eqref{otherZorros} in more detail: 
\be\label{Zorro1detail}
\begin{tikzpicture}[very thick,scale=1.0,color=blue!50!black, baseline=0cm]

\fill (0.2,1.6) circle (0pt) node {{\small $\Delta_V$}};
\fill (-0.2,-1.6) circle (0pt) node {{\small $\Delta_W$}};

\fill (1,1.8) circle (2.5pt) node[right] {{\small $\lambda$}};
\fill (-1,-1.8) circle (2.5pt) node[left] {{\small $\rho^{-1}$}};


\fill (-1.25,0) circle (0pt) node {{\footnotesize $S\vphantom{S}$}};
\fill (-0.5,0) circle (0pt) node {{\footnotesize $R\vphantom{R}$}};
\fill (0.5,0) circle (0pt) node {{\footnotesize $S\vphantom{R}$}};
\fill (1.25,0) circle (0pt) node {{\footnotesize $R\vphantom{S}$}};


\draw[dashed] (-0.5,0.75) .. controls +(0,0.75) and +(-0.25,-0.75) .. (1,1.8);
\draw[dashed] (0.5,-0.75) .. controls +(0,-0.75) and +(0.25,0.75) .. (-1,-1.8);

\draw[line width=0] 
(1,2.7) node[line width=0pt] (A) {{\small $X$}}
(-1,-2.7) node[line width=0pt] (A2) {{\small $X$}}; 
\draw[redirected] (0,0) .. controls +(0,1) and +(0,1) .. (-1,0);
\draw[redirected] (1,0) .. controls +(0,-1) and +(0,-1) .. (0,0);
\draw (-1,0) -- (A2); 
\draw (1,0) -- (A); 
\end{tikzpicture}
=
\begin{tikzpicture}[very thick,scale=1.0,color=blue!50!black, baseline=0cm]
\draw[line width=0] 
(0,2.7) node[line width=0pt] (A) {{\small $X$}}
(0,-2.7) node[line width=0pt] (A2) {{\small $X$}}; 
\draw (A2) -- (A); 
\end{tikzpicture}
. 
\ee
Here we label the domains by the rings $R = k[x]$ and $S = k[z]$ pertaining to the two objects of $\LG$. We call the left-hand side of (\ref{Zorro1detail}) the \textsl{Zorro map} and denote it $\mathcal Z$. It is the composite
\be\label{eq:zorro1a1}
\xymatrix@C+2pc{
X \ar[r]^-{\rho^{-1}} & X \otimes_R \Delta_W \ar[r]^-{1 \otimes\,\widetilde{\coev}} & X \otimes_R X^{\dagger} \otimes_S X \ar[r]^-{\widetilde{\eval}\, \otimes 1} & \Delta_V \otimes_S X \ar[r]^-{\lambda} & X\,.
}
\ee
We prove that this map is homotopic to the identity on $X$. Since the argument is very similar for the other identity of \eqref{otherZorros} and the two identities of \eqref{Zorros}, we do not give full details in these cases. 

Some elements of notation:
\begin{itemize}
\item $D = [d_X, -]$ is the differential on $\End(X) = \End_{R \otimes_k S}(X)$.
\item $\{ e_i \}_{i}$ denotes a homogeneous $(S \otimes_k R)$-basis of $X$ and $\{ e_i^* \}_{i}$ the dual basis.
\item $f_i = \partial_{x_i} W$ and we choose a null-homotopy $\lambda_i$ on $X$ for the action of $f_i$. For example $\lambda_i= -\partial_{x_i}d_X(x,z)$ will do. Postcomposition with $\lambda_i$ defines an operator $\psi \mapsto \lambda_i \circ \psi$ on $\End(X)$, which we also denote by $\lambda_i$. 
\item $\Re = R_1 \otimes_k R_2$ with $R_i = R$ and $\Se  = S_1 \otimes_k S_2$ with $S_i = S$.
\item $\Bar = \Omega_{R_2}(\Re)$. Recall that left and right multiplication in $\Omega_{R_2}(\Re)$ make $\Bar$ into an $\Re$-bimodule. 
\end{itemize}
To establish the expression for $\mathcal Z$ in Lemma \ref{lemma:Zorrointermediate} below we need:

\begin{definition} With $\underline{\operatorname{d}\!x}=\operatorname{d}\!x_1\ldots \operatorname{d}\!x_n$ we define the degree zero map
\be\label{eq:bigpairing_defn}
\big\langle\!\big\langle - \big\rangle\!\big\rangle: \Bar \lto R[n]
\, ,\qquad
\big\langle\!\big\langle \alpha \big\rangle\!\big\rangle = \Res_{\Re/R_2} \left[ \frac{ \varepsilon\Psi( \alpha ) \, \underline{\operatorname{d}\!x}}{f_1,\ldots,f_n} \right]
\ee
using $\Psi: \Bar \lto \Delta$ of \eqref{intro_psi} and $\varepsilon: \Delta \lto \Re[n]$ of \eqref{eq:vareps}.
\end{definition}

This map is \textsl{left} $R_2$-linear and \textsl{right} $R_1$-linear, that is:

\begin{lemma}\label{lemma:linearity_bigpair} For $\alpha \in \Bar$ and $r,r' \in R$,
\begin{align}
\big\langle\!\big\langle (r \otimes r') \cdot \alpha \big\rangle\!\big\rangle &= r' \big\langle\!\big\langle (r \otimes 1) \cdot \alpha \big\rangle\!\big\rangle\,, \label{eq:bigpairing_right_linear0}\\
\big\langle\!\big\langle \alpha \cdot (r \otimes r') \big\rangle\!\big\rangle &= r \big\langle\!\big\langle \alpha \cdot (1 \otimes r') \big\rangle\!\big\rangle\,. \label{eq:bigpairing_right_linear}
\end{align}
\end{lemma}
\begin{proof}
Recall that the left $\Re$-action on $\Bar$ is given by \eqref{eq:leftreactionb}, and that $\Psi$ is left $\Re$-linear. Since $\varepsilon$ is of course $\Re$-linear, and the residue is $R_2$-linear, we deduce \eqref{eq:bigpairing_right_linear0}. The linearity on the \textsl{right} is less obvious but more important. On noncommutative forms $\omega$ of degree $n$, Lemma \ref{lem:PsiToTheRight} tells us that $\Psi( \omega \cdot (r \otimes 1) ) = \Psi( \omega ) \cdot (1 \otimes r)$. Combined with the $R_2$-linearity of the residue and the fact that \eqref{eq:bigpairing_defn} is only nonvanishing on degree $n$ forms, \eqref{eq:bigpairing_right_linear} follows.
\end{proof}

The formula for $\mathcal Z$ further involves the Atiyah class of the complex $\End(X)$ of $(S \otimes_k R)$-modules with respect to the ring map $k \lto R$. This is an operator
\be\label{eq:atiyah_class_endX}
\At = \At_R( \End(X) ): \End(X) \otimes_R \Omega R \lto \End(X) \otimes_R \Omega R
\ee
which is $S$-linear, and also right $R$-linear using multiplication by $R$ on the right of $\Omega R$. As explained in Example \ref{remark:linearatotherR}, this can be recast using the Bar complex $\Bar = \Omega R_1 \otimes_k R_2$ by identifying the $R$ in \eqref{eq:atiyah_class_endX} with $R_1$ and tensoring on the right with $R_2$. Then $\At$ is an $(S \otimes_k R_1)$-linear operator\footnote{The $R_1$-linearity of $\At$ is with respect to right multiplication by $R_1$ on $\Bar$.}
\[
\At = [d,D]: \End(X) \otimes_{R_1} \Bar \lto \End(X) \otimes_{R_1} \Bar\,.
\]
The extension of the usual operator $d$ on $\Bar$ to a connection on $\End(X) \otimes_{R_1} \Bar$ is as defined in Section \ref{section:atiyahclasses} using the natural basis $e_i \circ e_j^*$ for $\End(X)$. Finally, the supertrace is a map
\be\label{eq:strasfunction}
\str = \str \otimes 1: \End(X) \otimes_{R_1} \Bar \lto (S \otimes_k R) \otimes_{R_1} \Bar \cong S \otimes_k \Bar\,.
\ee

\begin{lemma}\label{lemma:Zorrointermediate} We have
\be\label{Zorrointermediate}
\mathcal Z = \sum_j (-1)^{n+|e_j|} \big\langle\!\big\langle \str \big( \lambda_1 \ldots \lambda_n \At^n (-\circ e_j^*) \big) \big\rangle\!\big\rangle \cdot e_j\,.
\ee
\end{lemma}
\begin{proof}
Throughout the proof $\At = \At_{R_1}$. We begin by computing the image of a basis element $e_q$ of $X$ under $\mathcal Z$. The image of $e_q$ under the first two maps of \eqref{eq:zorro1a1} is computed by the expressions in \eqref{eq:formula_rhoinverse} and \eqref{eq:coev10-2}, 
\begin{align*}
(1 \otimes \widetilde{\coev}) \circ \rho^{-1}( e_q ) &= (1 \otimes \widetilde{\coev}) \sum_{l \ge 0} (-1)^l \Psi \At(X)^l (e_q)\\
&= \sum_{j} \sum_{l+l'=n} (-1)^{n+nl'+|e_j|} \varepsilon\Big( \Psi \At(X)^l( e_q ) \wedge \Psi \Atlarrow(X^{\dual} \otimes_S X)^{l'}( e_j^*  \otimes e_j ) \Big) \\
&= \sum_{j} \sum_{l+l' = n} (-1)^{n+nl'+|e_j|} \varepsilon\Psi\Big( \At(X)^l( e_q ) \times \Atlarrow(X^{\dual} \otimes_S X)^{l'}( e_j^* \otimes e_j ) \Big) \\ 
&= \sum_j (-1)^{|e_j| + n|e_j|} \varepsilon \Psi\Big( \At(X \otimes_R X^\dual \otimes_S X)^n(e_q \otimes e_j^* \otimes e_j ) \Big) \\
&= \sum_j (-1)^{|e_j| + n|e_j|} \varepsilon \Psi\Big( \At(X \otimes_R X^\dual)^n(e_q \otimes e_j^*) \otimes e_j \Big)
\end{align*} 
where we use that $\Psi$ intertwines the shuffle product with the exterior product, and Lemma \ref{lemma:atshufat} in order to rewrite the shuffle product of Atiyah classes as the Atiyah class of the tensor product. The last step follows from observing that the Atiyah class is defined as a commutator with a connection which is linear in the variables $S$ contained in the final $X$, so overall the Atiyah class acts as the identity on the last tensor component.

To this we apply $\lambda \circ (\widetilde\eval \otimes 1)$, which amounts to the map $\widetilde\eval_0$ of Lemma \ref{lemma:morphismevalzero}. Applying the first isomorphism of \eqref{eq:defneval0_00} has the effect of changing the argument of the Atiyah class from $X \otimes_R X^{\vee}[n]$ to $(X \otimes_R X^\vee)[n]$ and replacing the sign by $(-1)^{(n+1)|e_j| + n|e_q|}$. For any $Y$ the Atiyah class of $Y[n]$ agrees with that of $Y$ so we can drop the shift by $n$. The next step is to apply $(-) \otimes_{\Se} S$ after which $X \otimes_R X^\vee$ becomes identified with $\End(X)$, and we are left with
\[
\sum_j (-1)^{(n+1)|e_j|+n|e_q|} \varepsilon \Psi \Big( \At( \End(X) )^n( e_q \circ e_j^* ) \otimes e_j \Big)\,.
\]
To finish applying $\widetilde\eval_0$ we compose with $\lambda_1 \ldots \lambda_n$ and take the residue, leaving us with the right-hand side of \eqref{Zorrointermediate} evaluated on $e_q$. Thus we have shown that both sides of \eqref{Zorrointermediate} agree on $e_q$. But the right-hand side is $(S \otimes_k R)$-linear since both the Atiyah class and $\dlangle - \drangle$ are right $R_1$-linear by Lemma \ref{lemma:linearity_bigpair}, so this completes the proof.
\end{proof}

\begin{remark}\label{remark:recastmathcalz}
There is a canonical $(S \otimes_k R)$-linear isomorphism of complexes
$$
\varrho: \End(X) \lto \Hom_{S \otimes_k R}( \End(X), S \otimes_k R)\,, \qquad \varrho(e) = \str( e \circ - )
$$
under which $\mathcal{Z}$ corresponds, by Lemma \ref{lemma:Zorrointermediate}, to the map
$$
\varrho(\mathcal{Z}) = (-1)^n \dlangle \str \big( \lambda_1 \ldots \lambda_n \At^n (-) \big) \drangle
$$
To prove that $\mathcal{Z} \simeq 1_X$ it is therefore equivalent to prove that $\varrho(\mathcal{Z}) \simeq \varrho(1_X) = \str(-)$.
\end{remark}

We are led to study the map 
\begin{equation}\label{eq:strlambdaat}
\str\, \lambda_1 \ldots \lambda_n \At^n( - \otimes - ): \End(X) \otimes_{R_1} \Bar \lto S \otimes_k \Bar\,.
\end{equation}

\begin{lemma}\label{lemma:stratzero} 
$\str \circ \At = 0$ as maps $\End(X) \otimes_{R_1} \Bar \lto S \otimes_k \Bar$.
\end{lemma}
\begin{proof}
We have $[d, \str] = 0$, $\str \, \circ \, D = 0$ and $\At = [d,D]$. So
\[
\str \circ \At = \str \circ  ( d D + D d ) = d \str D + \str D d = 0\,.
\]
\end{proof}

In what follows we identify $f_i = \partial_{x_i} W$ with an operator of left multiplication on $\End(X) \otimes_{R_1} \Bar$ defined by $\psi \otimes \alpha \mapsto f_i \psi \otimes \alpha = \psi \otimes (f_i \otimes 1)\alpha$. Similarly we write $d f_i$ for left multiplication by this element, with Koszul signs $\psi \otimes \alpha \mapsto (-1)^{|\psi|} \psi \otimes d f_i \alpha$. We can use the graded Jacobi identity to see that $[\lambda_i, \At] = [d, f_i]$ up to homotopy, and more generally: 

\begin{lemma}\label{lemma:commutator1} $[ \lambda_1 \ldots \lambda_n, \At ] = \sum_{i=1}^n (-1)^i \big[d, f_i \lambda_1 \ldots \widehat{\lambda_i} \ldots \lambda_n \big] - [D, [d, \lambda_1 \ldots \lambda_n]]$.
\end{lemma}

\begin{lemma}\label{lemma:strlambdaatfirst} $\str \lambda_1 \ldots \lambda_n \At = \sum_{i=1}^n (-1)^i \big[d, f_i \str \lambda_1 \ldots \widehat{\lambda_i} \ldots \lambda_n \big] + (-1)^{n+1} [d, \str \lambda_1 \ldots \lambda_n] D$.
\end{lemma}
\begin{proof}
By Lemma \ref{lemma:stratzero}, $\str \lambda_1 \ldots \lambda_n \At = \str[\lambda_1 \ldots \lambda_n, \At]$, and the rest follows from Lemma \ref{lemma:commutator1}.
\end{proof}
Of course there are similar formulas when $\lambda_1 \ldots \lambda_n$ is replaced by any subproduct. 

The fact that we are working with noncommutative forms introduces various subtleties that will be dealt with carefully below, but if we ignore these complications for a moment and naively replace $[d, f_i]$ by $d f_i$, the lemma allows us to compute the map \eqref{eq:strlambdaat} up to homotopy by contracting Atiyah classes with $\lambda_i$'s in all possible ways, as in the following ``meta-proof'': 
\begin{align*}
\dlangle \str &\big( \lambda_1 \ldots \lambda_n \At^n( -) \big) \drangle \simeq 
-\sum_{i_1} (-1)^{i_1+1} \dlangle df_{i_1} \str \big( \lambda_1 \cdots \widehat{\lambda}_{i_1} \cdots \lambda_n \At^{n-1}(-) \big) \drangle\\
&\simeq (-1)^2 \sum_{i_1 \neq i_2} (-1)^{i_1+1+i_2+2} \dlangle (df_{i_1} df_{i_2} - df_{i_2} df_{i_1}) \str \big( \lambda_1 \cdots \widehat{\lambda}_{i_1} \cdots \widehat{\lambda}_{i_2} \cdots \lambda_n \At^{n-2}(-) \big) \drangle\\
&\simeq \ldots\\
&\simeq (-1)^n \sum_{\sigma \in S_n} (-1)^{|\sigma|} \dlangle df_{\sigma(1)} \cdots df_{\sigma(n)} \str( - ) \drangle\\
&= (-1)^n \str( - )
\end{align*}
where in the last step we use a fact of the residue calculus: integrating against the determinant of the matrix $( \partial_{[i]} f_j )$ is the identity. This calculation should not be taken literally, but it does serve as a useful sketch of how we will use the commutation relations for the Atiyah class to deduce the Zorro move in the rest of this section.

Given a sequence of distinct elements $\alpha = (\alpha_p,\ldots,\alpha_1)$ in $\{1, \ldots, n\}$ we define the length of $\alpha$ to be $\ell(\alpha) = p$, the inversion number to be $|\alpha| = \sum_{1 \le i < j \le p} \delta_{\alpha_i > \alpha_j}$ (where by definition $\delta_{\alpha_i > \alpha_j}$ equals 1 if $\alpha_i > \alpha_j$, and zero otherwise), and we set $\gamma(\alpha) = |\alpha| + \alpha_1 + \ldots + \alpha_p$. We write $\Lambda_\alpha$ for the product $\lambda_1 \ldots \lambda_n$ with the $\lambda_{\alpha_1},\ldots,\lambda_{\alpha_p}$ omitted, and we use the same symbol $\Lambda_\alpha$ for the operator on $\End(X) \otimes_{R_1} \Bar$ defined by postcomposition with $\Lambda_\alpha$.

\begin{proposition}\label{prop:mainzorrotrick} For $1 \le p \le n$ we have
\begin{align}
\str \,&\lambda_1 \ldots \lambda_n\, \At^p = \sum_{\ell(\alpha) = p} (-1)^{\gamma(\alpha)} \big[d, f_{\alpha_p} \big[d, f_{\alpha_{p-1}} \cdots \big[d, f_{\alpha_1} \str \, \Lambda_\alpha \big] \cdots \big]\big] \label{eq:mainzorrotrick}\\
&+ \sum_{\substack{0 \le q < p \\ \ell(\alpha) = q}} (-1)^{\gamma(\alpha) + n - q + 1} \big[d, f_{\alpha_q} \big[d, f_{\alpha_{q-1}} \cdots \big[d, f_{\alpha_1} \big[d, \str \, \Lambda_\alpha\big] \big] \cdots \big]\big] \At^{p-q-1} D\,. \nonumber
\end{align}
\end{proposition} 
\begin{proof}
The proof is by induction on $p$, with the case $p = 1$ given by Lemma \ref{lemma:strlambdaatfirst}. Now suppose that \eqref{eq:mainzorrotrick} holds for $p$. For convenience we write $C_i$ for $[d, f_i (-)]$. Since $[d, \At] = 0$ and $[D, \At] = 0$ we can multiply \eqref{eq:mainzorrotrick} on the right by an Atiyah class: 
\begin{align}
\str\, &\lambda_1 \ldots \lambda_n \At^{p+1} = \sum_{\ell(\alpha) = p} (-1)^{\gamma(\alpha)} C_{\alpha_p} \ldots C_{\alpha_1}( \str \Lambda_\alpha \At )\label{eq:mainzorrotrick2}\\
&+ \sum_{\substack{0 \le q < p \\ \ell(\alpha) = q}} (-1)^{\gamma(\alpha) + n - q + 1} C_{\alpha_q} \ldots C_{\alpha_1}\Big( \big[d, \str \Lambda_\alpha \big] \Big) \At^{p+1-q-1} D\,.\nonumber
\end{align}
Using Lemma \ref{lemma:strlambdaatfirst} the first summand may be rewritten as
\begin{align*}
&\sum_{\ell(\alpha) = p} (-1)^{\gamma(\alpha)} C_{\alpha_p} \ldots C_{\alpha_1}\Big( \sum_{j \notin \alpha} (-1)^{j + \delta_{\alpha_1 < j} + \cdots + \delta_{\alpha_p < j}} \big[d, f_j \str \Lambda_{\alpha,j}\big] + (-1)^{n-\ell(\alpha) + 1}\big[d, \str \Lambda_\alpha\big] D \Big)\\
&= \sum_{\ell(\alpha) = p+1} (-1)^{\gamma(\alpha)} C_{\alpha_{p+1}} \ldots C_{\alpha_1}\big( \str \, \Lambda_\alpha \big) + \sum_{\ell(\alpha) = p} (-1)^{\gamma(\alpha) + n - \ell(\alpha) + 1} C_{\alpha_p} \ldots C_{\alpha_1}\Big( \big[d, \str \Lambda_\alpha\big] \Big) D \, . 
\end{align*}
Substituting this back into \eqref{eq:mainzorrotrick2} yields \eqref{eq:mainzorrotrick} for $p + 1$ and completes the inductive step.
\end{proof}

\begin{corollary}\label{corollary:nearlymaintheorem} For $1 \le p \le n$ we have
\begin{align}
\str \,&\lambda_1 \ldots \lambda_n \At^p \equiv \sum_{\ell(\alpha) = p} (-1)^{\gamma(\alpha)} d f_{\alpha_p} d f_{\alpha_{p-1}} \cdots d f_{\alpha_1} \str \Lambda_\alpha \nonumber \\
&+ \sum_{\substack{0 \le q < p \\ \ell(\alpha) = q}} (-1)^{\gamma(\alpha) + n - q + 1} d f_{\alpha_q} d f_{\alpha_{q-1}} \cdots d f_{\alpha_1} \big[d, \str \, \Lambda_\alpha\big] \At^{p-q-1} D\,. \nonumber
\end{align}
where $\equiv$ denotes equality in $S \otimes_k \Bar$ modulo the left action of the $f_i$ by $s \otimes \alpha \mapsto s \otimes (f_i \otimes 1)\alpha$.
\end{corollary}
\begin{proof}
In the formula \eqref{eq:mainzorrotrick} we see expressions of the form
\be\label{eq:whofui}
\sum_{\ell(\alpha) = q} (-1)^{\gamma(\alpha)} \big[d, f_{\alpha_q} \big[d, f_{\alpha_{q-1}} \cdots \big[d, f_{\alpha_1} \theta \big] \cdots \big]\big]
\ee
for some function $\theta$. Working modulo $f_i$ the first commutator can be replaced by $d \circ f_{\alpha_q}$, where we write ``$\circ$'' to indicate that this is $d$ preceeded by the operator of left multiplication by $f_{\alpha_q}$, and not multiplication by $d f_{\alpha_q}$. Expanding the second commutator gives
\begin{align*}
\sum_{\ell(\alpha) = q} (-1)^{\gamma(\alpha)} \Big( d \circ f_{\alpha_q} \circ d \circ f_{\alpha_{q-1}} [d,  \cdots \big[d, f_{\alpha_1} \theta \big] \cdots \big]
\pm d \circ f_{\alpha_q} \circ f_{\alpha_{q-1}} [d, \cdots \big[d, f_{\alpha_1} \theta \big] \cdots \big] d \Big)
\end{align*}
in which the $\pm$ is some sign not depending on $\alpha$. The second summand, in which $f_{\alpha_q} f_{\alpha_{q-1}}$ appears, cancels with the corresponding term for the sequence which is obtained from $\alpha$ by switching the $q$-th entry with the $(q+1)$-th. By this logic we see that \eqref{eq:whofui} equals
\be\label{eq:whofui2}
\sum_{\ell(\alpha) = q} (-1)^{\gamma(\alpha)} d \circ f_{\alpha_q} \circ d \circ f_{\alpha_{q-1}} \circ \cdots \circ d \circ f_{\alpha_1} \circ \theta\,.
\ee
A similar argument using the Leibniz rule allows us to replace each $d \circ f_{\alpha_i}$ by left multiplication by $d f_{\alpha_i}$, completing the proof.
\end{proof}

\begin{proposition}\label{prop:mainzorro} The Zorro map $\mathcal{Z}$ is homotopic to $1_X$.
\end{proposition}
\begin{proof}
The functional $\dlangle - \drangle$ kills $f_i \Bar$ for all $i$ because a residue with denominator $f_1,\ldots,f_n$ vanishes on any polynomial in the ideal generated by these elements. It follows from Corollary \ref{corollary:nearlymaintheorem} that
\[
\dlangle \str \, \lambda_1 \ldots \lambda_n \At^n(-) \drangle = \sum_{\ell(\alpha) = n} (-1)^{\gamma(\alpha)} \dlangle d f_{\alpha_n} d f_{\alpha_{p-1}} \cdots d f_{\alpha_1} \str(-) \drangle + HD
\]
for a map $H: \End(X) \lto S \otimes_k R$. The explicit formula for $H$ is not important here, but we do need to know that $H$ is $(S \otimes_k R)$-linear since left multiplication by $df_i$, the commutator $[d, \str \Lambda_\alpha]$ and the Atiyah class are all right $R_1$-linear. 

A sequence $\alpha$ of length $n$ is a permutation, and if we define the sequence $\alpha'$ by $\alpha'_i = \alpha_{n+1-i}$ then
\[
\gamma(\alpha) = 1 + \cdots + n + |\alpha| = \binom{n+1}{2} + \binom{n}{2} + |\alpha'| = n + |\alpha'|
\]
from which we deduce that
\begin{align*}
\dlangle \str \, \lambda_1 \ldots \lambda_n \At^n(-) \drangle &= (-1)^n \sum_{\sigma \in S_n} (-1)^{|\sigma|} \dlangle d f_{\sigma(1)} d f_{\sigma(2)} \ldots d f_{\sigma(n)} \str(-) \drangle + HD\\
&= (-1)^n \str(-) + HD
\end{align*}
using the residue identity of Proposition~\ref{prop:determinantresidue}. By Remark \ref{remark:recastmathcalz}, this proves that $\mathcal{Z} \simeq 1_X$.
\end{proof}

\begin{theorem}\label{theorem:adjointbicategoryyeah}
The bicategory $\LG$ has left and right adjoints, given for a $1$-morphism $X: W \lto V$ represented by a finite-rank matrix factorisation by
\[
\xymatrix{
X^{\vee} \otimes_S S[m] = {}^\dual X \ar@{-|}[r] & X \ar@{-|}[r] & X^\dual = R[n] \otimes_R X^{\vee}
} .
\]
Explicit expressions for the evaluation and coevaluation maps are those of Section~\ref{sec:derivcoeval}. 
\end{theorem}
\begin{proof}
We have shown in Proposition \ref{prop:mainzorro} that the first Zorro move in \eqref{otherZorros} holds. To establish the adjunction $X \dashv X^\dual$ it remains to prove the second Zorro move in \eqref{otherZorros}. This involves the inverse of the left unit action \eqref{eq:formulalambdainv} and the coevaluation, for which we use the alternative formula \eqref{eq:alt_formula_coev} in terms of $\lAt_{R_2}$. With this substitution the proof proceeds exactly as before with $s$ replacing $d$. The verification of ${}^\dual X \dashv X$ follows in the same way, using the formulas of Section \ref{sec:derivcoeval}.

This proves that ${}^\dual X$ and $X^\dual$ are adjoint to $X$ whenever $X$ is a finite-rank matrix factorisation. As explained in Section \ref{subsec:bicat} 
(see~\eqref{eq:rightadjointfunctor} and~\eqref{eq:leftadjointfunctor}), 
formal properties of adjoints ensure that there are canonical functors ${}^\dual (-), (-)^\dual: \hmf(k[x,z], V - W) \lto \hmf(k[x,z], W - V)$ and by Proposition \ref{prop.morphismmigration} below these functors are given on 
morphisms~$\varphi$ by the obvious duals, to wit $\varphi^\vee$ shifted as many times as there are $z$- and $x$-variables, respectively. 

Now let $X \in \LG(W,V)$ be an infinite-rank matrix factorisation which is a summand of a finite-rank matrix factorisation $X'$. Let $\kappa: X \lto X', \rho: X' \lto X$ be such that $\rho \circ \kappa = 1_X$ and set $e = \kappa \circ \rho$. Then $e^\dual$ is an idempotent on $(X')^\dual$ in $\LG(V, W)$ and since this category is idempotent complete, we may define $X^\dual$ to be the object which splits this idempotent. That is, we may find an object $Z$ together with morphisms $\kappa', \rho'$ as in the diagram
\[
\xymatrix@C+2pc{
(X')^\dual \ar@<-1ex>[r]_-{\rho'} & Z \ar@<-1ex>[l]_-{\kappa'} 
}
\]
such that $\kappa' \circ \rho' = e^\dual$ and $\rho' \circ \kappa' = 1_Z$, and it is straightforward to check that the morphisms 
\begin{gather*}
\widetilde{\coev}_X := \xymatrix@C+0.7pc{ \Delta_W \ar[r]^-{\widetilde\coev_{X'}} & (X')^\dual \otimes X' \ar[r]^-{\rho' \otimes \rho} & Z \otimes X }
, \quad
\widetilde{\eval}_X := \xymatrix@C+0.7pc{ X \otimes Z \ar[r]^-{\kappa \otimes \kappa'} & X' \otimes (X')^\dual \ar[r]^-{\widetilde\eval_{X'}} & \Delta_V }
\end{gather*}
exhibit $Z$ as the right adjoint $X^\dual$ of $X$. A similar argument applies for left adjoints, completing the proof that every $1$-morphism in $\LG$ admits left and right adjoints.
\end{proof}



With the notation of Remark \ref{remark:integral_functors} this shows that there are adjunctions of integral functors
\[
\xymatrix@C+1pc{
\Phi_{{}^\dual X} \ar@{-|}[r] & \Phi_X \ar@{-|}[r] & \Phi_{X^\dual}
} .
\]
That is, for any $Y \in \hmf(k[x], W)^\omega$ and $Z \in \hmf(k[z], V)^\omega$ there are natural isomorphisms
\begin{align}
\underline{\Hom}( X \otimes_{k[x]} Y, Z ) & \xlto{\cong} \underline{\Hom}( Y, X^\dual \otimes_{k[z]} Z )\,,\\
\underline{\Hom}( {}^\dual X \otimes_{k[z]} Z, Y ) & \xlto{\cong} \underline{\Hom}( Z, X \otimes_{k[x]} Y ) \label{eq:adjunc_bikesoutside}
\end{align}
where $\underline{\Hom}$ denotes morphisms in one of the categories $\hmf(k[z], V)^\omega$ or $\hmf(k[x], W)^\omega$.

The units and counits of adjunction are derived from the evaluation and coevaluation maps of Section~\ref{sec:derivcoeval} in the obvious way; for example, the counit of the adjunction $\Phi_X \dashv \Phi_{X^\dual}$ is the natural transformation $\Phi_X \circ \Phi_{X^\dual} \lto 1$ which, evaluated on an object $Z \in \hmf(k[z], V)$, is the morphism
\[
\xymatrix@C+2pc{
(\Phi_X \circ \Phi_{X^\dual})(Z) = X \otimes_{k[x]} X^\dual \otimes_{k[z]} Z \ar[r]^-{\widetilde\eval_X \otimes 1_Z} & \Delta_V \otimes_{k[z]} Z \ar[r]^-{\lambda_Z} & Z
}.
\]

\begin{remark} When $k$ is a field we recover Serre duality in the triangulated category $\hmf(k[z], V)^\omega$ as the special case of a $1$-morphism $X: \mathbb{I} \lto V$ where $\mathbb{I} = (k,0)$, with its left adjoint ${}^\dual X \cong X^{\vee}[m]: V \lto \mathbb{I}$. Taking $x$ to be the empty set of variables and $Y = k$, and writing $\underline{\Hom}$ for morphisms in both $\hmf(k[z],V)^\omega$ and the homotopy category of $\mathbb{Z}_2$-graded complexes $\hmf(k, 0)$, there is a natural isomorphism by \eqref{eq:adjunc_bikesoutside}
\begin{align*}
\Hom_k(\underline{\Hom}(X, Z[m]),k) &= H^m \Hom_k( \Hom_{k[z]}(X,Z), k )\\
&\cong \underline{\Hom}( \Hom_{k[z]}(X, Z)[m], k )\\
&\cong \underline{\Hom}( X^{\vee}[m] \otimes_{k[z]} Z, k )\\
&\cong \underline{\Hom}( Z, X )\,.
\end{align*}
\end{remark}


We end this section with a discussion of the naturality of the evaluation and coevaluation maps, which in diagrammatic notation amounts to a freedom to ``migrate'' $2$-morphisms around a cap or a cup. For a morphism $\varphi: X \lra Y$ in $\hmf(k[x_1,\ldots,x_n,z_1,\ldots,z_m], V-W)$ 
we define
$$
{}^\dual \varphi = \varphi^\vee[m] : {}^\dual Y \lra {}^\dual X \, , \qquad
\varphi^\dual = \varphi^\vee[n] : Y^\dual \lra X^\dual \, . 
$$
Formally speaking, the next result shows that these morphisms agree with the ones uniquely determined by the adjunction data as explained in Section~\ref{subsec:bicat}, 
see~\eqref{eq:funcrightdual}. 

\begin{proposition}
\label{prop.morphismmigration}
 For a morphism $\varphi: X \lra Y$ as above we have
\begin{enumerate}
\item $(\varphi \otimes 1_{{}^\dual X}) \circ \coev_X = (1_Y \otimes {}^\dual \varphi) \circ \coev_Y$, 
\item $\eval_Y \circ (1_{{}^\dual Y} \otimes \varphi) = \eval_X \circ ({}^\dual \varphi \otimes 1_X)$, 
\item ${}^\dual \varphi = \lambda_{{}^\dual X} \circ (\eval_Y \otimes 1_{{}^\dual X}) \circ (1_{{}^\dual Y} \otimes \varphi \otimes 1_{{}^\dual X}) \circ (1_{{}^\dual Y} \otimes \coev_X) \circ \rho^{-1}_{{}^\dual Y}$
\end{enumerate}
and similarly for $\widetilde\eval$ and $\widetilde\coev$. Diagrammatically the above identities read
$$
\begin{tikzpicture}[very thick,scale=1.0,color=blue!50!black, baseline=.4cm]
\draw[line width=0pt] 
(3,0.5) node[line width=0pt] (D) {}
(2,0.5) node[line width=0pt] (s) {}; 
\draw[directed] (D) .. controls +(0,-1) and +(0,-1) .. (s);
\fill (2,0.5) circle (2.5pt) node[left] {{\small $\varphi$}};
\draw (3,0.35) -- (3,1)
node[above] {{{\small${}^\dual X$}}};
\draw (2,0.35) -- (2,1)
node[above] {{{\small$Y$}}};
\end{tikzpicture} 
=
\begin{tikzpicture}[very thick,scale=1.0,color=blue!50!black, baseline=.4cm]
\draw[line width=0pt] 
(3,0.5) node[line width=0pt] (D) {}
(2,0.5) node[line width=0pt] (s) {}; 
\draw[directed] (D) .. controls +(0,-1) and +(0,-1) .. (s);
\fill (3,0.5) circle (2.5pt) node[right] {{\small ${}^\dual \varphi$}};
\draw (3,0.35) -- (3,1)
node[above] {{{\small${}^\dual X$}}};
\draw (2,0.35) -- (2,1)
node[above] {{{\small$Y$}}};
\end{tikzpicture} 
\, , \qquad
\begin{tikzpicture}[very thick,scale=1.0,color=blue!50!black, baseline=.4cm]
\draw[line width=0pt] 
(3,0.5) node[line width=0pt] (D) {}
(2,0.5) node[line width=0pt] (s) {}; 
\draw[directed] (D) .. controls +(0,1) and +(0,1) .. (s);
\fill (3,0.5) circle (2.5pt) node[right] {{\small $\varphi$}};
\draw (3,0.65) -- (3,0)
node[below] {{{\small$X\vphantom{{}^\dual Y}$}}};
\draw (2,0.65) -- (2,0)
node[below] {{{\small${}^\dual Y$}}};
\end{tikzpicture} 
=
\begin{tikzpicture}[very thick,scale=1.0,color=blue!50!black, baseline=.4cm]
\draw[line width=0pt] 
(3,0.5) node[line width=0pt] (D) {}
(2,0.5) node[line width=0pt] (s) {}; 
\draw[directed] (D) .. controls +(0,1) and +(0,1) .. (s);
\fill (2,0.5) circle (2.5pt) node[left] {{\small ${}^\dual \varphi$}};
\draw (3,0.65) -- (3,0)
node[below] {{{\small$X\vphantom{{}^\dual Y}$}}};
\draw (2,0.65) -- (2,0)
node[below] {{{\small${}^\dual Y$}}};
\end{tikzpicture} 
\, , \qquad
\begin{tikzpicture}[very thick,scale=1.0,color=blue!50!black, baseline=.4cm]
\fill (3,0.5) circle (2.5pt) node[right] {{\small ${}^\dual \varphi$}};
\draw (3,0) -- (3,1.5)
node[above] {{{\small${}^\dual X$}}};
\draw (3,0) -- (3,-0.5)
node[below] {{{\small${}^\dual Y$}}};
\end{tikzpicture} 
=
\begin{tikzpicture}[very thick,scale=1.0,color=blue!50!black, baseline=.4cm]
\draw[line width=0pt] 
(3,0.5) node[line width=0pt] (D) {}
(2,0.5) node[line width=0pt] (s) {}; 
\draw[directed] (D) .. controls +(0,-1) and +(0,-1) .. (s);
\fill (2,0.5) circle (2.5pt) node[right] {{\small $\varphi$}};
\draw (2,0.32) -- (2,0.68);
\draw[line width=0pt] 
(2,0.5) node[line width=0pt] (D) {}
(1,0.5) node[line width=0pt] (s) {}; 
\draw[directed] (D) .. controls +(0,+1) and +(0,+1) .. (s);
\draw (3,0.35) -- (3,1.5)
node[above] {{{\small${}^\dual X$}}};
\draw (1,0.65) -- (1,-0.5)
node[below] {{{\small${}^\dual Y$}}};
\end{tikzpicture} 
\, . 
$$
\end{proposition}

\begin{proof}
Consider the diagrams
$$
\xymatrix{%
\Hom(\Delta, X \otimes {}^\dual X) \ar[rr]^-{(\varphi \otimes 1)\circ(-)} && \Hom(\Delta, Y \otimes {}^\dual X) \\
\Hom(X,X) \ar[rr]_-{\varphi\circ(-)} \ar[u]^-{\cong} && \Hom(X,Y)\,, \ar[u]_-{\cong}
}%
\quad
\xymatrix{%
\Hom(\Delta, Y \otimes {}^\dual Y) \ar[rr]^-{(1 \otimes {}^\dual \varphi)\circ(-)} && \Hom(\Delta, Y \otimes {}^\dual X) \\
\Hom(Y,Y) \ar[rr]_-{(-)\circ\varphi} \ar[u]^-{\cong} && \Hom(X,Y)\,, \ar[u]_-{\cong}
}%
$$
where the vertical maps are analogous to~\eqref{eq:isogivescoev2}
(the only change being a switch from left to right adjoints). 
By their naturality both diagrams commute. The counterclockwise maps send $1_X \in \Hom(X,X)$ and $1_Y \in \Hom(Y,Y)$ to the same element in $\Hom(\Delta, {}^\dual X \otimes Y)$. Thus the two sides of part~(i) are equal as they are the images of~$1_X$,~$1_Y$ under the clockwise maps. 

To prove part~(iii) we only have to apply a Zorro move and use part~(i), and with this part~(ii) follows together with another Zorro move: 
$$
\begin{tikzpicture}[very thick,scale=1.0,color=blue!50!black, baseline=.4cm]
\fill (3,0.5) circle (2.5pt) node[right] {{\small ${}^\dual \varphi$}};
\draw (3,0) -- (3,1.5)
node[above] {{{\small${}^\dual X$}}};
\draw (3,0) -- (3,-0.5)
node[below] {{{\small${}^\dual Y$}}};
\end{tikzpicture} 
=
\begin{tikzpicture}[very thick,scale=1.0,color=blue!50!black, baseline=.4cm]
\draw[line width=0pt] 
(3,0.5) node[line width=0pt] (D) {}
(2,0.5) node[line width=0pt] (s) {}; 
\draw[directed] (D) .. controls +(0,-1) and +(0,-1) .. (s);
\fill (3,0.5) circle (2.5pt) node[left] {{\small ${}^\dual \varphi$}};
\draw (2,0.32) -- (2,0.68);
\draw[line width=0pt] 
(2,0.5) node[line width=0pt] (D) {}
(1,0.5) node[line width=0pt] (s) {}; 
\draw[directed] (D) .. controls +(0,+1) and +(0,+1) .. (s);
\draw (3,0.35) -- (3,1.5)
node[above] {{{\small${}^\dual X$}}};
\draw (1,0.65) -- (1,-0.5)
node[below] {{{\small${}^\dual Y$}}};
\end{tikzpicture} 
=
\begin{tikzpicture}[very thick,scale=1.0,color=blue!50!black, baseline=.4cm]
\draw[line width=0pt] 
(3,0.5) node[line width=0pt] (D) {}
(2,0.5) node[line width=0pt] (s) {}; 
\draw[directed] (D) .. controls +(0,-1) and +(0,-1) .. (s);
\fill (2,0.5) circle (2.5pt) node[right] {{\small $\varphi$}};
\draw (2,0.32) -- (2,0.68);
\draw[line width=0pt] 
(2,0.5) node[line width=0pt] (D) {}
(1,0.5) node[line width=0pt] (s) {}; 
\draw[directed] (D) .. controls +(0,+1) and +(0,+1) .. (s);
\draw (3,0.35) -- (3,1.5)
node[above] {{{\small${}^\dual X$}}};
\draw (1,0.65) -- (1,-0.5)
node[below] {{{\small${}^\dual Y$}}};
\end{tikzpicture} 
\Rightarrow
\begin{tikzpicture}[very thick,scale=1.0,color=blue!50!black, baseline=.4cm]
\draw[line width=0pt] 
(3,0.5) node[line width=0pt] (D) {}
(2,0.5) node[line width=0pt] (s) {}; 
\draw[directed] (D) .. controls +(0,1) and +(0,1) .. (s);
\fill (3,0.5) circle (2.5pt) node[left] {{\small $\varphi$}};
\draw (3,0.65) -- (3,-0.5)
node[below] {{{\small$X\vphantom{{}^\dual Y}$}}};
\draw (2,0.65) -- (2,-0.5)
node[below] {{{\small${}^\dual Y$}}};
\end{tikzpicture} 
=
\begin{tikzpicture}[very thick,scale=1.0,color=blue!50!black, baseline=.4cm]
\draw[line width=0pt] 
(0,0.5) node[line width=0pt] (D) {}
(-1,0.5) node[line width=0pt] (s) {}; 
\draw[directed] (D) .. controls +(0,1) and +(0,1) .. (s);
\fill (0,0.5) circle (2.5pt) node[left] {{\small $\varphi$}};
\draw (2,0.65) -- (2,-0.5)
node[below] {{{\small$X\vphantom{{}^\dual Y}$}}};
\draw (-1,0.65) -- (-1,-0.5)
node[below] {{{\small${}^\dual Y$}}};
\draw (2,0.32) -- (2,0.68);
\draw[line width=0pt] 
(2,0.5) node[line width=0pt] (D) {}
(1,0.5) node[line width=0pt] (s) {}; 
\draw[directed] (D) .. controls +(0,+1) and +(0,+1) .. (s);
\draw[line width=0pt] 
(1,0.5) node[line width=0pt] (D) {}
(0,0.5) node[line width=0pt] (s) {}; 
\draw[directed] (D) .. controls +(0,-1) and +(0,-1) .. (s);
\draw (1,0.32) -- (1,0.68);
\draw (0,0.32) -- (0,0.68);
\end{tikzpicture} 
=
\begin{tikzpicture}[very thick,scale=1.0,color=blue!50!black, baseline=.4cm]
\draw[line width=0pt] 
(3,0.5) node[line width=0pt] (D) {}
(2,0.5) node[line width=0pt] (s) {}; 
\draw[directed] (D) .. controls +(0,1) and +(0,1) .. (s);
\fill (2,0.5) circle (2.5pt) node[right] {{\small ${}^\dual \varphi$}};
\draw (3,0.65) -- (3,-0.5)
node[below] {{{\small$X\vphantom{{}^\dual Y}$}}};
\draw (2,0.65) -- (2,-0.5)
node[below] {{{\small${}^\dual Y$}}};
\end{tikzpicture} 
\, . 
$$
The obvious analogous identities for $\widetilde\eval$ and $\widetilde\coev$ follow in the same way. 
\end{proof}

\section{Pivotality}\label{sec:wiggliesandsigns}

With the proof that the bicategory $\LG$ has adjoints behind us, we turn in the next several sections to applications. But first we introduce a refined diagrammatic calculus more suitable for working with adjoints; this is also the appropriate setup for realising the goal of translating arbitrary TFT correlators directly into diagrams in $\LG$, which we will do in Section~\ref{sec:ocTFT}. The main result of this section is that $\LG$ is graded pivotal; this will be used for example in Sections~\ref{sec:defectaction} and~\ref{sec:ocTFT} to prove the Cardy condition.

\medskip

As mentioned in the Introduction, the adjoint of a $1$-morphism is interpreted in the TFT picture as the same defect condition but with reversed orientation. In the categorical approach to rational conformal field theory (conjecturally related to Landau-Ginzburg models via the CFT/LG correspondence) the fact that reversing orientation twice does nothing is encapsulated in a \textsl{pivotal structure}, which is given either by a monoidal transformation between the identity and the functor $(-)^{\dual\dual}$ or by a monoidal isomorphism between the functors ${}^{\dual}(-)$ and $(-)^{\dual}$ \cite[Section $4$]{freydyetter}. 

For pivotal bicategories there is a richer diagrammatic language \cite{JSGoTCII, khovdia, ladia} in which we allow non-progressive diagrams with oriented lines. Thus lines may double back, forming caps and cups, and in computing the value of a diagram one interprets a line labelled $f$ with a downward orientation in terms of its adjoint. The bicategory $\LG$ is not pivotal as left and right adjoints may differ by a shift, but it satisfies an appropriate notion of \textsl{graded pivotality}. To incorporate the shifts diagramatically we borrow a device from \cite{ct1007.2679} where we follow the prescription of \cite{JSGoTCII} but enhance the diagrams with ``wiggly'' lines. 

To explain the new conventions we begin in the setting of an arbitrary bicategory $\cat{B}$. Suppose that for every object $A$ of $\cat{B}$ there is given a $1$-morphism $\Omega_A \in \cat{B}(A,A)$. To indicate the special role of this $1$-morphism it is denoted by a wiggly line. We also suppose that there is given a $2$-isomorphism
$$
\mu = 
\begin{tikzpicture}[very thick,scale=1.0,color=blue!50!black, baseline=0.3cm]
\draw[very thin, color=blue!55!white, decorate, decoration={snake, amplitude=0.2mm, segment length=1.0mm}]  (0,0) .. controls +(0.1,1.0) and +(-0.1,1.0) .. (0.75,0); 
\draw[line width=1pt];
\end{tikzpicture}
: \Omega_A \otimes \Omega_A \lra \Delta_A
\, , \qquad
\mu^{-1} = 
\begin{tikzpicture}[very thick,scale=1.0,color=blue!50!black, baseline=-0.4cm, rotate=180]
\draw[very thin, color=blue!55!white, decorate, decoration={snake, amplitude=0.2mm, segment length=1.0mm}]  (0,0) .. controls +(0.1,1.0) and +(-0.1,1.0) .. (0.75,0); 
\draw[line width=1pt];
\end{tikzpicture}
: \Delta_A \lra \Omega_A \otimes \Omega_A
$$ 
which makes $\Omega_A$ both left and right adjoint to itself. The identity
\be\label{eq:switcherooomegalines}
\begin{tikzpicture}[very thick,scale=1.0,color=blue!50!black, baseline=0.9cm]
\draw[very thin, color=blue!55!white, decorate, decoration={snake, amplitude=0.2mm, segment length=1.0mm}]  (0,0) .. controls +(0.1,1.0) and +(-0.1,1.0) .. (0.75,0); 
\draw[very thin, color=blue!55!white, decorate, decoration={snake, amplitude=0.2mm, segment length=1.0mm}]  (0,1.85) .. controls +(0.1,-1.0) and +(-0.1,-1.0) .. (0.75,1.85); 
%
\end{tikzpicture}
\;=\;
\begin{tikzpicture}[very thick,scale=1.0,color=blue!50!black, baseline=0.9cm]
\draw[very thin, color=blue!55!white, decorate, decoration={snake, amplitude=0.2mm, segment length=1.0mm}]  (0,0) -- (0,1.85); 
\draw[very thin, color=blue!55!white, decorate, decoration={snake, amplitude=0.2mm, segment length=1.0mm}]  (0.25,0) -- (0.25,1.85); 
\end{tikzpicture}
\ee
will be useful in manipulating diagrams. A wiggly line in a region labelled $A$ means $\Omega_A$ and we will usually omit the subscript $A$. Since $\Omega$ is self-adjoint we may also omit orientations on wiggly lines.

For every $1$-morphism $X \in \cat{B}(A, B)$ we also suppose that there is given $X^{\vee} \in \cat{B}(B, A)$ together with $2$-morphisms (writing ${}^\dual X = X^{\vee} \otimes \Omega$ and $X^\dual = \Omega \otimes X^{\vee}$)
\begin{align}
&
\eval_X = 
\begin{tikzpicture}[very thick,scale=1.0,color=blue!50!black, baseline=.3cm]
\draw[directed] (3,0) .. controls +(0,1) and +(0,1) .. (2,0);
\draw[very thin, color=blue!55!white, decorate, decoration={snake, amplitude=0.2mm, segment length=1.0mm}] (2.1,0) .. controls +(0.1,0.5) and +(-0.5,-0.45) .. (2.5,0.75); 
\end{tikzpicture}
: {}^\dual X \otimes X \lra \Delta_A
\, , \qquad
\coev_X = 
\begin{tikzpicture}[very thick,scale=1.0,color=blue!50!black, baseline=-.4cm,rotate=180]
\draw[redirected] (3,0) .. controls +(0,1) and +(0,1) .. (2,0);
\draw[very thin, color=blue!55!white, decorate, decoration={snake, amplitude=0.2mm, segment length=1.0mm}] (1.87,0) .. controls +(0.1,0.5) and +(-0.7,-0.15) .. (2.5,0.75); 
\end{tikzpicture}
: \Delta_B \lra X \otimes {}^\dual X \, , \nonumber \\
&
\widetilde\eval_X = 
\begin{tikzpicture}[very thick,scale=1.0,color=blue!50!black, baseline=.3cm]
\draw[redirected] (3,0) .. controls +(0,1) and +(0,1) .. (2,0);
\draw[very thin, color=blue!55!white, decorate, decoration={snake, amplitude=0.2mm, segment length=1.0mm}] (2.9,0) .. controls +(-0.1,0.5) and +(0.5,-0.45) .. (2.5,0.75); 
\end{tikzpicture}
: X \otimes X^\dual \lra \Delta_B
\, , \qquad
\widetilde\coev_X = 
\begin{tikzpicture}[very thick,scale=1.0,color=blue!50!black, baseline=-.4cm,rotate=180]
\draw[directed] (3,0) .. controls +(0,1) and +(0,1) .. (2,0);
\draw[very thin, color=blue!55!white, decorate, decoration={snake, amplitude=0.2mm, segment length=1.0mm}] (3.13,0) .. controls +(-0.1,0.5) and +(0.7,-0.15) .. (2.5,0.75); 
\end{tikzpicture}
: \Delta_A \lra X^\dual \otimes  X \label{eq:wigglyadjunctionmaps}
\end{align}
which make ${}^\dual X$ left adjoint to $X$ and $X^\dual$ right adjoint to $X$. In terms of pictures:
$$
\begin{tikzpicture}[very thick,scale=1.0,color=blue!50!black, baseline=0cm]
\draw[directed] (0,0) .. controls +(0,-1) and +(0,-1) .. (-1,0);
\draw[directed] (1,0) .. controls +(0,1) and +(0,1) .. (0,0);
\draw (-1,0) -- (-1,1.25); 
\draw (1,0) -- (1,-1.25); 
\draw[very thin, color=blue!55!white, decorate, decoration={snake, amplitude=0.2mm, segment length=1.0mm}] (-0.5,-0.75) .. controls +(0.75,0) and +(-0.5,-0.35) .. (0.5,0.75); 
\end{tikzpicture}
\;=\;
\begin{tikzpicture}[very thick,scale=1.0,color=blue!50!black, baseline=0cm]
\draw[
	decoration={markings, mark=at position 0.5 with {\arrow{>}}}, postaction={decorate}
	]
	(0,-1.25) -- (0,1.25); 
\end{tikzpicture}
\; , \qquad
\begin{tikzpicture}[very thick,scale=1.0,color=blue!50!black, baseline=0cm]
\draw[directed] (0,0) .. controls +(0,1) and +(0,1) .. (-1,0);
\draw[directed] (1,0) .. controls +(0,-1) and +(0,-1) .. (0,0);
\draw (-1,0) -- (-1,-1.25); 
\draw (1,0) -- (1,1.25); 
\draw[very thin, color=blue!55!white, decorate, decoration={snake, amplitude=0.2mm, segment length=1.0mm}] (-0.9,-1.25) .. controls +(0,0.5) and +(-0.5,-0.35) .. (-0.5,0.75); 
\draw[very thin, color=blue!55!white, decorate, decoration={snake, amplitude=0.2mm, segment length=1.0mm}] (0.5,-0.75) .. controls +(0.8,0.05) and +(0,-1) .. (1.1,1.25); 
\end{tikzpicture}
\;=\;
\begin{tikzpicture}[very thick,scale=1.0,color=blue!50!black, baseline=0cm]
\draw[
	decoration={markings, mark=at position 0.5 with {\arrow{<}}}, postaction={decorate}
	]
	(0,-1.25) -- (0,1.25); 
\draw[very thin, color=blue!55!white, decorate, decoration={snake, amplitude=0.2mm, segment length=1.0mm}] (0.15,-1.25) -- (0.15,1.25); 
\end{tikzpicture} 
\; , \qquad 
%
%
%
%
%
%
\begin{tikzpicture}[very thick,scale=1.0,color=blue!50!black, baseline=0cm]
\draw[redirected] (0,0) .. controls +(0,-1) and +(0,-1) .. (-1,0);
\draw[redirected] (1,0) .. controls +(0,1) and +(0,1) .. (0,0);
\draw (-1,0) -- (-1,1.25); 
\draw (1,0) -- (1,-1.25); 
\draw[very thin, color=blue!55!white, decorate, decoration={snake, amplitude=0.2mm, segment length=1.0mm}] (-1.1,1.25) .. controls +(0,-1) and +(-0.85,0) .. (-0.5,-0.75); 
\draw[very thin, color=blue!55!white, decorate, decoration={snake, amplitude=0.2mm, segment length=1.0mm}] (0.5,0.75) .. controls +(0.35,-0.05) and +(0,1) .. (0.85,-1.25); 
\end{tikzpicture}
\;=\;
\begin{tikzpicture}[very thick,scale=1.0,color=blue!50!black, baseline=0cm]
\draw[
	decoration={markings, mark=at position 0.5 with {\arrow{<}}}, postaction={decorate}
	]
	(0,-1.25) -- (0,1.25); 
\draw[very thin, color=blue!55!white, decorate, decoration={snake, amplitude=0.2mm, segment length=1.0mm}] (-0.15,-1.25) -- (-0.15,1.25); 
\end{tikzpicture}
\; , \qquad
\begin{tikzpicture}[very thick,scale=1.0,color=blue!50!black, baseline=0cm]
\draw[redirected] (0,0) .. controls +(0,1) and +(0,1) .. (-1,0);
\draw[redirected] (1,0) .. controls +(0,-1) and +(0,-1) .. (0,0);
\draw (-1,0) -- (-1,-1.25); 
\draw (1,0) -- (1,1.25); 
\draw[very thin, color=blue!55!white, decorate, decoration={snake, amplitude=0.2mm, segment length=1.0mm}] (-0.5,0.75) .. controls +(0.35,-0.1) and +(-0.8,-0.15) .. (0.5,-0.75); 
\end{tikzpicture}
\;=\;
\begin{tikzpicture}[very thick,scale=1.0,color=blue!50!black, baseline=0cm]
\draw[
	decoration={markings, mark=at position 0.5 with {\arrow{>}}}, postaction={decorate}
	]
	(0,-1.25) -- (0,1.25); 
\end{tikzpicture} \; .
$$
From now on a downwards oriented solid line labelled by a $1$-morphism $X$ represents $X^\vee$. We call a bicategory $\cat{B}$ equipped with the above data a \textsl{graded bicategory}.

Given this structure there is a $2$-isomorphism between left and right adjoints given by
\be\label{eq:isoq}
q_X :=
\begin{tikzpicture}[very thick,scale=0.8,color=blue!50!black, baseline=0.7cm]
\draw[line width=1pt] ;
\draw[
	decoration={markings, mark=at position 0.35 with {\arrow{<}}}, postaction={decorate}
	]
	(0,0) -- (0,2);
\draw[very thin, color=blue!55!white, decorate, decoration={snake, amplitude=0.2mm, segment length=1.0mm}] (0.2,0) -- (0.2,2); 
\draw[very thin, color=blue!55!white, decorate, decoration={snake, amplitude=0.2mm, segment length=1.0mm}] (-0.5,0) .. controls +(0,0.9) and +(0,00.9) .. (-0.25,0); 
\end{tikzpicture}
\,\,
: \Omega \otimes X^\dual \otimes \Omega \lra {}^\dual X\,.
\ee
We can ask whether this $2$-isomorphism is \textsl{monoidal}, by which we mean that it is compatible with composition of $1$-morphisms. 
If $X: A \lto B$ and $Y: B \lto C$ are $1$-morphisms then $X^{\dual} \otimes Y^{\dual}$ and ${}^\dual X \otimes {}^\dual Y$ are right and left adjoint to $Y \otimes X$, respectively, by the usual principle that the composite of adjoints is the adjoint of the composite; see \cite[\S 2.1]{kellystreet} and \cite[Proposition I, $6.3$]{gray}. Hence by uniqueness of adjoints there are canonical isomorphisms
\[
\mathscr{R}: (Y \otimes X)^{\dual} \stackrel{\cong}{\lra} X^{\dual} \otimes Y^{\dual}\, , \qquad 
\mathscr{L}: {}^\dual (Y \otimes X) \stackrel{\cong}{\lra} {}^\dual X \otimes {}^\dual Y\,.
\]
These canonical isomorphisms can be written explicitly in terms of the units and counits of adjunction \cite[Proposition I, $6.3$]{gray} and may thus be presented by diagrams, which are shown below. In such diagrams the dotted horizontal line denotes an identity map $Y \otimes X \lto Y \otimes X$, but with the domain presented as two lines labelled $Y,X$ and the target presented as a single line labelled $Y \otimes X$.
\[
\mathscr R = 
\begin{tikzpicture}[very thick,scale=0.8,color=blue!50!black, baseline, rotate=180]
\draw[line width=1pt] 
(2,-2) node[line width=0pt, right] (Y) {{\small $Y$}}
(3,-2) node[line width=0pt, right] (X) {{\small $X$}}
(-1,3) node[line width=0pt, right] (XY) {{\small $Y\otimes X$}}; 
\draw[redirected] (1,0) .. controls +(0,1) and +(0,1) .. (2,0);
\draw[redirected] (0,0) .. controls +(0,2) and +(0,2) .. (3,0);
\draw[redirected] (-1,0) .. controls +(0,-1) and +(0,-1) .. (0.5,0);
\draw (2,0) -- (2,-2);
\draw (3,0) -- (3,-2);
\draw[dotted] (0,0) -- (1,0);
\draw (-1,0) -- (-1,3);
\draw[very thin, color=blue!55!white, decorate, decoration={snake, amplitude=0.2mm, segment length=1.0mm}] (-0.35,-0.76) .. controls +(-0.7,0.5) and +(0.1,-1) .. (-0.85,3); 
\draw[very thin, color=blue!55!white, decorate, decoration={snake, amplitude=0.2mm, segment length=1.0mm}] (2.15,-2) .. controls +(0,1.3) and +(0.8,0.05) .. (1.5,0.75); 
\draw[very thin, color=blue!55!white, decorate, decoration={snake, amplitude=0.2mm, segment length=1.0mm}] (3.15,-2) .. controls +(0,1.5) and +(2.0,0.05) .. (1.5,1.5); 
\end{tikzpicture}
\, , \qquad 
\mathscr L = 
\begin{tikzpicture}[very thick,scale=0.8,color=blue!50!black, baseline, rotate=180]
\draw[line width=1pt] 
(-2,-2) node[line width=0pt, left] (X) {{\small $Y$}}
(-1,-2) node[line width=0pt, left] (Y) {{\small $X$}}
(2,3) node[line width=0pt, left] (XY) {{\small $Y\otimes X$}}; 
\draw[redirected] (0,0) .. controls +(0,1) and +(0,1) .. (-1,0);
\draw[redirected] (1,0) .. controls +(0,2) and +(0,2) .. (-2,0);
\draw[redirected] (2,0) .. controls +(0,-1) and +(0,-1) .. (0.5,0);
\draw (-1,0) -- (-1,-2);
\draw (-2,0) -- (-2,-2);
\draw[dotted] (0,0) -- (1,0);
\draw (2,0) -- (2,3);
\draw[very thin, color=blue!55!white, decorate, decoration={snake, amplitude=0.2mm, segment length=1.0mm}] (-1.15,-2) .. controls +(0,1.0) and +(-1.0,0.1) .. (-0.4,0.75); 
\draw[very thin, color=blue!55!white, decorate, decoration={snake, amplitude=0.2mm, segment length=1.0mm}] (-2.15,-2) .. controls +(0,1.5) and +(-2.3,-0.1) .. (-0.4,1.5); 
\draw[very thin, color=blue!55!white, decorate, decoration={snake, amplitude=0.2mm, segment length=1.0mm}] (1.25,-0.76) .. controls +(0.7,0.5) and +(0,-1) .. (1.85,3);
\end{tikzpicture}
\, . 
\]
Note that this graphical presentation makes it manifest that $\mathscr{R},\mathscr{L}$ are indeed isomorphisms in $\LG$: applying the Zorro moves one checks that their inverses are given by copies of themselves rotated by 180 degrees and with the wiggly lines suitably rearranged. A graded bicategory is graded pivotal if these canonical isomorphisms $\mathscr{R},\mathscr{L}$ interact naturally with the isomorphisms $q$ between left and right adjoints:

\begin{definition} 
A graded bicategory $\cat{B}$ is called \textsl{graded pivotal} if for every pair of composable $1$-morphisms $X,Y$ as above, the diagram
\be\label{eq:nicepivotalitysq}
\xymatrix@C+2pc{
\Omega \otimes (Y \otimes X)^\dual \otimes \Omega \ar[d]_-{\mathscr{R}}\ar[r]^-{q_{Y \otimes X}} & {}^\dual(Y \otimes X) \ar[dd]^{\mathscr{L}}\\
\Omega \otimes X^\dual \otimes Y^\dual \otimes \Omega \ar[d]^-{\cong}_-{\mu^{-1}}\\
\Omega \otimes X^\dual \otimes \Omega \otimes \Omega \otimes Y^\dual \otimes \Omega \ar[r]_-{q_X \otimes q_Y} & {}^\dual X \otimes {}^\dual Y
}\,.
\ee
commutes, or equivalently
after inverting $q_{Y \otimes X}$, 
\be\label{eq:nicepivotality}
\begin{tikzpicture}[very thick,scale=0.8,color=blue!50!black, baseline, rotate=180]
\draw[line width=1pt] 
(2,-2) node[line width=0pt, left] (Y) {{\small $Y$}}
(3,-2) node[line width=0pt, left] (X) {{\small $X$}}
(-1,3) node[line width=0pt, left] (XY) {{\small $Y\otimes X$}}; 
\draw[redirected] (1,0) .. controls +(0,1) and +(0,1) .. (2,0);
\draw[redirected] (0,0) .. controls +(0,2) and +(0,2) .. (3,0);
\draw[redirected] (-1,0) .. controls +(0,-1) and +(0,-1) .. (0.5,0);
\draw (2,0) -- (2,-2);
\draw (3,0) -- (3,-2);
\draw[dotted] (0,0) -- (1,0);
\draw (-1,0) -- (-1,3);
\draw[very thin, color=blue!55!white, decorate, decoration={snake, amplitude=0.2mm, segment length=1.0mm}] (-1.25,3) -- (-1.25,-2); 
\draw[very thin, color=blue!55!white, decorate, decoration={snake, amplitude=0.2mm, segment length=1.0mm}] (-0.35,-0.76) .. controls +(-0.7,0.5) and +(-2.5,-0.2) .. (1.5,2);
\draw[very thin, color=blue!55!white, decorate, decoration={snake, amplitude=0.2mm, segment length=1.0mm}] (1.5,1.5) .. controls +(0.7,0) and +(0,0.2) .. (3.0,1.05);
\draw[very thin, color=blue!55!white, decorate, decoration={snake, amplitude=0.2mm, segment length=1.0mm}] (3.0,1.05) .. controls +(0.5,-0.7) and +(2.0,0.2) .. (1.5,2);
\draw[very thin, color=blue!55!white, decorate, decoration={snake, amplitude=0.2mm, segment length=1.0mm}] (2.85,-2) .. controls +(0,1.3) and +(0.8,0.05) .. (1.5,0.75); 
\end{tikzpicture}
\; = \;
\begin{tikzpicture}[very thick,scale=0.8,color=blue!50!black, baseline, rotate=180]
\draw[line width=1pt] 
(-2,-2) node[line width=0pt, left] (X) {{\small $Y$}}
(-1,-2) node[line width=0pt, left] (Y) {{\small $X$}}
(2,3) node[line width=0pt, left] (XY) {{\small $Y\otimes X$}}; 
%
%
\draw[redirected] (0,0) .. controls +(0,1) and +(0,1) .. (-1,0);
\draw[redirected] (1,0) .. controls +(0,2) and +(0,2) .. (-2,0);
\draw[redirected] (2,0) .. controls +(0,-1) and +(0,-1) .. (0.5,0);
\draw (-1,0) -- (-1,-2);
\draw (-2,0) -- (-2,-2);
\draw[dotted] (0,0) -- (1,0);
\draw (2,0) -- (2,3);
\draw[very thin, color=blue!55!white, decorate, decoration={snake, amplitude=0.2mm, segment length=1.0mm}] (-1.15,-2) .. controls +(0,1.0) and +(-1.0,0.1) .. (-0.4,0.75); 
\draw[very thin, color=blue!55!white, decorate, decoration={snake, amplitude=0.2mm, segment length=1.0mm}] (-2.15,-2) .. controls +(0,1.5) and +(-2.3,-0.1) .. (-0.4,1.5); 
\draw[very thin, color=blue!55!white, decorate, decoration={snake, amplitude=0.2mm, segment length=1.0mm}] (1.25,-0.76) .. controls +(0.7,0.5) and +(0,-1) .. (1.85,3);
\end{tikzpicture}
.
\ee
\end{definition}




Now we specialise to the bicategory $\LG$. For an object $(k[x_1,\ldots,x_n], W)$ we define $\Omega_W = \Delta_W[n]$ and the $2$-isomorphism $\mu: \Omega_W \otimes \Omega_W \lto \Delta_W$ to be the composite
$$
\xymatrix{
\Delta_W[n] \otimes \Delta_W[n] \ar[r]^-{\cong} & \big( \Delta_W \otimes \Delta_W[n] \big)[n] \ar[r]^-{\cong} & \Delta_W \otimes \Delta_W[n][n] = \Delta_W \otimes \Delta_W \ar[r]^-{\cong}_-{\lambda_{\Delta} = \rho_{\Delta}} & \Delta_W\,.
}
$$
Given a $1$-morphism $X: (R,W) \lra (S,V)$ represented by a finite-rank matrix factorisation with $R=k[x_1,\ldots,x_n]$ and $S=k[z_1,\ldots,z_m]$ there are canonical isomorphisms
\[
X^\dual = R[n] \otimes_R X^{\vee} \cong \Delta_W[n] \otimes_R X^{\vee}, \qquad {}^\dual X = X^{\vee} \otimes_S S[m] \cong X^{\vee} \otimes_S \Delta_V[m]\,.
\]
The dual $Y^{\vee}$ of an infinite rank matrix factorisation in $\LG(W, V)$ is defined as follows: choose a finite rank factorisation $X$ together with an idempotent $e: X \lto X$ splitting to $Y$, and define $Y^{\vee}$ to be the splitting of $e^{\vee}$. With this structure, $\LG$ is a graded bicategory.



\begin{proposition}\label{prop:monoidalproperty} The bicategory $\LG$ is graded pivotal.
\end{proposition}
\begin{proof}
It suffices to consider $1$-morphisms which are finite-rank matrix factorisations: if \eqref{eq:nicepivotalitysq} commutes for a pair $Y,X$ then by functoriality it commutes for any pair $Y',X'$ where $Y'$ is a summand of $Y$ and $X'$ is a summand of $X$. So let $W \in R = k[x_1,\ldots,x_n]$, $V \in S = k[z_1,\ldots,z_m]$, $U \in T = k[y_1,\ldots,y_p]$ be potentials and suppose that we are given a pair of $1$-morphisms $X \in \hmf(k[x,z],V-W)$ and $Y\in \hmf(k[y,z],U-V)$.

To evaluate the diagrams in \eqref{eq:nicepivotality} we need to use the evaluation map for the infinite-rank matrix factorisation $Y \otimes X$ of $U - W$ over $k[x,y]$. The prescription given in the proof of Theorem \ref{theorem:adjointbicategoryyeah} is to define these evaluation maps by presenting $Y \otimes X$ as a summand of something of finite rank. In Appendix~\ref{section:symidem} we show that the quotient map
\be\label{eq:splitmonokappa}
\kappa: Y \otimes X \lto \bar Y \otimes \bar X := (Y \otimes_{k[z]} X) \otimes_{k[z]} k[z]/(\partial_{z_1}V, \ldots, \partial_{z_m}V)
\ee
is a split monomorphism in $\HMF(k[x,y], U - W)$. Denote the left inverse to $\kappa$ by $\rho$. The idempotent $e = \kappa \circ \rho$ on $\bar Y \otimes \bar X$ dualises to an idempotent $e^\vee$ on $(\bar Y \otimes \bar X)^{\vee} = \Hom_{k[x,y]}( \bar Y \otimes \bar X, k[x,y])$ and we define $Z$ to be the splitting of this idempotent in $\LG(U,W)$. Thus there are morphisms $\kappa', \rho'$ as in the diagram
\[
\xymatrix@C+2pc{
(\bar Y \otimes \bar X)^\vee \ar@<-1ex>[r]_-{\rho'} & Z \ar@<-1ex>[l]_-{\kappa'} 
}
\]
with $\rho' \kappa' = 1_{Z}$ and $\kappa' \rho' = e^{\vee}$. With this notation we explain in the proof of Theorem \ref{theorem:adjointbicategoryyeah} why $(Y \otimes X)^\dual := R[n] \otimes_R Z$ and ${}^\dual( Y \otimes X) := Z \otimes_T T[p]$ are the adjoints of $Y \otimes X$.

Omitting wiggly lines, the left-hand side of \eqref{eq:nicepivotality} is then (by definition) $\kappa'$ followed by
\be\label{eq:leftmonoidaldetail}
\begin{tikzpicture}[very thick,scale=1.0,color=blue!50!black, baseline, rotate=180]
\draw[line width=1pt] 
(2,-2) node[line width=0pt, right] (Y) {{\small $Y$}}
(3,-2) node[line width=0pt, right] (X) {{\small $X$}}
(-1,3) node[line width=0pt, right] (XY) {{\small $\bar Y\otimes \bar X$}}; 
\draw[redirected] (1,0) .. controls +(0,1) and +(0,1) .. (2,0);
\draw[redirected] (0,0) .. controls +(0,2) and +(0,2) .. (3,0);
\draw[redirected] (-1,0) .. controls +(0,-1) and +(0,-1) .. (0.5,0);
\draw (2,0) -- (2,-2);
\draw (3,0) -- (3,-2);
\draw[dotted] (0,0) -- (1,0);
\draw (-1,0) -- (-1,3);
%
%
%
%
%
\fill (0.5,0) circle (2.2pt) node[below] {{\small$\kappa$}};
\coordinate (lambdainv) at ($ (-1,2.2) $);
\fill (lambdainv) circle (2.2pt) node[right] {{\small$\lambda^{-1}$}};
\coordinate (rhoinv) at ($ (2.3,1.28) $);
\fill (rhoinv) circle (2.2pt) node[above] {{\small$\rho^{-1}$}};
\coordinate (rho) at ($ (2,-1.4) $);
\fill (rho) circle (2.2pt) node[below right] {{\small$\rho$}};
\draw[dashed] (rho)  .. controls +(-1,0.3) and +(0,-1) ..  (-0.31,-0.72);
\draw[dashed] (rhoinv)  .. controls +(-0.1,-0.5) and +(0,0.5) ..  (1.45,0.74);
\draw[dashed] (1.45,1.49)  .. controls +(-0.1,0.5) and +(0,-0.5) ..  (lambdainv);
\end{tikzpicture}
\, .
\ee
We denote this morphism by $\bar{\mathscr{R}}: (\bar{Y} \otimes \bar{X})^\dual \lto X^\dual \otimes Y^\dual$. Similarly the right-hand side of~\eqref{eq:nicepivotality} is $\kappa'$ followed by a morphism $\bar{\mathscr{L}}: {}^\dual(\bar{Y} \otimes \bar{X}) \lto {}^\dual X \otimes {}^\dual Y$. The first step is to compute $\bar{\mathscr{R}}$ and $\bar{\mathscr{L}}$.


Let $\{e_i\}_i$ be a $k[x,z]$-basis of~$X$, $\{f_j\}_j$ a $k[y,z]$-basis of~$Y$, and $\{g_\alpha\}_\alpha$ a $k$-basis of $k[z]/(\partial V)$. Then $\{ f_j \otimes g_\alpha e_i \}_{i,j,\alpha}$ is a $k[x,y]$-basis of $\bar{Y} \otimes \bar{X}$. Using the expressions~\eqref{eq:formulalambdainv} and~\eqref{eq:alt_formula_coev} together with Lemma~\ref{lemma:atshufat} we find that $\lambda^{-1}$ and $\widetilde\coev_X$ in \eqref{eq:leftmonoidaldetail} send $(f_q \otimes g_\alpha e_p)^* \in (\bar Y \otimes \bar X)^\vee$ to
$$
\sum_j (-1)^{|e_j| + n|f_q| + n|e_p|} \, e_j^* \otimes \varepsilon \Psi \overset{\rightarrow}{\lAt}_{R} \big(X\otimes (\bar Y \otimes \bar X)^\dual \big)^n \big(e_j\otimes (f_q \otimes g_\alpha e_p)^*\big) \, . 
$$
Similarly with~\eqref{eq:formula_rhoinverse} we compute that $\rho^{-1}$ and $\widetilde\coev_Y$ map this to
\begin{align*}
&\sum_{j,k} 
(-1)^{(m+1)|f_k| + |e_j| + n|f_q| + n|e_p|} \, \\
& \qquad\; \cdot \varepsilon \Psi \At_S(X^\dual \otimes Y^\dual)^m (e_j^* \otimes f_k^*)
\otimes \varepsilon \Psi \overset{\rightarrow}{\lAt}_R\big((Y\otimes X)\otimes( \bar Y \otimes \bar X)^\dual\big)^n \big( f_k \otimes e_j \otimes (f_q \otimes g_\alpha e_p)^*\big) \, . 
\end{align*}
Next we apply the morphism~$\kappa$. By naturality of the Atiyah class (Lemma \ref{lemma:atiyahnat}) the sole effect of~$\kappa$ on the above is to replace the Atiyah class of $(Y\otimes X)\otimes( \bar Y \otimes \bar X)^\dual$ by that of $(\bar Y\otimes \bar X)\otimes( \bar Y \otimes \bar X)^\dual$. This means that we are in precisely the same situation as in the proof of the Zorro move for $\bar Y \otimes \bar X$, i.\,e.~after applying $\widetilde\eval_{\bar Y \otimes \bar X}$ and~$\rho$ we are left with
\be\label{eq:valueofLpivot}
\bar{\mathscr{R}}( (f_q \otimes g_\alpha e_p )^* ) = (-1)^{m|f_q|} \, 
(1 \otimes g_\alpha^*) \varepsilon \Psi \At_S(X^\dual \otimes Y^\dual)^m (e_p^* \otimes f_q^*) \, . 
\ee
To explain the notation: the output of $\varepsilon$ is an element of $(X^\dual \otimes Y^\dual) \otimes_{S_1} \Se$ and to this we apply the functional $1 \otimes g_\alpha^*$ which sends $s \otimes t \in \Se$ to $s \cdot g_\alpha^*(t)$. Composing with the other maps in the square \eqref{eq:nicepivotalitysq} and using again naturality of the Atiyah class,
\be
(q_X \otimes q_Y)\mu^{-1}\bar{\mathscr{R}}( (f_q \otimes g_\alpha e_p )^* ) = (-1)^{m|f_q|+mp} \, 
(1 \otimes g_\alpha^*) \varepsilon \Psi \At_S({}^\dual X \otimes {}^\dual Y)^m (e_p^* \otimes f_q^*) \, .
\ee
Similarly one computes that
\be\label{eq:valueofRpivot}
\bar{\mathscr{L}}q_{Y \otimes X}( (f_q \otimes g_\alpha e_p)^* ) = (-1)^{m|f_q|+mp} \, 
(g_\alpha^* \otimes 1) \varepsilon \Psi \overset{\rightarrow}{\lAt}_S ({}^\dual X \otimes {}^\dual Y)^m (e_p^* \otimes f_q^*) \, . 
\ee
We prove that these maps are homotopic by showing that they are \textsl{equal} modulo the $\partial_{z_i} V$ (to see that this is sufficient, use a split monomorphism in $\HMF(k[x,y], U - W)$ like the one in \eqref{eq:splitmonokappa}). Since $\bar{X}^\dual \otimes \bar{Y}^\dual$ has a $k[x,y]$-basis consisting of tensors $e_i^* \otimes g_\alpha f_j^*$, to prove this equality modulo the $\partial_{z_i} V$ it suffices to prove both sides are equal after applying $g_\beta^*$ for every $\beta$. Thus we have to show
\be
(g_\beta^* \otimes g_\alpha^*) \varepsilon \Psi \At_S({}^\dual X \otimes {}^\dual Y)^m (e_p^* \otimes f_q^*) = (g_\alpha^* \otimes g_\beta^*) \varepsilon \Psi \overset{\rightarrow}{\lAt}_S ({}^\dual X \otimes {}^\dual Y)^m (e_p^* \otimes f_q^*)\,.
\ee
Let $\tau: \Se \lto \Se$ be the isomorphism $\tau(s \otimes s') = s' \otimes s$. We compute using the formulas of Section~\ref{section:atiyahclasses}, with $D$ the differential on ${}^\dual X \otimes {}^\dual Y$,
\begin{align}
&(g_\alpha^* \otimes g_\beta^*) \varepsilon \Psi \overset{\rightarrow}{\lAt}_S ({}^\dual X \otimes {}^\dual Y)^m (e_p^* \otimes f_q^*) \nonumber \\
&= (g_\beta^* \otimes g_\alpha^*) \tau \Big( (-1)^{mp+m|e_p| + m|f_q|} \sum_{p',q'} \left\{ \partial_{[m]} D \cdots \partial_{[1]} D \right\}_{p'q', pq} e_{p'}^* \otimes f_{q'}^* \Big) \nonumber \\
&= (g_\beta^* \otimes g_\alpha^*) (-1)^{mp+m|e_p| + m|f_q|} \sum_{p',q'} \left\{ \partial^{\,\sigma}_{[1]} D \cdots \partial^{\,\sigma}_{[m]} D \right\}_{p'q', pq} e_{p'}^* \otimes f_{q'}^* \label{eq:laststeppivotality}
\end{align}
using the notation of Appendix \ref{app:variableordering} for the variable ordering $z_{\sigma(1)}, \ldots, z_{\sigma(m)}$ with $\sigma(i) = m - i + 1$. But we recognise this as the explicit form of $\At_S$ evaluated using the $\Psi$ for this reversed variable ordering; by Appendix \ref{app:variableordering} the value of a diagram is independent of variable ordering, so~\eqref{eq:laststeppivotality} equals
\be
(g_\beta^* \otimes g_\alpha^*) \varepsilon \Psi \At_S({}^\dual X \otimes {}^\dual Y)^m( e_p^* \otimes f_q^* )
\ee
which completes the proof.
\end{proof}

\begin{remark}\label{remark:notpivotal} The sub-bicategory $\LG^{\text{even}}$ of potentials depending on an even number of variables is pivotal in the usual sense. More precisely, there is always an isomorphism
\be\label{delta}
\delta_X = 
\begin{tikzpicture}[very thick,scale=1.0,color=blue!50!black, baseline=.4cm]
\draw[line width=0pt] 
(3,0.5) node[line width=0pt] (D) {}
(2,0.5) node[line width=0pt] (s) {}; 
\draw[directed] (D) .. controls +(0,-1) and +(0,-1) .. (s);
\fill (2.6,-0.3) circle (0pt) node[below] {{\small $\coev_{X^\dual}$}};
\draw (2,0.32) -- (2,0.68);
\draw[line width=0pt] 
(1,0.5) node[line width=0pt] (D) {}
(2,0.5) node[line width=0pt] (s) {}; 
\draw[directed] (D) .. controls +(0,+1) and +(0,+1) .. (s);
\fill (1.6,1.23) circle (0pt) node[above] {{\small $\widetilde\eval_{X}$}};
\draw (3,0.35) -- (3,1.5)
node[above] {{{\small${}^\dual( X^{\dual}$)}}};
\draw (1,0.65) -- (1,-0.5)
node[below] {{{\small$X$}}};
\end{tikzpicture} 
: X\stackrel{\cong}{\lra} {}^\dual (X^\dual) \cong X^{\vee\vee}[n+n] = X^{\vee\vee}
\, 
\ee
and if the number of variables in the regions on either side of $X$ agree mod $2$, for example when $X$ is an endomorphism of an object of $\LG$, then the left and right adjoints actually are isomorphic and $\delta_X$ gives a pivotal structure, i.\,e., a monoidal isomorphism between $X$ and $X^{\dual \dual}$.
\end{remark}

\section{Defect action on bulk fields}\label{sec:defectaction}

In a graded bicategory there are natural operators defined on $2$-morphisms $\Delta \lto \Omega$ constructed by ``capturing'' such $2$-morphisms inside loops labelled by 1-morphisms. Below we present the details for the bicategory $\LG$ and give a natural interpretation in terms of defect actions on bulk fields in Landau-Ginzburg models. In the context of derived categories of coherent sheaves the space of $2$-morphisms $\Delta \lto \Omega$ was identified with Hochschild homology in \cite{ct1007.2679} and from this perspective the operators discussed below are a diagrammatic presentation of maps induced on Hochschild homology of matrix factorisation categories by integral functors, as studied in \cite{pv1002.2116}. 

Recall that for an object $(k[x_1,\ldots,x_n], W)$ of $\LG$ we have defined $\Omega_W = \Delta_W[n]$. In this section when we write $\uHom$ and $\uEnd$ to mean 2-morphisms in $\LG$. Given a $1$-morphism $X: W \lto V$ there are associated $k$-linear \textsl{defect operators} 
$$
\mathcal D_l(X): \uHom(\Delta_V, \Omega_V) \longrightarrow \uHom(\Delta_W, \Omega_W) \, , \qquad
\mathcal D_r(X): \uHom(\Delta_W, \Omega_W) \longrightarrow \uHom(\Delta_V, \Omega_V)
$$
defined as follows. For $\phi\in \uHom(\Delta_V, \Omega_V)$ and $\psi\in \uHom(\Delta_W, \Omega_W)$ we set 
\be\label{leftandrightdefectaction}
\mathcal D_l(X)(\phi) =
\begin{tikzpicture}[very thick,scale=0.6,color=blue!50!black, baseline,>=stealth]
\nicecolourscheme (0,0) circle (3.5);
\fill (-2.2,-2.2) circle (0pt) node[white] {{\small$W$}};
\nicepalecolourscheme (0,0) circle (2);
\fill (-1.1,-1.1) circle (0pt) node[white] {{\small$V$}};
\draw[very thin, color=blue!85!black, color=blue!55!white, decorate, decoration={snake, amplitude=0.2mm, segment length=1.0mm}] (135:0)  .. controls +(-0.25,1.0) and +(-2.25,-1.25) .. (90:2);
\draw[very thin, color=blue!85!black, color=blue!55!white, decorate, decoration={snake, amplitude=0.2mm, segment length=1.0mm}] (270:2) .. controls +(-3,0) and +(0,-1) .. (135:3.5); 
\draw (0,0) circle (2);
\fill (180:2) circle (0pt) node[right] {{\small$X$}};
\draw[<-, very thick] (0.100,-2) -- (-0.101,-2) node[above] {}; 
\draw[<-, very thick] (-0.100,2) -- (0.101,2) node[below] {}; 
\fill (135:0) circle (3.3pt) node[right] {{\small$\phi$}};
\fill (-50:2) circle (3.3pt) node[right] {{\small$\lambda^{-1}$}};
\draw[dashed] (135:0) .. controls +(0,-1) and +(-0.5,1) .. (-50:2);
\end{tikzpicture} 
\, , \qquad 
\mathcal D_r(X)(\psi) = 
\begin{tikzpicture}[very thick,scale=0.6,color=blue!50!black, baseline,>=stealth]
\nicepalecolourscheme (0,0) circle (3.5);
\fill (2.2,-2.2) circle (0pt) node[white] {{\small$V$}};
\nicecolourscheme (0,0) circle (2);
\fill (1.1,-1.1) circle (0pt) node[white] {{\small$W$}};
\draw[very thin, color=blue!85!black, color=blue!55!white, decorate, decoration={snake, amplitude=0.2mm, segment length=1.0mm}] (135:0)  .. controls +(0.25,1.0) and +(2.25,-1.25) .. (90:2);
\draw[very thin, color=blue!85!black, color=blue!55!white, decorate, decoration={snake, amplitude=0.2mm, segment length=1.0mm}] (270:2) .. controls +(3,0) and +(0,-1) .. (45:3.5); 
\draw (0,0) circle (2);
\fill (0:2) circle (0pt) node[left] {{\small$X$}};
\draw[<-, very thick] (0.100,2) -- (-0.101,2) node[above] {}; 
\draw[<-, very thick] (-0.100,-2) -- (0.101,-2) node[below] {}; 
\fill (135:0) circle (3.3pt) node[left] {{\small$\psi$}};
\fill (230:2) circle (3.3pt) node[left] {{\small$\rho^{-1}$}};
\draw[dashed] (135:0) .. controls +(0,-1) and +(0.5,1) .. (230:2);
\end{tikzpicture} 
\, , 
\ee
\begin{align}
\mathcal D_l(X)(\phi) & = \eval_X \circ (1_{X^\dual} \otimes ((\phi \otimes 1_X) \circ \lambda_X^{-1})) \circ \widetilde\coev_X \, , \nonumber \\ 
\mathcal D_r(X)(\psi) & = \widetilde\eval_X \circ (( (1_X \otimes \psi) \circ \rho_{X}^{-1}) \otimes 1_{{}^\dual X}) \circ \coev_X \, . \label{defectactionwrittenout}
\end{align}
For $W \in k[x_1,\ldots,x_n]$ there is a canonical isomorphism \cite{d0904.4713}
\[
\uHom(\Delta_W, \Omega_W) \cong \begin{cases} 0 & n \text{ odd}\\ k[x]/(\partial W) & n \text{ even} \end{cases}
\]
so the operators $\mathcal D_l(X)$ and $\mathcal D_r(X)$ are either zero, or they map between $\uEnd(\Delta_W)$ and $\uEnd(\Delta_V)$. In the latter case, so where $W$ and $V$ are both functions of an even number of variables, we define the left and right \textsl{quantum dimensions} of $X$ to be $\operatorname{dim}_l(X) = \mathcal D_l(X)(1)$ and $\operatorname{dim}_r(X) = \mathcal D_r(X)(1)$. 

\begin{remark} When $k$ is a field, Dyckerhoff proves in~\cite{d0904.4713} that $\uEnd(\Delta_W) = k[x]/(\partial W)$ is the Hochschild cohomology of $\hmf(k[\![x]\!], W)$. This also holds in the current context, where $k$ is an arbitrary ring; to see this, use a variant of the calculation \eqref{eq:isogivescoev2} with $R = k[x]$ to see that $\End_{\Re}(\Delta_W)$ is quasi-isomorphic to a shift of the complex
\[
\Big( \bigwedge \bigoplus_{i=1}^n R \theta_i, \sum_{i=1}^n \partial^{x,x'}_{[i]} W|_{x=x'} \theta_i \wedge (-) \Big) .
\]
Since $\partial^{x,x'}_{[i]} W|_{x=x'} = \partial_{x_i} W$ this is the Koszul complex of the partial derivatives, which by the definition of a potential is quasi-isomorphic to $k[x]/(\partial_{x_i} W)$.

When $k = \nC $ this space also describes bulk fields of Landau-Ginzburg models with potential~$W$, it is a commutative Frobenius algebra whose nondegenerate pairing \cite{kl0305, m0912.1629}
\be\label{bulktopmet}
\langle \phi, \psi \rangle_W = 
\Res_{k[x]/k} \left[ \frac{\phi\,  \psi \, \underline{\operatorname{d}\!x}}{\partial_{x_1}W, \ldots, \partial_{x_n} W}\right]
\ee
describes 2-point correlators on the sphere, see also Section~\ref{sec:ocTFT}. Furthermore, matrix factorisations of $V-W$ describe defect conditions between different Landau-Ginzburg models. Hence the maps~\eqref{leftandrightdefectaction} have the natural interpretation in terms of defect operators on bulk fields: for example, a bulk field~$\phi$ in the theory with potential~$V$ is mapped to the bulk field $\mathcal D_l(X)(\phi)$ in the theory with potential~$W$ by wrapping around its insertion on the worldsheet a defect line labelled by~$X$, and then collapsing this loop onto the insertion point. This limiting process is nonsingular as the bicategory $\LG$ describes the purely topological sector of Landau-Ginzburg models. 
\end{remark}

Using the ``folding trick'', which relates defects to boundary conditions in a product theory, one can argue for explicit expressions for $\mathcal D_l(X)$ and $\mathcal D_r(X)$. This was done in~\cite{cr1006.5609} for the case $V=W$ in one variable. Here we use our adjunction formulas to directly prove it for the general case. In addition, we want to consider operators $\mathcal D_l^{\Phi}(X)$ and $\mathcal{D}_r^{\Phi}(X)$ defined by inserting a $2$-morphism $\Phi: X \lra X$ in the obvious way on the circles in \eqref{leftandrightdefectaction}, that is replacing $1_X$ by $\Phi$ in \eqref{defectactionwrittenout}.

\begin{proposition}\label{prop:defectaction}
Suppose that $m$ and $n$ are even. Then for any $1$-morphism $X: W \lto V$ given by a finite-rank matrix factorisation
\[
X\in \hmf(k[x_1,\ldots,x_n,z_1,\ldots,z_m], V-W)
\]
together with $\Phi \in \uEnd(X)$, $\phi\in \uEnd(\Delta_V)$ and $\psi\in \uEnd(\Delta_W)$ we have
\begin{align*}
\mathcal D_l^\Phi(X)(\phi) & 
= (-1)^{{{n+1}\choose 2}}\Res_{k[x,z]/k[x]} \left[ \frac{\phi(z) \str\big( \Phi \, \partial_{x_1} d_{X}\ldots \partial_{x_n} d_{X} \,  \partial_{z_1} d_{X}\ldots \partial_{z_m} d_{X} \big) \underline{\operatorname{d}\! z}}{\partial_{z_1} V, \ldots, \partial_{z_m} V} \right] , \\
\mathcal D_r^\Phi(X)(\psi) & = (-1)^{{{m+1}\choose 2}}  \Res_{k[x,z]/k[z]} \left[ \frac{\psi(x) \str\big( \Phi \, \partial_{x_1} d_{X}\ldots \partial_{x_n} d_{X} \, \partial_{z_1} d_{X}\ldots \partial_{z_m} d_{X} \big) \underline{\operatorname{d}\! x}}{\partial_{x_1} W, \ldots, \partial_{x_n} W} \right] . 
\end{align*}
\end{proposition}

\begin{proof}
We treat the case of $\mathcal D_r^\Phi(X)$ in detail, the argument for $\mathcal D_l^\Phi(X)$ is similar. Since $\uEnd(\Delta_W) = k[x]/(\partial W)$ and $\uEnd(\Delta_V) = k[z]/(\partial V)$ we are free to identify the variables on both sides of the unit 1-endomorphisms at appropriate places. In particular we identify $\psi$ with a polynomial.

In the lower part of the expression for $\mathcal D_r^\Phi(X)(\psi)$ in~\eqref{leftandrightdefectaction} we have (using \eqref{eq:coev_divideddiff})
$$
\coev_X(1) = \sum_{i,j} (-1)^{{m+1\choose 2}} \big\{ \partial^{z,z'}_{[1]} d_X \ldots \partial^{z,z'}_{[m]} d_X \big\}_{ij} e_i \otimes e_j^* \, . 
$$
Next we apply~$\Phi$, $\psi$ and $\widetilde\eval_X$. Since the latter maps back to~$\Delta_V$ we may set $\partial_{[i]} d_X = \partial_{z_i} d_X$. Thus
$$
\mathcal D_r^\Phi(X)(\psi) = 
(-1)^{{m+1\choose 2}} \Res_{k[x,z]/k[z]} \left[ \frac{\psi(x) \str\big( \Phi \, \partial_{x_1} d_{X}\ldots \partial_{x_n} d_{X} \, \partial_{z_1} d_{X}\ldots \partial_{z_m} d_{X} \big) \underline{\operatorname{d}\! x}}{\partial_{x_1} W, \ldots, \partial_{x_n} W} \right] + \mathcal O(\theta) \, . 
$$
Here we collectively denote the contributions from $\widetilde\eval_X$ of non-zero $\mathbb{Z}$-degree in the Koszul complex $\Delta_V$ by $\mathcal O(\theta)$. Since we know that $\mathcal D_r^\Phi(X)(\psi)$ is a morphism in $\uEnd(\Delta_V) = k[z]/(\partial V)$ it follows that $\mathcal O(\theta)$ must be null-homotopic, thus concluding the proof. 
\end{proof}


\begin{corollary}\label{corollary:uses_transitivity}
$\mathcal D_l(X)$ and $\mathcal D_r(X)$ are adjoint with respect to the pairings~\eqref{bulktopmet}, i.\,e.~we have
\be\label{Dadjoint}
\big\langle \mathcal D_l(X)(\phi), \psi \big\rangle_W = \big\langle \phi , \mathcal D_r(X)(\psi) \big\rangle_V
\ee
for all $\phi\in \uEnd(\Delta_V)$ and $\psi\in \uEnd(\Delta_W)$. 
\end{corollary}

\begin{proof}
This directly follows from the explicit expressions for $\mathcal D_l(X)$, $\mathcal D_r(X)$ and $\big\langle -, -\big\rangle_W$, $\big\langle -, -\big\rangle_V$, together with the transitivity rule for residues (Proposition \ref{prop:transitivity}). 
\end{proof}

\begin{remark}
We recall the physical interpretation of the relation~\eqref{Dadjoint}. Both sides of this equation are 2-point correlators on the Riemann sphere, with a defect line labelled by~$X$ wrapped around counterclockwise the bulk field~$\phi$, or wrapped around~$\psi$ in clockwise fashion. That both correlators should be equal follows from the fact that the topological defect can be moved around the sphere at no cost: 
$$
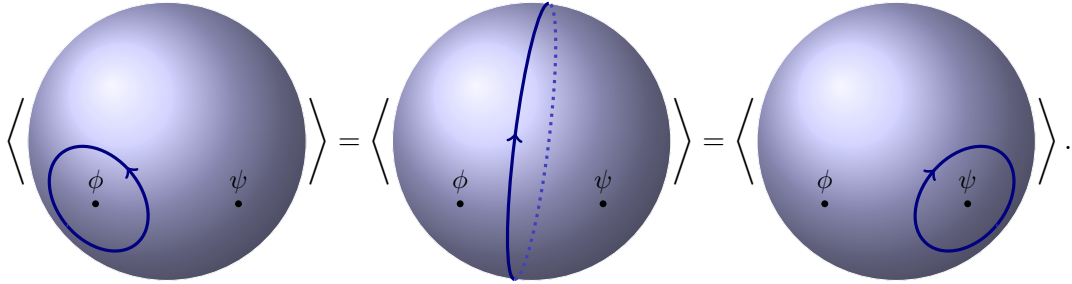
\left\langle
\begin{tikzpicture}[baseline=-0.1cm]
\def\R{1.85}
\def\angEl{45}
\filldraw[ball color= white!77!blue,draw=white] (0,0) circle (\R);
\DrawLatitudeCircleU[\R,rotate=130,very thick, blue]{65}
\fill (-0.95,-0.83) circle (1.3pt) node[above] {{\small$\phi$}}; 
\fill (0.95,-0.83) circle (1.3pt) node[above] {{\small$\psi$}}; 
\end{tikzpicture}
\right\rangle
=
\left\langle
\begin{tikzpicture}[baseline=-0.1cm]
\def\R{1.85}
\def\angEl{45}
\filldraw[ball color= white!77!blue,draw=white] (0,0) circle (\R);
\DrawLongitudeCircle[\R]{80}
\fill (-0.95,-0.83) circle (1.3pt) node[above] {{\small$\phi$}}; 
\fill (0.95,-0.83) circle (1.3pt) node[above] {{\small$\psi$}}; 
\end{tikzpicture}
\right\rangle
=
\left\langle
\begin{tikzpicture}[baseline=-0.1cm]
\def\R{1.85}
\def\angEl{45}
\filldraw[ball color= white!77!blue,draw=white] (0,0) circle (\R);
\DrawLatitudeCircle[\R,rotate=-130, very thick, blue]{65}
\fill (-0.95,-0.83) circle (1.3pt) node[above] {{\small$\phi$}}; 
\fill (0.95,-0.83) circle (1.3pt) node[above] {{\small$\psi$}}; 
\end{tikzpicture}
\right\rangle .
$$
\end{remark}

Defect operators satisfy the following compatibility conditions: 

\begin{proposition}
\label{prop:defectactionproperties}
Let $W\in k[x_1,\ldots,x_n]$, $V\in k[z_1,\ldots,z_m]$, $U\in k[y_1,\ldots,y_p]$ be potentials and $Y\in \hmf(k[y,z],U-V)$, $X\in \hmf(k[x,z],V-W)$. 
\begin{enumerate}
\item $\mathcal D_l(\Delta) = 1 = \mathcal D_r(\Delta)$. 
\item $\mathcal D_l(X) = \mathcal D_r(X^\vee)$ and $\mathcal D_r(X) = \mathcal D_l(X^\vee)$.
\item $\mathcal D_l(X) \circ \mathcal D_l(Y) = \mathcal D_l(Y \otimes X)$ and $\mathcal D_r(Y) \circ \mathcal D_r(X) = \mathcal D_r(Y \otimes X)$. 
\item $\mathcal D_l(X[a]) = (-1)^a \mathcal D_l(X)$ and $\mathcal D_r(X[a]) = (-1)^a \mathcal D_r(X)$ for $a \in \nZ$.
\end{enumerate}
Analogous results hold for the operators $\mathcal D^\Phi_l, \mathcal D^\Phi_r$ decorated with endomorphisms of $X$ and $Y$.
\end{proposition}

\begin{proof}
(i) is a direct consequence of Proposition~\ref{prop:defectaction} and the results of~\cite[Section\,5.1, Eq.\,(21)]{kr0405232} while (iv) is trivial.
(Note that we do not see that part~(i) is a special case of part (iii) because the latter only implies that the defect action of~$\Delta$ is an idempotent -- which acts as the identity on~$1_\Delta$, but possibly not on all of $\End(\Delta)$.)

(ii): By the argument of Remark~\ref{remark:indeptlambda} reordering the partial derivatives of $d_{X^\vee}$ leaves $\mathcal D_r(X^\vee)$ invariant up to a sign. Together with $(d_{X^\vee})_{ij} = (-1)^{|e_i|}(d_X)_{ji}$ one checks that the claim then follows from Proposition~\ref{prop:defectaction}. 

(iii): We give the proof for $\mathcal D_l$, the other case is analogous. 
The basic argument is a standard calculation with string diagrams. There are however slight modifications needed due to the fact that $\LG$ is graded pivotal as discussed in Section~\ref{sec:wiggliesandsigns}. 
In fact we observe that the following proof, and hence part (iii) of the proposition, holds in any graded pivotal bicategory. 

We start by writing $\mathcal D_l(X) \circ \mathcal D_l(Y)$ applied to some element of $\uHom(\Delta_V, \Omega_V)$ as
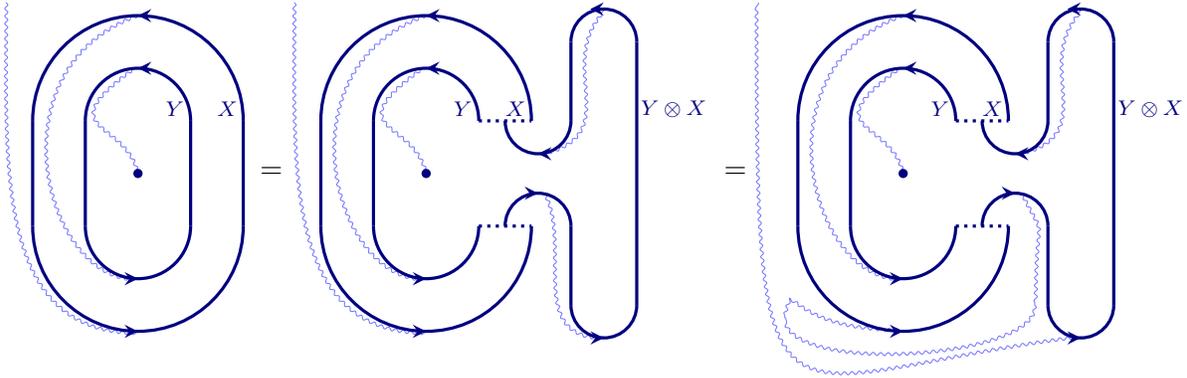
$$
\begin{tikzpicture}[very thick,scale=0.7,color=blue!50!black, baseline=-0.75cm,>=stealth]
%
\draw[very thin, color=blue!85!black, color=blue!55!white, decorate, decoration={snake, amplitude=0.2mm, segment length=1.0mm}] (0,-1)  .. controls +(-0.25,1.0) and +(-1.75,-1.25) .. (0,1);
\draw[very thin, color=blue!55!white, decorate, decoration={snake, amplitude=0.2mm, segment length=1.0mm}] 
(0,2) .. controls +(-2.75,-1.0) and +(-1.75,0) .. (0,-3); 
\draw[very thin, color=blue!55!white, decorate, decoration={snake, amplitude=0.2mm, segment length=1.0mm}] 
(-2.5,2.25) .. controls +(0,-3.0) and +(-2.75,0) .. (0,-4); 
%
\draw[
	decoration={markings, mark=at position 0.5 with {\arrow{>}}}, postaction={decorate}
	] 
	(0:1) arc (0:180:1);
\draw[
	decoration={markings, mark=at position 0.5 with {\arrow{>}}}, postaction={decorate}
	] 
	(0:2) arc (0:180:2);
\draw[
	decoration={markings, mark=at position 0.5 with {\arrow{>}}}, postaction={decorate}
	] 
	(-1,-2) arc (180:360:1);
\draw[
	decoration={markings, mark=at position 0.5 with {\arrow{>}}}, postaction={decorate}
	] 
	(-2,-2) arc (180:360:2);
\draw (-1,0) -- (-1,-2);
\draw (-2,0) -- (-2,-2);
\draw (1,0) -- (1,-2);
\draw (2,0) -- (2,-2);
\draw (0.7,0.23) node {{\scriptsize $Y$}};
\draw (1.7,0.23) node {{\scriptsize $X$}};
\fill (0,-1) circle (2.5pt) node[right] {};
\end{tikzpicture}
\,=\!\!
\begin{tikzpicture}[very thick,scale=0.7,color=blue!50!black, baseline=-0.75cm,>=stealth]
%
\draw[very thin, color=blue!85!black, color=blue!55!white, decorate, decoration={snake, amplitude=0.2mm, segment length=1.0mm}] (0,-1)  .. controls +(-0.25,1.0) and +(-1.75,-1.25) .. (0,1);
\draw[very thin, color=blue!55!white, decorate, decoration={snake, amplitude=0.2mm, segment length=1.0mm}] 
(0,2) .. controls +(-2.75,-1.0) and +(-1.75,0) .. (0,-3); 
\draw[very thin, color=blue!55!white, decorate, decoration={snake, amplitude=0.2mm, segment length=1.0mm}] 
(-2.5,2.25) .. controls +(0,-3.0) and +(-2.75,0) .. (0,-4); 
\draw[very thin, color=blue!55!white, decorate, decoration={snake, amplitude=0.2mm, segment length=1.0mm}] 
(2.25,-0.625) .. controls +(1.5,0.4) and +(-1,-1) .. (3.375,2.1); 
\draw[very thin, color=blue!55!white, decorate, decoration={snake, amplitude=0.2mm, segment length=1.0mm}] 
(2.25,-1.375) .. controls +(0.5,-1) and +(-1.5,0) .. (3.375,-4.1); 
%
\draw[
	decoration={markings, mark=at position 0.5 with {\arrow{>}}}, postaction={decorate}
	] 
	(0:1) arc (0:180:1);
\draw[
	decoration={markings, mark=at position 0.5 with {\arrow{>}}}, postaction={decorate}
	] 
	(0:2) arc (0:180:2);
\draw[
	decoration={markings, mark=at position 0.5 with {\arrow{>}}}, postaction={decorate}
	] 
	(-1,-2) arc (180:360:1);
\draw[
	decoration={markings, mark=at position 0.5 with {\arrow{>}}}, postaction={decorate}
	] 
	(-2,-2) arc (180:360:2);
\draw[
	decoration={markings, mark=at position 0.5 with {\arrow{>}}}, postaction={decorate}
	] 
	(2.75,0) arc (360:180:0.625);
\draw[
	decoration={markings, mark=at position 0.5 with {\arrow{>}}}, postaction={decorate}
	] 
	(1.5,-2) arc (180:0:0.625);
\draw[
	decoration={markings, mark=at position 0.5 with {\arrow{<}}}, postaction={decorate}
	] 
	(2.75,1.5) arc (180:0:0.625);
\draw[
	decoration={markings, mark=at position 0.5 with {\arrow{>}}}, postaction={decorate}
	] 
	(2.75,-3.5) arc (180:360:0.625);	
\draw (-1,0) -- (-1,-2);
\draw (-2,0) -- (-2,-2);
\draw (2.75,-2) -- (2.75,-3.5);
\draw (2.75,0) -- (2.75,1.5);
\draw (4,-3.5) -- (4,1.5);
\draw[dotted] (1,0)[] -- (2,0);
\draw[dotted] (1,-2) -- (2,-2);
\draw (0.7,0.23) node {{\scriptsize $Y$}};
\draw (1.7,0.23) node {{\scriptsize $X$}};
\draw (4.7,0.225) node {{\scriptsize $Y \otimes X$}};
\fill (0,-1) circle (2.5pt) node[right] {};
\end{tikzpicture}
=
\begin{tikzpicture}[very thick,scale=0.7,color=blue!50!black, baseline=-0.75cm,>=stealth]
%
\draw[very thin, color=blue!85!black, color=blue!55!white, decorate, decoration={snake, amplitude=0.2mm, segment length=1.0mm}] (0,-1)  .. controls +(-0.25,1.0) and +(-1.75,-1.25) .. (0,1);
\draw[very thin, color=blue!55!white, decorate, decoration={snake, amplitude=0.2mm, segment length=1.0mm}] 
(0,2) .. controls +(-2.75,-1.0) and +(-1.75,0) .. (0,-3); 
\draw[very thin, color=blue!55!white, decorate, decoration={snake, amplitude=0.2mm, segment length=1.0mm}] 
(2.25,-0.625) .. controls +(1.5,0.4) and +(-1,-1) .. (3.375,2.1); 
\draw[very thin, color=blue!55!white, decorate, decoration={snake, amplitude=0.2mm, segment length=1.0mm}] 
(2.25,-1.375) .. controls +(0.5,-1) and +(0,2) .. (2.5,-3.5); 
\draw[very thin, color=blue!55!white, decorate, decoration={snake, amplitude=0.2mm, segment length=1.0mm}] 
(2.5,-3.5) .. controls +(0,-1) and +(0,-1.1) .. (-2.25,-3.7); 
\draw[very thin, color=blue!55!white, decorate, decoration={snake, amplitude=0.2mm, segment length=1.0mm}] 
(-2.25,-3.7) .. controls +(0.2,0.75) and +(-2.25,0) .. (0,-4); 
\draw[very thin, color=blue!55!white, decorate, decoration={snake, amplitude=0.2mm, segment length=1.0mm}] 
(3.375,-4.1) .. controls +(-6,-0.75) and +(0,-9) .. (-2.75,2.25); 
%
\draw[
	decoration={markings, mark=at position 0.5 with {\arrow{>}}}, postaction={decorate}
	] 
	(0:1) arc (0:180:1);
\draw[
	decoration={markings, mark=at position 0.5 with {\arrow{>}}}, postaction={decorate}
	] 
	(0:2) arc (0:180:2);
\draw[
	decoration={markings, mark=at position 0.5 with {\arrow{>}}}, postaction={decorate}
	] 
	(-1,-2) arc (180:360:1);
\draw[
	decoration={markings, mark=at position 0.5 with {\arrow{>}}}, postaction={decorate}
	] 
	(-2,-2) arc (180:360:2);
\draw[
	decoration={markings, mark=at position 0.5 with {\arrow{>}}}, postaction={decorate}
	] 
	(2.75,0) arc (360:180:0.625);
\draw[
	decoration={markings, mark=at position 0.5 with {\arrow{>}}}, postaction={decorate}
	] 
	(1.5,-2) arc (180:0:0.625);
\draw[
	decoration={markings, mark=at position 0.5 with {\arrow{<}}}, postaction={decorate}
	] 
	(2.75,1.5) arc (180:0:0.625);
\draw[
	decoration={markings, mark=at position 0.5 with {\arrow{>}}}, postaction={decorate}
	] 
	(2.75,-3.5) arc (180:360:0.625);	
\draw (-1,0) -- (-1,-2);
\draw (-2,0) -- (-2,-2);
\draw (2.75,-2) -- (2.75,-3.5);
\draw (2.75,0) -- (2.75,1.5);
\draw (4,-3.5) -- (4,1.5);
\draw[dotted] (1,0)[] -- (2,0);
\draw[dotted] (1,-2) -- (2,-2);
\draw (0.7,0.23) node {{\scriptsize $Y$}};
\draw (1.7,0.23) node {{\scriptsize $X$}};
\draw (4.7,0.225) node {{\scriptsize $Y \otimes X$}};
\fill (0,-1) circle (2.5pt) node[right] {};
\end{tikzpicture}
\vspace{-1.25cm}
$$
where we used the Zorro moves to re-express the identity on $Y\otimes X$, as well as~\eqref{eq:switcherooomegalines}. Inserting the quantum dimension of a wiggly line (which equals one) we are now in a position to apply~\eqref{eq:nicepivotality} to the lower part of the diagram on the right-hand side to find that it equals
$$
\!\!\!\!\!\!
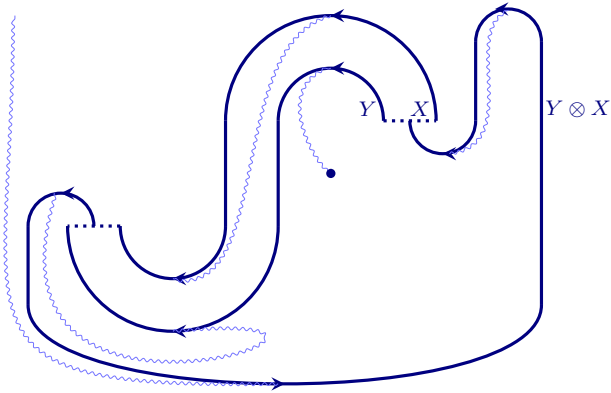
\begin{tikzpicture}[very thick,scale=0.7,color=blue!50!black, baseline=-0.75cm,>=stealth]
\draw[
	decoration={markings, mark=at position 0.5 with {\arrow{>}}}, postaction={decorate}
	] 
	(0:1) arc (0:180:1);
\draw[
	decoration={markings, mark=at position 0.5 with {\arrow{>}}}, postaction={decorate}
	] 
	(0:2) arc (0:180:2);
\draw[
	decoration={markings, mark=at position 0.5 with {\arrow{>}}}, postaction={decorate}
	] 
	(-1,-2) arc (360:180:2);
\draw[
	decoration={markings, mark=at position 0.5 with {\arrow{>}}}, postaction={decorate}
	] 
	(-2,-2) arc (360:180:1);
\draw[
	decoration={markings, mark=at position 0.5 with {\arrow{>}}}, postaction={decorate}
	] 
	(2.75,0) arc (360:180:0.625);
\draw[
	decoration={markings, mark=at position 0.5 with {\arrow{>}}}, postaction={decorate}
	] 
	(-4.5,-2) arc (0:180:0.625);
\draw[
	decoration={markings, mark=at position 0.5 with {\arrow{<}}}, postaction={decorate}
	] 
	(2.75,1.5) arc (180:0:0.625);
\draw[
	decoration={markings, mark=at position 0.5 with {\arrow{>}}}, postaction={decorate}
	] 
	(-5.75,-3.5) .. controls +(0,-2) and +(0,-2) .. (4,-3.5);
\draw[very thin, color=blue!55!white, decorate, decoration={snake, amplitude=0.2mm, segment length=1.0mm}] 
(0,-1) .. controls +(-0.3,0.2) and +(-1.1,-0.4) .. (0,1); 
\draw[very thin, color=blue!55!white, decorate, decoration={snake, amplitude=0.2mm, segment length=1.0mm}] 
(-3,-4) .. controls +(5.5,0.5) and +(-1.75,-5.5) .. (-5.25,-1.375); 
\draw[very thin, color=blue!55!white, decorate, decoration={snake, amplitude=0.2mm, segment length=1.0mm}] 
(2.25,-0.625) .. controls +(1.4,0.2) and +(-1.0,-0.32) .. (3.375,2.1); 
\draw[very thin, color=blue!55!white, decorate, decoration={snake, amplitude=0.2mm, segment length=1.0mm}] 
(-1,-5) .. controls +(-6.5,-0.1) and +(0,-6) .. (-6,2); 
\draw[very thin, color=blue!55!white, decorate, decoration={snake, amplitude=0.2mm, segment length=1.0mm}] 
(0,2) .. controls +(-2.0,-0.6) and +(2.0,-0.4) .. (-3,-3); 
\draw (-1,0) -- (-1,-2);
\draw (-2,0) -- (-2,-2);
\draw (-5.75,-2) -- (-5.75,-3.5);
\draw (2.75,0) -- (2.75,1.5);
\draw (4,-3.5) -- (4,1.5);
\draw[dotted] (1,0) -- (2,0);
\draw[dotted] (-4,-2) -- (-5,-2);
\draw (0.7,0.23) node {{\scriptsize $Y$}};
\draw (1.7,0.23) node {{\scriptsize $X$}};
\draw (4.7,0.23) node {{\scriptsize $Y \otimes X$}};
\fill (0,-1) circle (2.5pt) node[right] {};
\end{tikzpicture}
\, . 
\vspace{-1.2cm}
$$
After another application of the identity~\eqref{eq:switcherooomegalines} in conjunction with the Zorro moves this becomes 
$$
\!\!\!\!\!\!\!\!\!\!\!\!\!\!\!\!\!\!\!\!\!\!\!\!
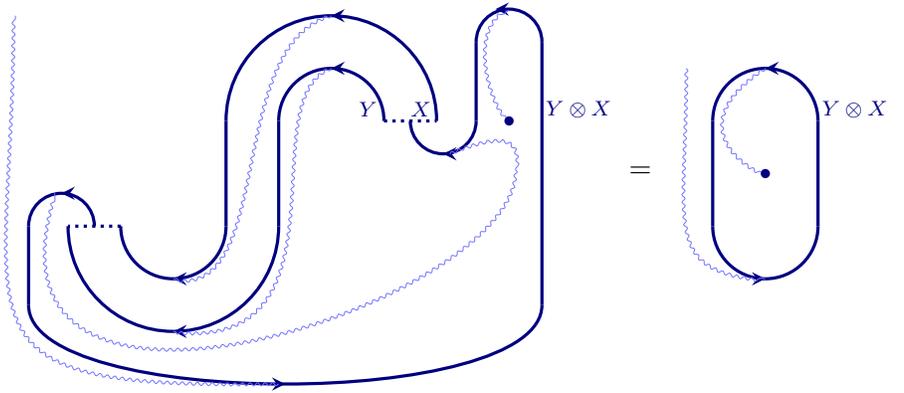
\begin{tikzpicture}[very thick,scale=0.7,color=blue!50!black, baseline=-0.75cm,>=stealth]
\draw[
	decoration={markings, mark=at position 0.5 with {\arrow{>}}}, postaction={decorate}
	] 
	(0:1) arc (0:180:1);
\draw[
	decoration={markings, mark=at position 0.5 with {\arrow{>}}}, postaction={decorate}
	] 
	(0:2) arc (0:180:2);
\draw[
	decoration={markings, mark=at position 0.5 with {\arrow{>}}}, postaction={decorate}
	] 
	(-1,-2) arc (360:180:2);
\draw[
	decoration={markings, mark=at position 0.5 with {\arrow{>}}}, postaction={decorate}
	] 
	(-2,-2) arc (360:180:1);
\draw[
	decoration={markings, mark=at position 0.5 with {\arrow{>}}}, postaction={decorate}
	] 
	(2.75,0) arc (360:180:0.625);
\draw[
	decoration={markings, mark=at position 0.5 with {\arrow{>}}}, postaction={decorate}
	] 
	(-4.5,-2) arc (0:180:0.625);
\draw[
	decoration={markings, mark=at position 0.5 with {\arrow{<}}}, postaction={decorate}
	] 
	(2.75,1.5) arc (180:0:0.625);
\draw[
	decoration={markings, mark=at position 0.5 with {\arrow{>}}}, postaction={decorate}
	] 
	(-5.75,-3.5) .. controls +(0,-2) and +(0,-2) .. (4,-3.5);
\draw[very thin, color=blue!55!white, decorate, decoration={snake, amplitude=0.2mm, segment length=1.0mm}] 
(0,1) .. controls +(-1.5,-0.6) and +(3.5,-0.4) .. (-3,-4); 
\draw[very thin, color=blue!55!white, decorate, decoration={snake, amplitude=0.2mm, segment length=1.0mm}] 
(3.5,-0.925) .. controls +(0,-2.5) and +(-1.75,-5.5) .. (-5.25,-1.375);

\draw[very thin, color=blue!55!white, decorate, decoration={snake, amplitude=0.2mm, segment length=1.0mm}] 
(2.25,-0.625) .. controls +(1.25,0.5) and +(0,0.4) .. (3.5,-0.925); 

\draw[very thin, color=blue!55!white, decorate, decoration={snake, amplitude=0.2mm, segment length=1.0mm}] 
(3.375,0) .. controls +(-0.3,0.2) and +(-0.8,-0.4) .. (3.375,2.1); 
\draw[very thin, color=blue!55!white, decorate, decoration={snake, amplitude=0.2mm, segment length=1.0mm}] 
(-1,-5) .. controls +(-6.5,-0.1) and +(0,-6) .. (-6,2); 
\draw[very thin, color=blue!55!white, decorate, decoration={snake, amplitude=0.2mm, segment length=1.0mm}] 
(0,2) .. controls +(-2.0,-0.6) and +(2.0,-0.4) .. (-3,-3); 
\draw (-1,0) -- (-1,-2);
\draw (-2,0) -- (-2,-2);
\draw (-5.75,-2) -- (-5.75,-3.5);
\draw (2.75,0) -- (2.75,1.5);
\draw (4,-3.5) -- (4,1.5);
\draw[dotted] (1,0) -- (2,0);
\draw[dotted] (-4,-2) -- (-5,-2);
\draw (0.7,0.23) node {{\scriptsize $Y$}};
\draw (1.7,0.23) node {{\scriptsize $X$}};
\draw (4.7,0.23) node {{\scriptsize $Y \otimes X$}};
\fill (3.375,0) circle (2.5pt) node[right] {};
\end{tikzpicture}
=
\begin{tikzpicture}[very thick,scale=0.7,color=blue!50!black, baseline=-0.75cm,>=stealth]
\draw[
	decoration={markings, mark=at position 0.5 with {\arrow{>}}}, postaction={decorate}
	] 
	(0:1) arc (0:180:1);
\draw[
	decoration={markings, mark=at position 0.5 with {\arrow{>}}}, postaction={decorate}
	] 
	(-1,-2) arc (180:360:1);
%
\draw[very thin, color=blue!55!white, decorate, decoration={snake, amplitude=0.2mm, segment length=1.0mm}] 
(0,1) .. controls +(-1.5,-1.0) and +(-0.5,0.2) .. (0,-1); 
\draw[very thin, color=blue!55!white, decorate, decoration={snake, amplitude=0.2mm, segment length=1.0mm}] 
(-1.5,1) .. controls +(0,-3.0) and +(-2.0,0.2) .. (0,-3); 
\draw (-1,0) -- (-1,-2);
\draw (1,0) -- (1,-2);
%
\draw (1.7,0.23) node {{\scriptsize $Y \otimes X$}};
\fill (0,-1) circle (2.5pt) node[right] {};
\end{tikzpicture}
\, , 
\vspace{-1.0cm}
$$
thus ending the proof. 
\end{proof}

\section{Open/closed topological field theory}\label{sec:ocTFT}

In this section we explain how the structure of two-dimensional open/closed topological field theory (TFT) is completely captured in the bicategory $\LG$, and how string diagrams enable us to prove nontrivial statements about matrix factorisations in a way that follows the physical intuition. This perspective also allows for the computation of any correlator (of worldsheets of arbitrary genus, possibly with boundaries and defect lines). As an example we will compute an annulus correlator, thereby providing a new, simple proof of the Cardy condition.


Let $k$ be a field (of arbitrary characteristic). All algebras will be $k$-algebras, and categories will be $k$-linear. Recall from~\cite{l0010269, ms0609042} that one way to present a two-dimensional open/closed TFT is by the data of 
\begin{itemize}
\item a commutative Frobenius algebra~$C$, 
\item a Calabi-Yau category~$\mathcal O$, 
\item \textsl{bulk-boundary maps} $\beta_A: C \lra \End_{\mathcal O}(A)$ and \textsl{boundary-bulk maps} $\beta^A: \End_{\mathcal O}(A) \lra C$ for all $A\in \mathcal O$. 
\end{itemize}
These data are subject to the following conditions. 
\begin{itemize}
\item The bulk-boundary maps $\beta_A$ are morphisms of unital algebras that map into the centre of $\End_{\mathcal O}(A)$. 
\item $\beta_A$ and $\beta^A$ are mutually adjoint with respect to the nondegenerate pairings $\langle -,- \rangle$ on~$C$ and $\langle -,- \rangle_A$ on $\End_{\mathcal O}(A)$ (which are part of the Frobenius and Calabi-Yau structure): 
$$
\langle \beta_A(\phi) , \psi \rangle_A = \langle \phi, \beta^A(\psi) \rangle
$$
for all $\phi\in C$ and $\psi\in \End_{\mathcal O}(A)$. 
\item The \textsl{Cardy condition} is satisfied, i.\,e.~we have
$$
\str ( {}_\psi m_\varphi ) = \langle \beta^A(\varphi), \beta^B(\psi) \rangle
$$
for all $\varphi: A \lra A$, $\psi: B \lra B$ where ${}_\psi m_\varphi (\alpha) = \psi\alpha\varphi$ for all $\alpha \in \Hom_{\mathcal O}(A,B)$. 
\end{itemize}

Every Landau-Ginzburg model with potential $W\in R = k[x_1,\ldots,x_n]$ gives rise to an open/closed TFT with $C = R/(\partial W)$, $\mathcal O = \hmf(R,W)$, 
\be\label{bubobobu}
\beta_X : \phi \lmt \phi \cdot 1_X \, , \qquad  
\beta^X : \psi \lmt (-1)^{n+1\choose 2} \str(\psi \, \partial_{x_1} d_X \ldots \partial_{x_n} d_X) 
\ee
and the bulk and boundary pairings
\begin{align}
\langle \phi_1 , \phi_2 \rangle 
& = 
\Res_{k[x]/k} \left[ \frac{\phi_1 \phi_2 \, \underline{\operatorname{d}\!x}}{\partial_{x_1}W, \ldots, \partial_{x_n} W}\right] , \label{bulkpairing} \\
\langle \psi_1 , \psi_2 \rangle_X 
& = 
\Res_{k[x]/k} \left[ \frac{ \str (\psi_1 \psi_2 \, \partial_{x_1} d_X \ldots \partial_{x_n} d_X) \,  \underline{\operatorname{d}\!x}}{\partial_{x_1}W, \ldots, \partial_{x_n} W}\right] . \label{KapuLi}
\end{align}
The hardest part in establishing this result is to prove the nondegeneracy of the Kapustin-Li pairing~\eqref{KapuLi} and that the Cardy condition holds (the fact that~\eqref{bulkpairing} is nondegenerate is a classical result in residue theory, and checking the remaining axioms is obvious or straightforward); this was first done in~\cite[Theorem\,6.2]{m0912.1629} and~\cite[Theorem\,4.1.4]{pv1002.2116}, respectively.

Before turning to the Cardy condition we wish to explain how the above data can be extracted from the bicategory $\LG$ using the string diagram language.\footnote{Note however that $\LG$ contains much more information than just the structure of open/closed TFT.} For $X\in \hmf(R,W)$ the bulk-boundary and boundary-bulk maps~\eqref{bubobobu} are given by
$$
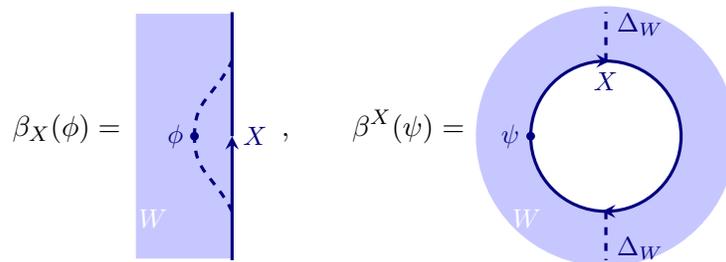
\beta_X(\phi) = 
\begin{tikzpicture}[very thick,scale=0.5,color=blue!50!black, baseline,>=stealth]
\nicecolourscheme (-2.6,-3.3) rectangle (0,3.3);
\fill (-2.1,-2.2) circle (0pt) node[white] {{\small$W$}};
\draw[->, very thick] (0,-3.3) -- (0,0); 
\draw[very thick] (0,0) -- (0,3.3);
\draw[dashed] (0,-2) .. controls +(-0.5,1) and +(0,-1) .. (-1,0);
\draw[dashed] (-1,0) .. controls +(0,1) and +(-0.5,-1) .. (0,2);
\fill (0,0) circle (0pt) node[right] {{\small$X$}};
\fill (-1,0) circle (3.3pt) node[left] {{\small$\phi$}};
\end{tikzpicture} 
\, , \qquad 
\beta^X(\psi) = 
\begin{tikzpicture}[very thick,scale=0.5,color=blue!50!black, baseline,>=stealth]
\nicecolourscheme (0,0) circle (3.5);
\fill (-2.1,-2.2) circle (0pt) node[white] {{\small$W$}};
\shadedraw[top color=white, bottom color=white, draw=white] (0,0) circle (2);
\draw (0,0) circle (2);
\draw[->, very thick] (0.100,-2) -- (-0.101,-2) node[above] {}; 
\draw[->, very thick] (-0.100,2) -- (0.101,2) node[below] {}; 
\fill (90:2) circle (0pt) node[below] {{\small$X$}};
\fill (180:2) circle (3.3pt) node[left] {{\small$\psi$}};
\draw[dashed] (270:2) -- (270:3.3)
node[near end,right] {{{\small$\Delta_W$}}};
\draw[dashed] (90:2) -- (90:3.3)
node[near end,right] {{{\small$\Delta_W$}}};
\end{tikzpicture} 
$$
where we used the identification $R/(\partial W) \cong \End_{\hmf(R,W)}(\Delta_W)$ and a special case of Proposition~\ref{prop:defectaction}.\footnote{As a side remark we observe that the Chern character $\operatorname{ch}(X) = \beta^X(1_X)$ can be thought of as the supertrace of the exponential of the Atiyah class of~$X$ as follows from the expressions~\eqref{eq:coev102} and~\eqref{eq:constructevalmap00} that compose to $\operatorname{ch}(X)$.} Another special case of the same proposition allows us to recover the Kapustin-Li pairing as the obvious 2-morphism representing the disc correlator: 
$$
\langle \psi_1 , \psi_2 \rangle_X = 
\begin{tikzpicture}[very thick,scale=0.6,color=blue!50!black, baseline,>=stealth]
\nicecolourscheme (0,0) circle (2);
\fill (-1.1,-1.1) circle (0pt) node[white] {{\small$W$}};
\draw (0,0) circle (2);
\draw[<-, very thick] (0.100,-2) -- (-0.101,-2) node[above] {}; 
\draw[<-, very thick] (-0.100,2) -- (0.101,2) node[below] {}; 
\fill (-22.5:2) circle (3.3pt) node[right] {{\small$\psi_2$}};
\fill (22.5:2) circle (3.3pt) node[right] {{\small$\psi_1$}};
\fill (270:2) circle (0pt) node[below] {{\small$X$}};
\end{tikzpicture} 
\, .
$$
Note that here and from now on we do no longer display dashed lines for the unit~$\Delta_W$. 

The bulk pairing $\langle -,- \rangle = \langle -,- \rangle_W$ describes the 2-point \textsl{sphere} correlator; flattening the sphere suggests the identity
\be\label{DeltaDisk}
\langle \phi_1, \phi_2 \rangle_W = 
\begin{tikzpicture}[very thick,scale=0.6,color=blue!50!black, baseline,>=stealth]
\nicecolourscheme (0,0) circle (2);
\fill (0,0) circle (0pt) node[white] {{\small$W(x)-W(y)$}};
\draw (0,0) circle (2);
\draw[<-, very thick] (0.100,-2) -- (-0.101,-2) node[above] {}; 
\draw[<-, very thick] (-0.100,2) -- (0.101,2) node[below] {}; 
\fill (-22.5:2) circle (3.3pt) node[right] {{\small$\phi_2(x)$}};
\fill (22.5:2) circle (3.3pt) node[right] {{\small$\phi_1(x)$}};
\fill (270:2) circle (0pt) node[below] {{\small$\Delta_W$}};
\end{tikzpicture} 
\ee
where we view $\Delta_W$ as a 1-morphism $(k,0) \lra (\Re, \widetilde W)$, i.\,e.~as a boundary condition of the doubled theory with potential $W(x)-W(y)$. The 
identity~\eqref{DeltaDisk}
indeed holds as follows from our explicit expressions for the adjunction maps together with~\cite[Proposition\,4.1.2]{pv1002.2116}. A more conceptual derivation of~\eqref{DeltaDisk} would involve endowing the bicategory $\LG$ with a monoidal structure (so as to give rigorous meaning to the process of folding worldsheets together, i\,.e.~tensoring objects in the bicategory $\LG$).\footnote{In fact it can be checked that $\LG$ can be endowed with data that satisfies the necessary conditions for a symmetric monoidal structure given in \cite{s1004.0993}, but that is beyond the scope of the present paper.}

We have just seen how the structure of open/closed TFT embeds into $\LG$. In principle this is enough to compute arbitrary correlators using the factorisation property of TFT. However, one can also compute more general correlators directly in $\LG$: all we have to do is to interpret the physical picture of a worldsheet~$\Sigma$ with insertions and defect lines as the associated string diagram representing a 2-morphism $k \lra k$ which is the value of the correlator of~$\Sigma$.

\subsection{Cardy condition}

For Landau-Ginzburg models the Cardy condition was recently proven using different methods for~$k$ a field of characteristic zero in~\cite{pv1002.2116}, and there were subsequent proofs in \cite{dm1102.2957,buchweitzstratten}.

Let $k$ be an arbitrary ring. We fix a potential $W\in k[x] = k[x_1,\ldots,x_n]$ in the sense of Definition~\ref{defn:potential}, matrix factorisations $X,Y \in \hmf(k[x], W)$, and morphisms $\varphi: X\lra X$, $\psi: Y\lra Y$. Then the $2$-morphism
\be\label{annulus}
\begin{tikzpicture}[very thick,scale=0.8,color=blue!50!black, baseline,>=stealth]
\nicecolourscheme (0,0) circle (2);
\nicereallynocolourscheme (0,0) circle (1);
\fill (1.5,0) circle (0pt) node[white] {{\small$W$}};
\draw (0,0) circle (2);
\draw[->, very thick] (-0.100,2) -- (-0.101,2) node[above] {}; 
\fill (45:2) circle (2.5pt) node[right] {{\small$\psi$}};
\draw (0,0) circle (1);
\draw[->, very thick] (0.100,1) -- (0.101,1) node[above] {}; 
\fill (135:1) circle (2.5pt) node[left] {{\small$\varphi$}};
\fill (180:0.9) circle (0pt) node[left] {{\small$X$}};
\fill (0:2.7) circle (0pt) node[left] {{\small$Y$}};
\end{tikzpicture} 
\equiv \; 
\begin{tikzpicture}[very thick,scale=0.8,color=blue!50!black, baseline,>=stealth]
\nicecolourscheme (0,0) circle (2);
\nicereallynocolourscheme (0,0) circle (1);
\fill (1.5,0) circle (0pt) node[white] {{\small$W$}};
\draw (0,0) circle (2);
\draw[->, very thick] (-0.100,2) -- (-0.101,2) node[above] {{\small$\eval_Y$}}; 
\draw[->, very thick] (0.100,-2) -- (0.101,-2) node[below] {{\small$\widetilde\coev_Y$}}; 
\fill (45:2) circle (2.5pt) node[right] {{\small$\psi$}};
\draw (0,0) circle (1);
\draw[->, very thick] (0.100,1) -- (0.101,1) node[above] {{\small$\widetilde\eval_X$}}; 
\draw[->, very thick] (-0.100,-1) -- (-0.101,-1) node[below] {{\small$\coev_X$}}; 
\fill (135:1) circle (2.5pt) node[left] {{\small$\varphi$}};
\draw[very thin, color=blue!55!white, decorate, decoration={snake, amplitude=0.2mm, segment length=1.0mm}] (270:1) .. controls +(-2,0) and +(-2,-0.2) .. (90:2); 
\end{tikzpicture} 
\ee
is a map $k\lra k$, i.\,e.~an element of $k$, that we call~$C$. We can think of~$C$ as the value of the annulus correlator with boundary conditions $X,Y$ and boundary fields $\varphi,\psi$. 

There are two natural ways to compute the value of~\eqref{annulus}. One possibility is to first collapse the $X$-loop to obtain a ``$W$-bubble'', and then to collapse the boundary loop of this bubble to arrive at a scalar. The main result on defect actions (Proposition \ref{prop:defectaction}) shows that the first step produces the 2-morphism $(-1)^{n+1\choose 2} \str(\varphi \Lambda_X)\in\End(\Delta_W)$ where $\Lambda_Z = \partial_{x_1}d_Z \ldots \partial_{x_n}d_Z$ for $Z \in \{X,Y\}$. Similarly, the second step of collapsing the~$Y$-loop produces a residue: 
$$
C 
= 
\begin{tikzpicture}[very thick,scale=0.6,color=blue!50!black, baseline,>=stealth]
\nicecolourscheme (0,0) circle (2);
\nicereallynocolourscheme (0,0) circle (1);
\fill (1.5,0) circle (0pt) node[white] {{\small$W$}};
\draw (0,0) circle (2);
\draw[->, very thick] (-0.100,2) -- (-0.101,2) node[above] {}; 
\draw[->, very thick] (0.100,-2) -- (0.101,-2) node[below] {}; 
\fill (45:2) circle (2.5pt) node[right] {{\small$\psi$}};
\draw (0,0) circle (1);
\draw[->, very thick] (0.100,1) -- (0.101,1) node[above] {}; 
\draw[->, very thick] (-0.100,-1) -- (-0.101,-1) node[below] {}; 
\fill (135:1) circle (2.5pt) node[left] {{\small$\varphi$}};
\end{tikzpicture} 
= 
(-1)^{n+1\choose 2} \,  
\begin{tikzpicture}[very thick,scale=0.6,color=blue!50!black, baseline,>=stealth]
\nicecolourscheme (0,0) circle (2);
\fill (1.5,0) circle (0pt) node[white] {{\small$W$}};
\draw (0,0) circle (2);
\draw[->, very thick] (-0.100,2) -- (-0.101,2) node[above] {}; 
\draw[->, very thick] (0.100,-2) -- (0.101,-2) node[below] {}; 
\fill (45:2) circle (2.5pt) node[right] {{\small$\psi$}};
\fill (45:0) circle (2.5pt) node[below] {{\small$\str(\varphi\Lambda_X)$}};
\end{tikzpicture} 
= 
(-1)^{n+1\choose 2} \,  
\Res_{k[x]/k} \left[ \frac{\str\! \big( \str( \varphi \Lambda_X) \psi \Lambda_Y \big) \underline{\operatorname{d}\! x}}{\partial_1 W, \ldots, \partial_n W} \right] .
$$
We observe that this expression vanishes unless~$n$ is even, 
since the supertrace of an odd operator is zero, and~$\Lambda_X$, say, is odd for odd~$n$. 
Writing $\beta^X(\varphi) = (-1)^{n+1\choose 2} \str(\varphi \Lambda_X)$ and $\beta^Y(\psi) = (-1)^{n+1\choose 2} \str(\psi \Lambda_Y)$ for the boundary-bulk maps we thus find from the above that
\be\label{C1}
C = (-1)^{n+1\choose 2} \,  
\Res_{k[x]/k} \left[ \frac{\beta^X(\varphi) \, \beta^Y (\psi) \, \underline{\operatorname{d}\! x}}{\partial_1 W, \ldots, \partial_n W} \right] .
\ee

A second way of computing~$C$ is to inflate the $X$-loop in \eqref{annulus} and fuse it with the outer $Y$-loop, creating a single loop labelled by $X^\vee \otimes Y \in \End_{\LG}( (k,0) )$. Rigorously, this process is described by the series of equalities
\[
\begin{tikzpicture}[very thick,scale=0.8,color=blue!50!black, baseline,>=stealth]
\nicecolourscheme (0,0) circle (2);
\nicereallynocolourscheme (0,0) circle (1);
\fill (1.5,0) circle (0pt) node[white] {{\small$W$}};
\draw (0,0) circle (2);
\draw[->, very thick] (-0.100,2) -- (-0.101,2) node[above] {}; 
\fill (45:2) circle (2.5pt) node[right] {{\small$\psi$}};
\draw (0,0) circle (1);
\draw[->, very thick] (0.100,1) -- (0.101,1) node[above] {}; 
\fill (135:1) circle (2.5pt) node[left] {{\small$\varphi$}};
\fill (180:1.0) circle (0pt) node[right] {{\small$X$}};
\fill (0:2.7) circle (0pt) node[left] {{\small$Y$}};
\end{tikzpicture} 
=
\begin{tikzpicture}[very thick,scale=0.8,color=blue!50!black, baseline,>=stealth]
\nicecolourscheme (0,0) circle (2);
\nicereallynocolourscheme (0,0) circle (1);
\fill (1.5,0) circle (0pt) node[white] {{\small$W$}};
\draw (0,0) circle (2);
\draw[->, very thick] (-0.100,2) -- (-0.101,2) node[above] {}; 
\fill (45:2) circle (2.5pt) node[right] {{\small$\psi$}};
\draw (0,0) circle (1);
\draw[->, very thick] (-0.100,1) -- (-0.101,1) node[above] {}; 
\fill (45:1) circle (2.5pt) node[right] {{\small$\varphi^{\vee}$}};
\fill (180:1.0) circle (0pt) node[right] {{\small$X^{\vee}$}};
\fill (0:2.7) circle (0pt) node[left] {{\small$Y$}};
\end{tikzpicture} 
= \; 
\begin{tikzpicture}[very thick,scale=0.8,color=blue!50!black, baseline,>=stealth]
\draw[ultra thick] (0,0) circle (2);
\draw[->, ultra thick] (-0.100,2) -- (-0.101,2) node[above] {}; 
\draw[->, ultra thick] (0.100,-2) -- (0.101,-2) node[below] {}; 
\fill (45:2) circle (3.3pt) node[right] {{\small$\varphi^\vee\otimes\psi$}};
\fill (0:2.0) circle (0pt) node[right] {{\small$X^{\vee} \otimes Y$}};
\end{tikzpicture} 
\]
where in the first step we use Proposition \ref{prop:defectactionproperties}(ii) to reorient the inner loop, and in the second step we use (iii) to fuse the two loops (note that the proof of (iii) relies on Proposition~\ref{prop:monoidalproperty}). The complexes $X^{\vee} \otimes Y$ and $\Hom(X,Y)$ are isomorphic, so
\be\label{eq:finaldiagramcardycond}
\begin{tikzpicture}[very thick,scale=0.6,color=blue!50!black, baseline,>=stealth]
\draw[ultra thick] (0,0) circle (2);
\draw[->, ultra thick] (-0.100,2) -- (-0.101,2) node[above] {}; 
\draw[->, ultra thick] (0.100,-2) -- (0.101,-2) node[below] {}; 
\fill (45:2) circle (3.3pt) node[right] {{\small$\varphi^\vee\otimes\psi$}};
\fill (-30:2) circle (0pt) node[right] {{\small$X^{\vee} \otimes Y$}};
\end{tikzpicture} 
= \; 
\begin{tikzpicture}[very thick,scale=0.6,color=blue!50!black, baseline,>=stealth]
\draw[ultra thick] (0,0) circle (2);
\draw[->, ultra thick] (-0.100,2) -- (-0.101,2) node[above] {}; 
\draw[->, ultra thick] (0.100,-2) -- (0.101,-2) node[below] {}; 
\fill (45:2) circle (3.3pt) node[right] {{\small${}_\psi m_{\varphi}$}};
\fill (-30:2) circle (0pt) node[right] {{\small$\Hom(X,Y)$}};
\end{tikzpicture} 
= \str({}_\psi m_{\varphi})
\ee
where $ {}_\psi m_\varphi \in \End_k(\Hom(X,Y))$ is the operator that sends~$\alpha$ to $ \psi\circ\alpha\circ\varphi$. 

The value of an oriented loop labelled by a finite-rank $\nZ_2$-graded $k$-complex $F$ with the insertion of a morphism of complexes $\phi: F \lto F$ is easily checked to be the supertrace of $\phi$. If $k$ were a field the infinite-rank $k$-complex $\Hom(X,Y)$ would be isomorphic, in $\HMF(k,0)$, to the finite-dimensional $\nZ_2$-graded vector space $H^*\Hom(X,Y)$ with zero differential, and the value of \eqref{eq:finaldiagramcardycond} would be the supertrace of the action of ${}_\psi m_{\varphi}$ on this cohomology.

In general $\Hom(X,Y)$ is not isomorphic to its cohomology, but it is at least a summand in the homotopy category of a finite-rank $\nZ_2$-graded free $k$-module and therefore any operator on $\Hom(X,Y)$ has a well-defined supertrace. Indeed, once we choose such a split embedding of $\Hom(X,Y)$ in something finite-rank, the supertrace can be computed by writing out the diagram \eqref{eq:finaldiagramcardycond} as a series of morphisms, using the evaluation and coevaluation maps.

Comparing~\eqref{C1} and~\eqref{eq:finaldiagramcardycond} we arrive at:

\begin{theorem}\label{thm:CardyCondition}
The Cardy condition holds in $\LG$: given 1-morphisms $X,Y \in \hmf( k[x_1,\ldots,x_n], W)$ and 2-morphisms $\varphi: X\lra X$, $\psi: Y\lra Y$ we have 
$$
\str ( {}_\psi m_\varphi ) 
= 
(-1)^{n+1\choose 2} \,  
\Res_{k[x]/k} \left[ \frac{\beta^X(\varphi) \, \beta^Y (\psi) \, \underline{\operatorname{d}\! x}}{\partial_1 W,
 \ldots, \partial_n W} \right] 
$$
with ${}_\psi m_\varphi(\alpha) = \psi\circ\alpha\circ\varphi$ for $\alpha\in \Hom(X,Y)$. 
\end{theorem}

\begin{remark}
The Cardy condition also holds if one or both of the maps $\varphi, \psi$ are odd (which we present as morphisms $X\lra X[1]$ in $\LG$). The proof of this fact proceeds similarly to the above. 
\end{remark}

\section{Shadows}\label{sec:shadows}

The adjunctions in $\LG$ afford us the construction of a bicategorical trace in terms of shadow functors~\cite{p0807.1471}. We will also see that shadows allow to recover and generalise the boundary-bulk and bulk-boundary maps of the two-dimensional TFTs based on Landau-Ginzburg models. 

\begin{definition}
A bicategory~$\mathcal B$ \textsl{has shadows} if there is a category~$\mathcal C$ together with functors
$$
\langle\!\langle - \rangle\!\rangle : \mathcal B(A,A) \lra \mathcal C 
$$
for every object $A\in \mathcal B$ and natural isomorphisms $\theta : \langle\!\langle X \otimes Y \rangle\!\rangle \lra \langle\!\langle Y \otimes X \rangle\!\rangle$ for every pair of composable 1-morphisms $X: A \lto B,Y: B \lto A$, such that the diagrams
$$
\xymatrix@C+0.5pc{%
\langle\!\langle (X \otimes Y) \otimes Z \rangle\!\rangle \ar[r]^-{\theta} \ar[d]_-{\langle\!\langle \alpha \rangle\!\rangle} & 
\langle\!\langle Z \otimes (X \otimes Y) \rangle\!\rangle \ar[r]^-{\langle\!\langle \alpha^{-1} \rangle\!\rangle} & 
\langle\!\langle (Z \otimes X) \otimes Y \rangle\!\rangle \\
\langle\!\langle X \otimes (Y \otimes Z) \rangle\!\rangle \ar[r]^-{\theta} & 
\langle\!\langle (Y \otimes Z) \otimes X \rangle\!\rangle \ar[r]^-{\langle\!\langle \alpha \rangle\!\rangle} & 
\langle\!\langle Y \otimes (Z \otimes X) \rangle\!\rangle \ar[u]_-{\theta}
}%
$$
and
$$
\xymatrix{%
\langle\!\langle X \otimes \Delta_A \rangle\!\rangle \ar[r]^-{\theta} \ar[dr]_-{\langle\!\langle \rho \rangle\!\rangle} & 
\langle\!\langle \Delta_A \otimes X \rangle\!\rangle \ar[r]^-{\theta} \ar[d]^-{\langle\!\langle \lambda \rangle\!\rangle} & 
\langle\!\langle X \otimes \Delta_A \rangle\!\rangle \ar[dl]^-{\langle\!\langle \rho \rangle\!\rangle}  \\
 & 
\langle\!\langle X \rangle\!\rangle & 
}%
$$
commute whenever they make sense. 
\end{definition}

Let $\mathcal{M}_2(k)$ denote the category of $\nZ _2$-graded $k$-modules. If $Y$ is a $\nZ _2$-graded complex of $k$-modules then the cohomology $H^*(Y)$ is naturally an object of $\mathcal{M}_2(k)$.

\begin{proposition}
The bicategory $\LG$ has shadows given by 
\begin{gather*}
\langle\!\langle - \rangle\!\rangle : \LG \big( (R,W), (R,W) \big) \lra \mathcal{M}_2(k) 
\, , \qquad
Z \lmt H^*(Z\otimes_{\Re} R)
\end{gather*}
and the isomorphism $\theta: \langle\!\langle X \otimes Y \rangle\!\rangle \lra \langle\!\langle Y\otimes X  \rangle\!\rangle$ induced by the graded swap map $X \otimes_{\Re} Y \lra Y \otimes_{\Se} X$ 
for $X \in \LG( (R,W), (S,V) )$ and $Y \in \LG( (S,V), (R,W) )$. 
\end{proposition}

The proof is a straightforward check of the axioms. Since $\LG$ is a bicategory with adjoints and shadows it is automatically equipped with a 2-categorical trace operation as introduced and discussed at length in~\cite{p0807.1471, ps0910.1306}. We only quote the definition 
from \cite[Definition\,4.5.1]{p0807.1471}: 

\begin{definition}
Let~$\mathcal B$ be a bicategory with shadows and~$Y$ a 1-morphism with left adjoint~${}^\dual Y$. The \textsl{trace} of a 2-morphism $\psi: X\otimes Y \lra Y \otimes Z$ is the map 
$$
\xymatrix{%
\langle\!\langle X \rangle\!\rangle \ar[rr]^-{\langle\!\langle1\otimes \coev_Y\rangle\!\rangle} && 
\langle\!\langle X \otimes Y \otimes {}^\dual Y \rangle\!\rangle \ar[r]^-{\langle\!\langle\psi \otimes 1\rangle\!\rangle} & 
\langle\!\langle Y \otimes Z \otimes {}^\dual Y \rangle\!\rangle \ar[r]^-{\theta} &
\langle\!\langle {}^\dual Y \otimes Y \otimes Z \rangle\!\rangle \ar[rr]^-{\langle\!\langle\eval_Y \otimes 1\rangle\!\rangle} && 
\langle\!\langle Z \rangle\!\rangle 
}%
. 
$$
\end{definition}

Next we wish to point out a connection between shadows and the structure of two-dimensional open/closed topological field theory (TFT) for Landau-Ginzburg models. For this we merely observe that the bulk-boundary and boundary-bulk maps~\eqref{bubobobu} can also be recovered from the adjunction and shadow structure of $\LG$ as follows. On the one hand we have 
$$
\langle\!\langle \Delta_W \rangle\!\rangle 
= 
H^*\Big(\bigwedge (\bigoplus_{i=1}^n R \theta_i) , \sum_{i=1}^n \partial_{x_i} W \cdot \theta_i\Big) 
\cong 
R/(\partial W)[n]\,.
$$
On the other hand for $X\in \hmf(R,W)$ viewed as a $1$-morphism $(k,0) \lto (R,W)$ we have 
\[
\langle\!\langle X^\dual \otimes X \rangle\!\rangle = H^*\big( (X^\vee \otimes_k X) \otimes_{\Re} R \big)[n] \cong H^*\End_R(X)[n] \, . 
\]
Thus from the explicit expressions in Section~\ref{sec:derivcoeval} we find that $\beta_X = \langle\!\langle \widetilde\coev_X \rangle\!\rangle$ and $\beta^X = \langle\!\langle \eval_X \rangle\!\rangle$. 

This construction can be extended to any 1-morphism in $\LG$: for $X \in \hmf(k[x,z], V-W)$ we define the \textsl{generalised bulk-boundary} and \textsl{boundary-bulk maps} to be
$$
\beta_X= \langle\!\langle \widetilde\coev_X \rangle\!\rangle : \langle\!\langle \Delta_{W} \rangle\!\rangle \lra \langle\!\langle X^\dual \otimes X \rangle\!\rangle
\, , \qquad 
\beta^X= \langle\!\langle \eval_X \rangle\!\rangle : \langle\!\langle {}^\dual X \otimes X \rangle\!\rangle \lra \langle\!\langle \Delta_{W} \rangle\!\rangle \, , 
$$
respectively. Substituting the expressions for the adjunction maps in Section~\ref{sec:derivcoeval} we find that the form of $\beta_X$ stays the same while for $m\neq 0$ the generalised boundary-bulk map $\beta^X$  involves a residue and additional derivatives: 
\begin{align*}
\beta_X : k[x]/(\partial W) & \lra H^*\End_{k[x,z]}(X) \, , \\
\phi & \lmt \phi \cdot 1_{X} \, , \\
\beta^X : H^*\End_{k[x,z]}(X)[m] & \lra k[x]/(\partial W)[n] \, , \\
\psi & \lmt (-1)^{{m+1\choose 2}} \Res_{k[x,z]/k[z]} \left[ \frac{\str\left( \psi\,  \partial_{x_1} d_X \ldots \partial_{x_n} d_X \, \partial_{z_1} d_X \ldots \partial_{z_m} d_X \right) \underline{\operatorname{d}\!z}}{\partial_{z_1} V, \ldots, \partial_{z_m} V} \right] . 
\end{align*}

\appendix

\section{A split monomorphism}\label{section:symidem}

In this appendix we present, following \cite{dm1102.2957}, a split monomorphism used in the proof of pivotality in Section \ref{sec:wiggliesandsigns}. Let $R$ be a $k$-algebra which is projective as a $k$-module with a quasi-regular sequence $f = (f_1, \ldots, f_n)$ such that $\bar{R} = R/(f_1,\ldots,f_n)R$ is a finitely generated projective $k$-module. Let $W \in k$ be given and let $Z$ be a finite-rank matrix factorisation of $W$ over $R$ with each $f_i$ acting null-homotopically on $Z$. We assume that the Koszul complex $K(f_1,\ldots,f_n)$ is exact except in degree zero\footnote{This follows from automatically from quasi-regularity of the sequence $f$ in the case where $k$ is noetherian.}, so that by \cite[Section 7]{dm1102.2957} the projection
\begin{equation}\label{eq:symidem1}
\pi: Z \otimes K(f) \lto Z \otimes_R \bar{R} = \bar{Z}
\end{equation}
is a homotopy equivalence over $k$. Next we construct a split monomorphism $\kappa: Z \lto Z \otimes K(f)$. It will be helpful to introduce formal symbols $\theta_i$ of homological degree $-1$, so that $K(f_i) = \bigwedge( R \theta_i )$. There is an exact sequence
\[
\xymatrix{
0 \ar[r] & Z \ar[r]^-{\kappa_i} & Z \otimes K(f_i) \ar[r]^-{\varepsilon_i} & Z\theta_i \ar[r] & 0
}
\]
where $\kappa_i(x) = x \otimes 1$ is the inclusion and $\varepsilon_i( x \theta_i ) = (-1)^{|x|} x$. This sequence is split exact: the morphism $\rho_i(x + y \theta_i) = x + (-1)^{|y|} \lambda_i(y)$ satisfies $\rho_i \circ \kappa_i = 1$. Iterating this we see that the map $\kappa: Z \lto Z \otimes K(f)$ defined by $\kappa(x) = x \otimes 1$ is a split monomorphism. Combining this with \eqref{eq:symidem1} we obtain:

\begin{lemma}\label{lemma:quotientsplitmono} In the homotopy category of $k$-linear factorisations of $W$ the quotient map $Z \lto \bar{Z}$ is a split monomorphism.
\end{lemma}

\section{Homotopies and residues}\label{app:reshom}

In the the evaluation maps homotopies $\lambda_i$ appear and there are natural questions of independence of the choice of homotopy, and independence up to sign of the choice of ordering. Our aim here is to answer these questions.

The setup is as follows: $k$ is a ring, $R$ is a $k$-algebra, $W \in R$ and $X$ is a finite-rank matrix factorisation of $W$ over $R$. We assume that $(f_1,\ldots,f_n)$ is a regular sequence in $R$ such that $R/(f_1,\ldots,f_n)$ is a finitely generated projective $k$-module and that each $f_i$ acts null-homotopically on $X$, with $\lambda_i$ denoting a null-homotopy. We also assume given some $\omega \in \Omega^n_{R/k}$ and write $D = [d_X, -]$ for the differential on $\End(X) = \Hom_R(X,X)$.

For example, if $R = k[x_1,\ldots,x_n]$ and $W \in R$ is a potential in the sense of the body of the paper, all hypotheses are satisfied for $f_i = \partial_{x_i} W$, $\lambda_i = \partial_{x_i} d_X$ and $\omega = \ud x_1 \ldots \ud x_n$. We define
\[
\langle - \rangle = \Res_{R/k} \!\!\left[ \frac{ (-) \cdot \omega }{f_1, \ldots, f_n} \right].
\]
We may define the closed degree $n$ map
\be\label{eq:applanglerangle1}
\langle \str( - \circ \lambda_1 \cdots \lambda_n ) \rangle: \End(X) \lto k\,.
\ee
In this section all maps and homotopies are $k$-linear. Everything rests on the fact that the supertrace is a closed map, and thus for homogeneous $\alpha, \beta \in \End(X)$ we have
\[
0 = \str( D( \alpha \beta ) ) = \str( D(\alpha) \beta ) + (-1)^{|\alpha|} \str( \alpha D(\beta) )\,.
\]

\begin{lemma}\label{lemma:inreschoosehom} The map \eqref{eq:applanglerangle1} is independent, up to homotopy, of the choice of null-homotopies $\lambda_i$.
\end{lemma}
\begin{proof}
We show that replacing $\lambda_1$ by another null-homotopy $\mu_1$ has the effect of changing \eqref{eq:applanglerangle1} by a null-homotopic linear functional; the argument for the other $\lambda_i$ is identical. Given $\alpha \in \End(X)$, we have using the residue identity \eqref{eq:defin_residue_2}
\begin{align}
\langle \str( \alpha \lambda_1 \ldots \lambda_n ) \rangle &= \Res_{R/k} \!\!\left[ \frac{ \str( \alpha f_1 \lambda_1 \ldots \lambda_n ) \omega }{f_1^2, f_2, \ldots, f_n} \right]\label{eq:inreschoosehom1}\\
&= \Res_{R/k} \!\!\left[ \frac{ \str( \alpha D(\mu_1) \lambda_1 \ldots \lambda_n ) \omega }{f_1^2, f_2, \ldots, f_n} \right]\\
&= \Res_{R/k} \!\!\left[ \frac{ \str\big( (-1)^{|\alpha|}D(\alpha) \mu_1 \lambda_1 \ldots \lambda_n + \alpha \mu_1 D(\lambda_1 \ldots \lambda_n) \big) \omega }{f_1^2, f_2, \ldots, f_n} \right] . 
\end{align}
The first summand inside the supertrace gives a null-homotopic functional, and in $D(\lambda_1 \ldots \lambda_n) = D(\lambda_1) \lambda_2 \ldots \lambda_n - \lambda_1 D(\lambda_2 \ldots \lambda_n)$ the second summand contains factors of $f_i$ for $i > 2$ which make the residue vanish since the numerator belongs to the ideal generated by the denominators. Finally we have that \eqref{eq:applanglerangle1} is homotopic to the map
\begin{align*}
\Res_{R/k} \!\!\left[ \frac{ \str( - \circ \mu_1 D(\lambda_1) \lambda_2 \ldots \lambda_n ) \omega }{f_1^2, f_2, \ldots, f_n} \right] = \langle \str( - \circ \mu_1 \lambda_2 \ldots \lambda_n ) \rangle
\end{align*}
as claimed.
\end{proof}

\begin{lemma}\label{lemma:symmetrystrres} For any $\sigma \in S_n$ the map
\[
(-1)^{|\sigma|} \langle \str( - \circ \lambda_{\sigma(1)} \ldots \lambda_{\sigma(n)} ) \rangle: \End(X) \lto k
\]
is homotopic to \eqref{eq:applanglerangle1}.
\end{lemma}
\begin{proof}
If in \eqref{eq:inreschoosehom1} we instead multiplied top and bottom of the residue by $f_2$, so that the numerator reads $\str( \alpha f_2 \lambda_1 \ldots \lambda_n )$, then the same argument will show that there is a homotopy
\[
\langle \str( - \circ \lambda_1 \lambda_2 \ldots \lambda_n ) \rangle \simeq - \langle \str( - \circ \lambda_2 \lambda_1 \ldots \lambda_n )\,.
\]
The argument for a general permutation is the same.
\end{proof}

\section{Independence of variable ordering}\label{app:variableordering}

Objects of the bicategory $\LG$ are pairs consisting of a polynomial ring $R$ and potential $W$, together with a chosen ordering of the ring variables. Thus the value of a diagram in $\LG$ depends \textsl{a priori} on the ordering of the ring variables in each of its $2$-dimensional regions. In this appendix we show that the value is actually independent of these orderings, up to the natural permutation signs.

To take a concrete example, consider for $X \in \hmf(k[z_1,\ldots,z_m], V)$ the diagram
\be\label{eq:signdependencepermdia}
\begin{tikzpicture}[very thick,scale=0.4,color=blue!50!black, baseline,>=stealth]
\nicecolourscheme (0,0) circle (3.5);
\fill (-2.1,-2.1) circle (0pt) node[white] {{\small$V$}};
\shadedraw[top color=white, bottom color=white, draw=white] (0,0) circle (2);
\draw[very thin, color=blue!55!white, decorate, decoration={snake, amplitude=0.2mm, segment length=1.0mm}] (270:2) .. controls +(3,0) and +(0,-1) .. (45:3.5); 
\fill (180:1.8) circle (0pt) node[left] {{\small$X$}};
\draw (0,0) circle (2);
\draw[<-, very thick] (0.100,2) -- (-0.101,2) node[above] {}; 
\draw[<-, very thick] (-0.100,-2) -- (0.101,-2) node[below] {}; 
\end{tikzpicture} 
\, . 
\ee
We show in Section \ref{sec:defectaction} that the value of this diagram is the endomorphism of $\Delta_V$ given by multiplication with the polynomial $(-1)^{{m+1\choose 2}} \str( \partial_{z_1} d_X \ldots \partial_{z_m} d_X )$. While there is an obvious dependence on the variable ordering, we prove that the homotopy equivalence class of this endomorphism changes only by a permutation sign if we change the ordering of the variables~$z_i$ in the outer region.

Let $(R = k[x_1,\ldots,x_n], W)$ be an object of $\LG$. Given $\sigma \in S_n$ consider the matrix factorisation $(\Delta^\sigma_W, d_{\Delta^\sigma_W})$ of $\widetilde W$ over $\Re$ with the same underlying graded module as $\Delta_W$, but the differential that we would have written down if we had begun with the variable ordering $x_{\sigma(1)},\ldots,x_{\sigma(n)}$, namely
\be
d_{\Delta^\sigma_W} = \delta^\sigma_+ + \delta_- \, , \qquad \delta_+ = \sum_{i=1}^n \partial^{\,\sigma}_{[i]}W\cdot \theta_i\, , \qquad \delta_- =  \sum_{i=1}^n (x_i-x'_i) \cdot \theta_i^*
\ee
where $\partial^{\,\sigma}_{[i]}$ are the modified difference quotient operators, defined for $1 \le i \le n$ by
\be\label{diffquotopmod}
\partial^{\,\sigma}_{[i]}: k[x,x'] \lra k[x,x'] \, , \qquad f \lmt \frac{{}^{t_{\sigma(1)}\ldots t_{\sigma(i-1)}}f - {}^{t_{\sigma(1)}\ldots t_{\sigma(i)}}f}{x_{\sigma(i)}-x'_{\sigma(i)}} \, . 
\ee
Together with the morphisms $\rho, \lambda$ of \eqref{lambdarho} this matrix factorisation serves as an alternative unit $1$-endomorphism of $(R,W)$. That is, $\rho$ and $\lambda$ are both isomorphisms and the coherence axioms for the unit in a bicategory holds for $\Delta^\sigma_W$. By uniqueness of units there is an isomorphism in $\hmf(\Re, \widetilde W)$
\[
\xi: \Delta_W \lto \Delta^\sigma_W
\]
unique with the property that, for any $1$-morphism $X: (k[x],W) \lto (k[z],V)$, the diagram
\be\label{eq:propertydeltasigma}
\xymatrix{
X \otimes_R \Delta_W \ar[dr]_{\rho}\ar[rr]^-{1 \otimes \xi} & & X \otimes_R \Delta^\sigma_W \ar[dl]^{\rho}\\
& X
}
\ee
commutes in $\HMF(k[x,z], V - W)$. Explicitly, $\xi$ can be constructed as the composite
$$
\xymatrix@C+1pc{
\Delta_W \ar[r]^-{\lambda_{\Delta^\sigma_W}^{-1}} & \Delta^\sigma_W \otimes_R \Delta_W \ar[r]^-{\rho_{\Delta_W}} & \Delta^\sigma_W
} ,
$$
but we will only use commutativity of \eqref{eq:propertydeltasigma}. Taking $X = \Delta_W^{\vee}$ and tensoring with $R$ over $\Re$ turns occurrences of $\Delta_W^{\vee} \otimes_R -$ into $\Hom_{\Re}( \Delta_W, - )$ and in particular $X$ becomes $\Hom_{\Re}( \Delta_W, R )$. From this we deduce that after applying $\Hom_{\Re}( \Delta_W, -)$ the diagram
\be
\xymatrix{
\Delta_W \ar[dr]_{\pi}\ar[rr]^-{\xi} & & \Delta^\sigma_W\ar[dl]^{\pi}\\
& R
}
\ee
commutes. But this means that the diagram itself commutes in $\HF(\Re, \widetilde W)$, i.\,e.~$\xi$ is the isomorphism connecting these two stabilisations of $R$.

To understand how permuting variables affects the values of diagrams like \eqref{eq:signdependencepermdia}, we need to understand the effect on the evaluation and coevaluation maps of Section \ref{sec:derivcoeval}. For the remainder of this section we fix potentials $W \in R = k[x_1,\ldots,x_n]$, $V \in S = k[z_1,\ldots,z_m]$ and $X \in \hmf(k[x,z], V - W)$.

By inspection $\widetilde\coev_X$ depends only on the order of the $x$-variables and $\coev_X$ depends only on the order of the $z$-variables, through the maps $\Psi, \varepsilon$. Since $\Psi$ is independent of the ordering up to homotopy, the only ``real'' dependence in both cases is via $\varepsilon$. Similarly $\widetilde \eval_X$ depends on the order of the $x$-variables via $\Lambda^{(x)}$, $\underline{\ud x}$ and the order of the partial derivatives in the denominator of the residue, and $\eval_X$ depends on the order of the $z$-variables via $\Lambda^{(z)}$, $\underline{\ud z}$ and the residue denominator.

Let us fix permutations $\sigma \in S_n$ and $\tau \in S_m$. If we take the variable orderings $x_{\sigma(1)},\ldots,x_{\sigma(n)}$ and $z_{\tau(1)}, \ldots, z_{\tau(m)}$ in Section \ref{sec:derivcoeval} the result will be morphisms
\begin{align}
\coev^{\sigma,\tau}_X &: \Delta^\tau_V \lto X \otimes_R {}^\dual X
\, , \qquad
\eval^{\sigma,\tau}_X: {}^\dual X \otimes_{S} X \lto \Delta^\sigma_W \, ,\label{eq:permutedevcoev1}\\
\widetilde\coev^{\sigma,\tau}_X&: \Delta^\sigma_W \lto X^\dual \otimes_{S} X
\, , \qquad
\widetilde\eval^{\sigma,\tau}_X: X \otimes_{R} X^\dual \lto \Delta^\tau_V\,. \label{eq:permutedevcoev2}
\end{align}
These are related to the original evaluation and coevaluation maps as follows.

\begin{lemma}\label{lemma:evsignspermute} The diagrams
\be\label{eq:evsignspermute}
\xymatrix@C+3pc@R-1pc{
& \Delta_V \ar[dd]^-{\xi}\\
X \otimes_R X^{\dual}\ar[ur]^-{\widetilde\eval_X}\ar[dr]_-{(-1)^{|\sigma|} \widetilde \eval^{\sigma, \tau}_X} \\
& \Delta^\tau_V, 
}
\qquad
\xymatrix@C+3pc@R-1pc{
& \Delta_W \ar[dd]^-{\xi}\\
{}^\dual X \otimes_S X \ar[ur]^-{\eval_X}\ar[dr]_-{(-1)^{|\tau|} \eval^{\sigma,\tau}_X}\\
& \Delta^\sigma_W
}
\ee
commute in $\HMF(S^{\operatorname{e}}, \widetilde V)$ and $\HMF(\Re, \widetilde W)$, respectively.
\end{lemma}
\begin{proof}
We give the argument for the first diagram, the second is similar. By the universal property of the stabilisation it suffices to prove that the $\Se$-linear morphisms
\[
\widetilde \eval_0\, , \; (-1)^{|\sigma|} \widetilde \eval^{\sigma, \tau}_0: X \otimes_R X^\dual \lto S
\]
are homotopic. For this it is enough to show that the $S$-linear maps
\[
\Res_{k[x]/k} \!\!\left[ \frac{ \str( \lambda_1 \ldots \lambda_n \circ (-) ) \, \underline{\ud x}}{\partial_{x_1} W \ldots \partial_{x_n} W} \right] , \qquad
(-1)^{|\sigma|}\Res_{k[x]/k} \!\!\left[  \frac{\str( \lambda_{\sigma(1)} \ldots \lambda_{\sigma(n)} \circ (-) ) \, \ud x_{\sigma(1)} \ldots \ud x_{\sigma(n)}}{\partial_{x_{\sigma(1)}} W \ldots \partial_{x_{\sigma(n)}} W} \right]
\]
are homotopic, but this follows from Appendix \ref{app:reshom}.
\end{proof}

\begin{lemma}\label{lemma:coevsignspermute} The diagrams
\be\label{eq:coevsignspermute}
\xymatrix@C+3pc@R-1pc{
\Delta_W \ar[dr]^-{\widetilde \coev_X} \ar[dd]_-{\xi}\\
& X^\dual \otimes_S X , \\
\Delta^\sigma_W \ar[ur]_-{(-1)^{|\sigma|} \widetilde \coev^{\sigma, \tau}}
}
\qquad
\xymatrix@C+3pc@R-1pc{
\Delta_V \ar[dd]_-{\xi} \ar[dr]^-{\coev^{\sigma,\tau}_X}\\
& X \otimes_R {}^\dual X\\
\Delta^\tau_V \ar[ur]_-{(-1)^{|\tau|} \coev^{\sigma,\tau}_X}
}
\ee
commute in $\HMF(\Re, \widetilde W)$ and $\HMF(\Se, \widetilde V)$, respectively.
\end{lemma}
\begin{proof}
The morphisms in \eqref{eq:permutedevcoev1} and \eqref{eq:permutedevcoev2} satisfy the Zorro moves, so commutativity of the diagrams in \eqref{eq:coevsignspermute} follows from commutativity of those in \eqref{eq:evsignspermute}.
\end{proof}

The general rule is that permuting the variables in a region by $\tau$ changes the value of a diagram by a factor of $(-1)^{M|\tau|}$ where $M$ is the number of wiggly lines (see Section~\ref{sec:wiggliesandsigns}) entering or departing an evaluation or coevaluation within that region. The simplest example is \eqref{eq:signdependencepermdia} where the value changes by the sign $(-1)^{|\tau|}$.

\newcommand{\etalchar}[1]{$^{#1}$}
\providecommand{\href}[2]{#2}

\end{document}